\documentclass[11pt]{article}
\usepackage{amssymb}
\usepackage[english]{babel}
\usepackage{amsmath,amsthm}
\usepackage{amsfonts}
\usepackage{mathrsfs}
\usepackage{color}
\usepackage{bbm}
\usepackage{enumerate}

\newtheorem{theorem}{Theorem}[section]
\newtheorem{corollary}[theorem]{Corollary}
\newtheorem{lemma}[theorem]{Lemma}

\newtheorem{assumption}[theorem]{Assumption}
\theoremstyle{definition}
\newtheorem{definition}[theorem]{Definition}
\theoremstyle{remark}
\newtheorem{remark}[theorem]{Remark}
\newtheorem{example}[theorem]{Example}
\numberwithin{equation}{section}

\def\Pro{{\mathbb{P}}}
\def\E{{\mathbb{E}}}
\def\e{{\varepsilon}}

\def\dist{{\mathnormal{dist}}}

\def\T{{\mathcal{T}}}
\def\Tr{{\textnormal{Tr}}}

\newcommand{\clf}{{\mathcal{F}}}

\newcommand{\cli}{{\mathcal{I}}}

\newcommand{\RR}{{\mathbb{R}}}
\newcommand{\NN}{{\mathbb{N}}}

\begin{document}

\title{Uniform large deviation principles for Banach space valued stochastic
differential equations}
\author{A. Budhiraja \thanks{%
Research supported in part by the National Science
Foundation (DMS-1305120) and the Army Research Office (W911NF-14-1-0331) University of North Carolina, Chapel Hill, Department of Statistics and
Operations Research},
P. Dupuis
\thanks{Research supported in part by the National Science Foundation (DMS-1317199). Division of Applied Mathematics, Brown University, Providence, USA},
 M. Salins \thanks{
Boston University, Department of Mathematics and Statistics}}
\maketitle

\begin{abstract}
We prove a large deviation principle (LDP) for a general class of Banach
space valued stochastic differential equations (SDE) that is uniform with
respect to initial conditions in bounded subsets of the Banach space. A key step in the proof is
showing that a uniform large deviation principle over compact sets is
implied by a uniform over compact sets Laplace principle. Because bounded
subsets of infinite dimensional Banach spaces are in general not relatively
compact in the norm topology, we embed the Banach space into its double dual
and utilize the weak-$\star $ compactness of closed bounded sets in the
double dual space. We prove that a modified version of our stochastic
differential equation satisfies a uniform Laplace principle over weak-$\star
$ compact sets and consequently a uniform over bounded sets large deviation
principle. We then transfer this result back to the original equation using a
contraction principle. The main
motivation for this uniform LDP is to generalize results of Freidlin and
Wentzell concerning the behavior of finite dimensional SDEs. Here we apply
the uniform LDP to study the asymptotics of exit times from bounded sets of
Banach space valued small noise SDE, including reaction diffusion equations
with multiplicative noise and $2$-dimensional stochastic Navier-Stokes
equations with multiplicative noise.\newline\ \newline

\noindent

\ \newline

\end{abstract}




\section{Introduction}

The goal of this work is to develop a general approach for establishing
large deviation principles (LDP) for separable Banach space valued small
noise stochastic differential equations (SDE) that are uniform for initial
conditions in bounded sets. The analogous large deviation results in finite
dimensions have been studied in the works of Freidlin and Wentzell \cite{F-W-book}
and play a central role in the study of asymptotics of exit times
from bounded domains and of invariant measures.

In infinite dimensions,
proving even finite time large deviation estimates using the classical methods of
Freidlin and Wentzell can be technically daunting. For some examples of
LDPs for infinite dimensional systems using these methods see \cite{b-1991,
c-1999,cr-2004,cm-1997,fjl-1982,f-1988,g-2005,kx-1996,mos-1989,s-1992} and
the citations therein.

An alternative approach,
which uses the equivalence between a large deviation principle and a Laplace
principle together with variational representations for infinite
dimensional Brownian motions, is significantly more tractable \cite{bd-2000,bdm-2008,de-2011}.
This approach (sometimes referred to as the `weak convergence approach') has been used to prove LDPs for a wide variety of infinite
dimensional systems \cite{bm-2009,bm-2012,bdm-2008,bdm-2010,cm-2010,dm-2009,l-2010,os-2011,rzz-2010}.
   Another benefit of this approach is that proving a Laplace principle that
is uniform over initial conditions in a compact set (cf. Definition \ref{def:unilappri}) is typically not much harder than proving the Laplace
principle for a fixed initial condition. One of the results in the current
work (Theorem \ref{thm:ULP-implies-LDP}) shows that a `uniform over compact
sets' Laplace principle implies a corresponding `uniform over compact sets'
LDP. 

The uniform LDP over compact sets, for finite dimensional diffusions, has been used by Freidlin and Wentzell \cite{F-W-book} to study the exit time and
exit place asymptotics for bounded subsets of $\mathbb{R}^d$.
Unfortunately, a uniform LDP over compact sets in infinite
dimensional Banach spaces is not very useful for studying escape time problems for bounded sets because such sets
are generally not relatively compact in
the norm topology.
In particular, open sets in infinite dimensional spaces are not relatively compact and every
infinite dimensional compact set has an empty interior.
Chenal and Millet \cite{cm-1997}  deal with these  issues
in studying exit time problems for a class of reaction diffusion equations in H\"{o}%
lder spaces by compactly embedding the H\"{o}lder space of
interest into a larger H\"{o}lder space, but there are limitations to this
approach due to the degeneracy of compact sets.

To study asymptotics of exit times from general bounded sets (e.g.,
a ball) in an infinite dimensional Banach space, a different approach is
needed. The starting point of this work is the observation that while closed
bounded sets of infinite dimensional Banach spaces are not compact, they
\emph{are} (relatively) compact when viewed under a different topology.
Specifically, if $E$ is a reflexive Banach space, then closed bounded balls
are compact in the weak topology by Alaoglu's Theorem (see Theorem \ref%
{thm:alaoglu}). If $E$ is not reflexive, then closed balls are not
necessarily compact in the weak topology, but in this case $E$ can be
canonically embedded into its double dual $E^{\star \star }$, which, once
more by Alaoglu's Theorem, has the property that closed bounded subsets of $%
E^{\star \star }$ (under norm topology on $E^{\star \star }$) are compact in
the weak-$\star $ topology (induced by the norm topology on $E^{\star }$).

In order to say why this fact is useful for the large deviation analysis we
now need to be a bit more specific about the form of the models that will be
considered in this work. Let $E$ be a separable Banach space and consider
the $E$-valued small noise SDE
\begin{equation}
\begin{cases}
dX_{x}^{\varepsilon }(t)=[AX_{x}^{\varepsilon }(t)+B(t,X_{x}^{\varepsilon
}(t))]dt+\sqrt{{\varepsilon }}G(t,X_{x}^{\varepsilon }(t))dw(t), \\
X_{x}^{\varepsilon }(0)=x\in E.%
\end{cases}%
 \label{eq:intro-abstract}
\end{equation}%
In (\ref{eq:intro-abstract}), $A$ is the infinitesimal generator of a
compact $C_{0}$-semigroup $S(t)$ (see Definition \ref{def:c0semigroup}).
The noise $w(t)$ is a cylindrical Wiener process on some Hilbert space $H$
(see Definition \ref{def:cylbm}) and the nonlinear operators $B$ and $G$ are
regular enough for there to be a unique mild solution (see Definition \ref%
{def:mildsoln}) of this equation (see Theorem \ref{thm:uniqmilsoln}). The
exact assumptions to be used are given in Section \ref{S:notes-assums}. In
typical examples, the Banach space $E$ may be the space $L^{p}(\mathcal{O})$
of functions on some domain $\mathcal{O}$ with finite $p$-th norm, the space
of continuous functions $C(\mathcal{O})$, a H\"{o}lder space $C^{\theta }(%
\mathcal{O})$, or a Sobolev space $W^{k,p}(\mathcal{O})$.

{\ A mild solution to
\eqref{eq:intro-abstract} can be interpreted as the solution to the integral equation
\begin{equation} \label{eq:intro-X-mild}
  X^\e_x(t) = S(t)x + \int_0^t S(t-s)B(s,X^\e_x(s))ds + \sqrt{\e}\int_0^t S(t-s)G(s,X^\e_x(s))dw(s).
\end{equation}
We will be particularly interested in the  stochastic convolution term on the right-hand side, namely,
 \begin{equation}
Y_{x}^{\varepsilon }(t) \doteq \sqrt{{\varepsilon }}\int_{0}^{t}S(t-s)G(s,X_{x}^{%
\varepsilon }(s))dw(s).  \label{eq:intro-Y-mild}
\end{equation}%
A precise definition of an $E$-valued stochastic integral of the
form that appears in \eqref{eq:intro-Y-mild} is given in Section %
\ref{S:notes-assums}.

Let $C([0,T]:E)$ be the space of continuous functions from $[0,T]$ to $E$ endowed with the supremum norm. Assume that for any $\Psi \in C([0,T]:E)$, there is a unique $v_\Psi \in C([0,T]:E)$ solving
\begin{equation}
  v_\Psi(t) = \int_0^t S(t-s)B(s, v_\Psi(s) + \Psi(s))ds.
\end{equation}
Then we define $\mathcal{M}$ to be the mapping that sends $\Psi \mapsto v_{\Psi
}+\Psi $. If we set $\Psi (t)=S(t)x+Y_{x}^{\varepsilon }(t)$, then the
solution $X_{x}^{\varepsilon }$ to \eqref{eq:intro-X-mild} is an $E$-valued
continuous process satisfying $X_{x}^{\varepsilon }=\mathcal{M}(S(\cdot
)x+Y_{x}^{\varepsilon })$ and \eqref{eq:intro-Y-mild}.

From this point of view, to describe a solution of \eqref{eq:intro-abstract}
 we need three ingredients: the semigroup $S(t)$, the mapping $\mathcal{M}$%
, and the stochastic convolution $\int_0^t S(t-s)G(s, X^\e_x(s))dw(s)$. Our
assumptions on these three objects are given as Assumptions \ref%
{assum:semigroup}, \ref{assum:mathcal-M}, and \ref{assum:G-new}.%

We can now explain the role weak-$\star $ compactness plays in our analysis.
Our main result establishes an LDP for $\{X_{x}^{{\varepsilon }}\}$ that is
uniform over $x$ in any closed bounded set $E_{0}$ in $E$ (Theorem \ref%
{thm:LDP}). For this we identify a family of rate functions $%
I_{x}:C([0,T]:E)\rightarrow \lbrack 0,\infty ]$, indexed by $x\in E$, and
their associated level sets $\Phi _{x}(s)=\{\varphi \in
C([0,T]:E):I_{x}(\varphi )\leq s\}$, such that for any bounded set $E_{0}$,
the following uniform lower and upper bounds hold.

\begin{enumerate}
\item For any $\delta>0$, $0< s_0 <\infty$,
\begin{equation}  \label{eq:intro-LDP-lower-unif}
\liminf_{{\varepsilon} \to 0}\inf_{x \in E_0} \inf_{\varphi \in \Phi(s_0)}
\left\{{\varepsilon} \log \left({\mathbb{P}} \left( |X^{\varepsilon}_x -
\varphi|_{C([0,T]:E)}<\delta \right)\right)+ I_{x}(\varphi)\right\} \geq 0.
\end{equation}

\item For any $\delta>0$, $0<s_0< \infty$,
\begin{equation}  \label{eq:intro-LDP-upper-unif}
\limsup_{{\varepsilon} \to 0}\sup_{x \in E_0} \sup_{s \leq s_0} \left\{ {%
\varepsilon} \log\left( {\mathbb{P}} \left( \mathnormal{dist}%
_{C([0,T]:E)}\left(X^{\varepsilon}_x, \Phi_{x}(s) \right) \geq \delta
\right)\right)+s\right\} \leq 0.
\end{equation}
\end{enumerate}

The main idea in the proof is to reduce the problem to establishing a
uniform Laplace principle, since, as was observed earlier, establishing a
uniform Laplace principle is not much harder than proving a pointwise
Laplace principle. A family satisfies a Laplace principle uniformly
 over $x \in E_0$
if for any bounded and continuous $h: C([0,T]:E) \to \mathbb{R}$,
\begin{equation} \label{eq:intro-Laplace}
\lim_{\e \to 0} \sup_{x \in E_0} \left|\e \log \E \exp \left(-\frac{h(X^\e_x)}{\e} \right)
+ \inf_{\varphi \in C([0,T]:E)} \{h(\varphi) + I_x(\varphi) \}\right| = 0.
\end{equation}

We prove in Theorem \ref{thm:ULP-implies-LDP} that if $%
E_0 $ is compact, then the uniform large deviations principle is implied
by the uniform Laplace principle. But since the Banach space $E$ is infinite
dimensional, the same is not true for general closed bounded sets $E_0$ and
thus Theorem \ref{thm:ULP-implies-LDP} is not directly useful. Nevertheless,
this theorem does provide a key ingredient to the proof of Theorem \ref%
{thm:LDP}. Roughly speaking, this is due to the regularizing property of the
semigroup, specifically the fact that since $\{S(t)\}$ is a compact
semigroup, for any bounded set $E_0$, the image of $E_0$ under $S(t)$ is a
compact set for every $t>0$.

From the compactness of the semigroup $\{S(t)\}$ it follows that $S(t)$ can
be uniquely extended to a compact linear operator from $E^{\star\star} \to E$
and the extension $S(t)$ is continuous from $E^{\star\star}$, equipped with
the weak-$\star$ topology, to $E$ (with the norm topology) for every $t>0$
(see Theorem \ref{thm:doubledualextend} and Lemma \ref{lem:U-compact}(2)).
The assumption that the semigroup is compact is a very natural assumption, an example being the heat equation semigroup.

The regularizing property of $\{S(t)\}_{t>0}$ just described suggests
viewing a subset $E_0 \subset E$ as equipped with the
weak-$\star $ topology (by regarding it as a subset of $E^{\star \star }$)
under which it is relatively compact.
However, the fact that $S(0)=I$, the
identity operator, does not map a ball with weak- $\star $ topology
continuously into $E$ implies that the map $x\mapsto X_{x}^{\varepsilon }$
cannot be a continuous map from a ball (with the weak-$\star $ topology) to $%
C([0,T]:E)$ (the space of continuous functions from $[0,T]$ to $E$ equipped
with the uniform metric). Thus the approach for establishing a uniform
Laplace principle for $\{X_{x}^{{\varepsilon }}\}$ over $x$ in closed
bounded sets in $E$, by considering the weak-$\star $ topology on this set,
is not directly applicable.

The key insight comes from the work of Sowers \cite{s-1992} which suggests
first proving the LDP for the stochastic convolution $Y_{x}^{\varepsilon }$
defined in \eqref{eq:intro-Y-mild}. 
It will be seen that
under Assumptions \ref{assum:semigroup}, \ref{assum:mathcal-M}, and \ref%
{assum:G-new}, that \eqref{eq:intro-Y-mild} has a unique solution (see Theorem \ref%
{thm:Y-well-defined}). In particular, Assumption \ref{assum:mathcal-M} says
that for every $\Psi \in L^{\infty }([0,T]:E)$ there is a unique $\mathcal{M}%
(\Psi )$ in $L^{\infty }([0,T]:E)$ as defined above and the map $\mathcal{M}$
is Lipschitz in an appropriate sense. From the unique solvability of %
\eqref{eq:intro-abstract} and the definition of $\mathcal{M}$ it will then
follow that $X_{x}^{\varepsilon }=\mathcal{M}(S(\cdot )x+Y_{x}^{{\varepsilon
}})$ and that $Y_{x}^{\varepsilon }$ can be characterized as the unique
continuous $\{{\mathcal{F}}_{t}\}$-progressively measurable solution of
\begin{equation}
Y_{x}^{\varepsilon }(t)=\sqrt{{\varepsilon }}\int_{0}^{t}S(t-s)G(s,\mathcal{M%
}(Y_{x}^{\varepsilon }+S(\cdot )x)(s))dw(s),\;t\in \lbrack 0,T].
\label{eq:yeps2}
\end{equation}%
In proving a uniform large deviation principle, it will be more convenient to
work with $Y_{x}^{\varepsilon }$ rather than $X_{x}^{\varepsilon }$ since
the $Y_{x}^{\varepsilon }$ processes for all $x\in E$ have the same initial
value, namely $0$, and because, under our conditions, \eqref{eq:yeps2} is in
fact well defined as a $C([0,T]:E)$-valued random variable  for all
$x\in E^{\star \star }$ (see Theorem \ref{thm:Y-well-defined}).

Using the variational representations  of \cite%
{bd-1998,bd-2000,bdm-2008}, it follows that for any $\e>0$, $x\in E^{\star\star}$,
and any bounded continuous function $h: C([0,T]:E) \to \mathbb{R}$
\begin{equation}
  -\e \log \E \exp \left(- \frac{h(Y^\e_x)}{\e} \right) =
  \inf_{u \in \mathcal{P}_2} \left\{\frac{1}{2}\int_0^T |u(s)|_H^2 ds  + h(Y^{\e,u}_x) \right\}
\end{equation}
where $H$ is a separable Hilbert space related to the infinite dimensional Wiener process $w(t)$ and $\mathcal{P}_2$ is a collection of $L^2([0,T]:H)$ controls
that are progressively measurable with respect to the natural filtration of $w(t)$. The controlled
process $Y^{\e,u}_x$ on the right-hand side solves the integral equation
\begin{align} \label{eq:intro-mild-control}
  Y^{\e,u}_x(t) = &\sqrt{\e}\int_0^t S(t-s) G(s,\mathcal{M}(S(\cdot)x + Y^{\e,u}_x)(s))dw(s) \nonumber \\
  &+ \int_0^t S(t-s) G(s,\mathcal{M}(S(\cdot)x + Y^{\e,u}_x)(s))u(s)ds.
\end{align}

We prove in Theorem \ref{thm:unif-Laplace-Y}
that, under conditions, $\{Y_{x}^{\varepsilon }\}$ satisfies a uniform
Laplace principle with respect to $x$ in bounded subsets of $E^{\star \star
} $. This is proven by showing that the mapping
$(\e,x,u) \mapsto Y^{\e,u}_x$ is continuous in an appropriate sense as $\e \to 0$. This requires
some care because $E^{\star\star}$ is not generally metrizable when endowed with the weak-$\star$
topology. In particular, it is not true that every sequence in a weak-$\star$ compact
set has a convergent subsequence in the weak-$\star$ topology. The assumption that $S(t)$ is a compact
semigroup, however, does guarantee that a subsequence of the sequence of trajectories
$S(\cdot)x_n$ will converge to $S(\cdot)x$ in $L^p([0,T]:E)$  for any
$p \in [1,\infty)$ and for some $x \in E^{\star\star}$ (See Lemma \ref{lem:U-compact}(3)).

Using Theorem \ref{thm:ULP-implies-LDP} that was referred to earlier, it
then follows that $Y_{x}^{\varepsilon }$ also satisfies a corresponding
uniform large deviations principle. For both of these results, compactness
of closed bounded subsets of $E^{\star \star }$ in the weak-$\star $
topology is critical. The final step is to translate this uniform large
deviations principle, valid for $x$ in closed bounded sets of $E$, from $%
\{Y_{x}^{\varepsilon }\}$ to $\{X_{x}^{\varepsilon }\}$. However this
follows from an elementary application of the contraction principle, the
uniform continuity of the map $\mathcal{M}$ (Assumption \ref{assum:mathcal-M}%
) and the relation $X_{x}^{\varepsilon }=\mathcal{M}(Y_{x}^{\varepsilon
}+S(\cdot )x)$ (see the proof of Theorem \ref{thm:LDP} in Section \ref%
{sec:secthmldp}).

We remark that this procedure is significantly simpler in the case
 where $G(t,x) \equiv Q$ does not depend on $t$ or $x$,
called the additive noise case.
In the
additive noise case, \eqref{eq:yeps2} simplifies to $Y^\e_x(t) = Y^{\varepsilon }(t)=\sqrt{{\varepsilon }}%
\int_{0}^{t}S(t-s)Qdw(s)$, which is a stochastic convolution
that is Gaussian and independent of the initial condition $x$.
Because $Y^{\varepsilon }$ are Gaussian random
variables, the associated rate function is classical (see \cite[Chapter 12.2]%
{DaP-Z}). Then one checks that, under conditions, the mapping $\mathcal{M}$
introduced above is uniformly continuous and finally appeals to (a uniform
version of the) contraction principle and the relation $X_{x}^{\varepsilon }=%
\mathcal{M}(S(\cdot )x+Y^{\varepsilon })$ to obtain an LDP that is uniform
over all $x$ in bounded subsets of $E$ (see, for example \cite[Theorem 12.17]{DaP-Z}, \cite[Theorem 5.3]{bcf-2015}).
Of course, proving an LDP for general $Y_{x}^{{\varepsilon }}$ is less straightforward because there
is feedback in the stochastic term. 

Several authors have proven large deviations principles that are uniform
with respect to initial conditions in bounded subsets of a Banach space (e.g., \cite{cr-2004,s-1992}) even in the case of multiplicative noise. However
the proofs in these papers are based on more traditional methods and tend to be quite technical. The main goal of this paper is to establish a general
method, based on the weak convergence approach, that can be used to
establish a uniform large deviations principle with respect to initial
conditions in bounded sets for a broad class of models. As is shown in
Section \ref{S:examples}, the uniform large deviation results of \cite%
{cr-2004,s-1992} follow as a consequence of Theorem \ref{thm:LDP} of the
current work. This section also illustrates the applicability of the theorem
through other examples. In particular, we consider a two dimensional
stochastic Navier-Stokes equation with multiplicative noise for which a
uniform large deviation principle has not been previously studied.

The paper is organized as follows. In Section \ref{S:notes-assums} we give
some background definitions and formulate our main assumptions. Section \ref%
{S:Laplace} presents the main results, in particular the uniform large
deviation principle for $X_{x}^{\varepsilon }$ is given in Theorem \ref%
{thm:LDP}. In Section \ref{sec:secthmldp} we present the proof of Theorem %
\ref{thm:LDP}. Sections \ref{sec:uniqmildsolns}-\ref{sec:fin} then give the
proofs of the key results used in Section  \ref{sec:secthmldp}. In Section \ref%
{sec:uniqmildsolns} we prove the well-posedness of $X_{x}^{\varepsilon }$, $%
Y_{x}^{\varepsilon }$, of the associated controlled stochastic processes $%
X_{x}^{{\varepsilon },u}$ and $Y_{x}^{{\varepsilon },u}$, and of their
vanishing noise limits. In Section \ref{sec:pfuniflap} we prove Theorem \ref%
{thm:compact-level-sets-tilde}, which states that the processes $Y_{x}^{{%
\varepsilon }}$ satisfy a uniform Laplace principle over $x$ in bounded
subsets of $E^{\star \star }$. A central result in our work is Theorem \ref%
{thm:ULP-implies-LDP} which says that a uniform Laplace principle implies a
uniform LDP, where the uniformity is over compact sets. This result is
proved in Section \ref{sec:fin}. In Section \ref{S:exit}, we use the uniform
large deviations principle (Theorem \ref{thm:LDP}) to prove certain exit
time and exit location results in a general setting. In Section \ref%
{sec:suff-conds}, we present a simpler set of sufficient
conditions under which the assumptions of Section \ref{S:notes-assums} hold.
Section \ref{S:examples} gives two applications of the theory, the first to
a stochastic reaction diffusion with multiplicative noise and the second to
a two-dimensional Navier-Stokes equation with multiplicative noise. Finally,
in Appendix \ref%
{S:doubledual} we recall some basic facts about dual spaces, the weak-$\star
$ topology, and compact semigroups.

The following notation will be used. The norm on a Banach space $E$ will be
denoted by $|\cdot |_{E}$. Similarly, the inner product on a Hilbert space $%
H $ will be denoted by $\langle \cdot ,\cdot \rangle _{H}$. We will
frequently use various spaces of functions from subsets of $\mathbb{R}$ to a
Banach space $E$. For any $p\in \lbrack 1,\infty )$ and $T\in \lbrack
0,\infty )$, define the Banach spaces
\begin{equation*}
L^{p}([0,T]:E)\doteq \left\{ f:[0,T]\rightarrow
E:\int_{0}^{T}|f(s)|_{E}^{p}ds<\infty \right\}
\end{equation*}%
with associated norm
\begin{equation*}
|f|_{L^{p}([0,T]:E)}\doteq \left( \int_{0}^{T}|f(s)|_{E}^{p}ds\right) ^{1/p}.
\end{equation*}%
The space $L^{\infty }([0,T]:E)$ is defined in a similar manner. We also
define $C([0,T]:E)$ to be the Banach space of continuous functions from $%
[0,T]\rightarrow E$ endowed with the norm
\begin{equation*}
|f|_{C([0,T]:E)}=|f|_{L^{\infty }([0,T]:E)}\doteq \sup_{t\in \lbrack
0,T]}|f(t)|_{E}.
\end{equation*}%
For Banach spaces $E_{1}$ and $E_{2}$ and a bounded linear operator $%
A:E_{1}\rightarrow E_{2}$, we define the operator norm by
\begin{equation*}
|A|_{\mathscr{L}(E_{1},E_{2})}\doteq \sup_{|x|_{E_{1}}\leq 1}|Ax|_{E_{2}}.
\end{equation*}%
If $E_{1}=E_{2}$ we use the notation $|A|_{\mathscr{L}(E_{1})}=|A|_{%
\mathscr{L}(E_{1},E_{1})}$. We also will make use of the distance function.
 Given a probability space $(\Omega ,%
\mathcal{F},{\mathbb{P}})$, we define the Banach spaces of Banach space
valued random variables,
\begin{equation*}
L^{p}(\Omega :E)\doteq \left\{ X:\Omega \rightarrow E:{\mathbb{E}}\left\vert
X\right\vert _{E}^{p}<\infty \right\} .
\end{equation*}%
For $N\in {\mathbb{N}}$ and a separable Hilbert space $H$, let
\begin{equation}
\mathcal{S}^{N}\doteq \left\{ u\in L^{2}([0,T]:H):\int_{0}^{T}|u(s)|_{H}^{2}ds\leq N\right\}
.  \label{eq:SN-def}
\end{equation}%
By a filtered probability space $(\Omega ,{\mathcal{F}},{\mathbb{P}},\{{%
\mathcal{F}}_{t}\}_{0\leq t\leq T})$ we mean a complete probability space $%
(\Omega ,{\mathcal{F}},{\mathbb{P}})$ equipped with a filtration $\{{%
\mathcal{F}}_{t}\}$ satisfying the usual conditions. We will denote by $%
\mathcal{P}_{2}$ the collection of $H$-valued $\{\mathcal{F}_{t}\}$%
-progressively measurable processes such that
\begin{equation*}
{\mathbb{P}}\left( \int_{0}^{T}|u(s)|_{H}^{2}ds<\infty \right) =1
\end{equation*}%
and for any $N\in {\mathbb{N}}$, $\mathcal{P}_{2}^{N}$ is the set
\begin{equation}
\mathcal{P}_{2}^{N}\doteq \left\{ u\in \mathcal{P}_{2}:%
\int_{0}^{T}|u(s)|_{H}^{2}ds\leq N\text{ a.s.}\right\} .  \label{eq:PN2}
\end{equation}%
For any Banach space $E$, we define the dual space $E^{\star }$ to be the
set of bounded linear functionals from $E$ to $\mathbb{R}$. For any $%
x^{\star }\in E^{\star }$ and $x\in E$ we denote the duality by $\langle
x,x^{\star }\rangle _{E,E^{\star }}=x^{\star }(x)$. $E^{\star }$ is a Banach
space endowed with the norm $|x^{\star }|_{E^{\star }}=\sup_{|x|_{E}\leq
1}\left\langle x,x^{\star }\right\rangle _{E,E^{\star }}$. We will make
extensive use of the weak-$\star $ topology on dual spaces. A net $%
\{x_{i}^{\star }\}_{i\in {\mathcal{I}}}\subset E^{\star }$ for some directed
set ${\mathcal{I}}$, converges in the weak-$\star $ topology to $x^{\star }$
if $\lim \left\langle x,x_{i}^{\star }\right\rangle _{E,E^{\star
}}=\left\langle x,x^{\star }\right\rangle _{E,E^{\star }}$ for every $x\in E$%
. Note that we make use of nets instead of sequences because the weak-$\star
$ topology on $E^{\star }$ is generally not metrizable. For a Polish space $\mathcal{E}$ with metric $\rho$,  $K\subset \mathcal{E}$ and  $x\in \mathcal{E}$, define ${\mathnormal{dist}}_{\mathcal{E}}(x,K)=\inf_{y\in K}\rho(x,y)$. We will suppress $\mathcal{E}$ from notation when clear from the context. We will usually
denote by $\kappa ,\kappa _{1},\kappa _{2},\dotsc $, the constants that
appear in various estimates within a proof. The value of these constants may
change from one proof to another.

\section{Assumptions}

\label{S:notes-assums} In this section we introduce the assumptions that
will be needed for our main large deviations result in Section \ref%
{S:Laplace} (Theorem \ref{thm:LDP}). We begin by introducing the notion of a
cylindrical Brownian motion. Let $H$ be a separable Hilbert space and fix a
filtered probability space $(\Omega, {\mathcal{F}}, {\mathbb{P}}, \{{%
\mathcal{F}}_t\}_{0\le t\le T})$.

\begin{definition}
\label{def:cylbm} A collection of continuous real valued stochastic
processes $w \doteq \{\{w_h(t)\}_{0\le t \le T}: h \in H\}$ on the filtered
probability space $(\Omega, {\mathcal{F}}, {\mathbb{P}}, \{{\mathcal{F}}%
_t\}_{0\le t\le T})$ is said to be a $H$-cylindrical Brownian motion (or $H$%
-cBm) if for every $h\in H$, $w_h$ is a ${\mathcal{F}}_t$-Brownian motion
with variance parameter $\|h\|^2_H$ and for any $t\in [0,T]$, $h_1, h_2 \in
H $ and $\alpha_1, \alpha_2 \in {\mathbb{R}}$, $w_{\alpha_1h_1 + \alpha_2
h_2}(t) = \alpha_1 w_{h_1}(t) + \alpha_2 w_{h_2}(t)$ a.s.
\end{definition}

With an abuse of notation, for $h\in H$ we write $\left\langle
w(t),h\right\rangle _{H}\doteq w_{h}(t)$. If $w$ is a $H$-cBm as in
Definition \ref{def:cylbm} then for any complete orthonormal system (CONS) $%
\{e_{k}\}_{k\in {\mathbb{N}}}$ in $H$ the collection $\beta \doteq \{\beta
_{i}\doteq w_{e_{i}}\}_{i=1}^{\infty }$ is a collection of iid standard ${%
\mathcal{F}}_{t}$-Brownian motions. The collection $\beta $ can be viewed as
a random variable with values in the Polish space $C([0,T]:{\mathbb{R}}%
^{\infty })$, where ${\mathbb{R}}^{\infty }$ is the usual sequence space of
all maps from ${\mathbb{N}}\rightarrow {\mathbb{R}}$ which is equipped with
any metric that is consistent with componentwise convergence.

Given an $\{{\mathcal{F}}_{t}\}$-progressively measurable $H$-valued process
$\{u(t)\}_{0\leq t\leq T}$ the stochastic integral $\int_{0}^{T}\left\langle
u(s),dw(s)\right\rangle _{H}$ is defined to be the $L^{2}(\Omega ,{\mathbb{P}%
})$ limit
\begin{equation}
\int_{0}^{T}\left\langle u(s),dw(s)\right\rangle _{H}\doteq
\sum_{k=1}^{\infty }\int_{0}^{T}\left\langle u(s),e_{k}\right\rangle
_{H}d\beta _{k}(s).  \label{eq:stoch-int-def}
\end{equation}%
We occasionally also write this stochastic integral as $\int_{0}^{T}\left%
\langle dw(s),u(s)\right\rangle _{H}$. The infinite sum converges in $%
L^{2}(\Omega )$ as long as ${\mathbb{E}}\int_{0}^{T}|u(s)|_{H}^{2}ds<\infty $
and we have the following isometry:
\begin{equation}
{\mathbb{E}}\left( \int_{0}^{T}\left\langle u(s),dw(s)\right\rangle
_{H}\right) ^{2}=\sum_{k=1}^{\infty }{\mathbb{E}}\int_{0}^{T}\left\langle
u(s),e_{k}\right\rangle _{H}^{2}={\mathbb{E}}\int_{0}^{T}|u(s)|_{H}^{2}ds.
\label{eq:stoch-int-L2-norm}
\end{equation}%
By a localization argument the definition of the stochastic integral is
extended to all progressively measurable $H$-valued processes $u$ that
satisfy $\int_{0}^{T}|u(s)|_{H}^{2}ds<\infty $ a.s.

In this work we will need to consider Banach space valued stochastic
integrals. We begin with an elementary lemma about random variables in
separable Banach spaces. The proof is omitted.

\begin{lemma}
\label{lem:separable-dual-rv} Assume $E$ is a separable Banach space. Let $%
Z_{1}$ and $Z_{2}$ be $E$-valued random variables with the property that ${%
\mathbb{P}}(\left\langle Z_{1},x^{\star }\right\rangle _{E,E^{\star
}}=\left\langle Z_{2},x^{\star }\right\rangle _{E,E\star })=1$ for all $%
x^{\star }\in E^{\star }$. Then ${\mathbb{P}}(Z_{1}=Z_{2})=1$.
\end{lemma}


For the rest of this work $E$ will be a separable Banach space unless
specified otherwise. The following notion of a $E$-valued stochastic
integral was introduced in \cite{vN-2005}.

\begin{definition}
\label{def:stoch-int} Suppose that $w(t)$ is a cylindrical Brownian motion
on $H$ as in Definition \ref{def:cylbm}. Let $\Phi (s)$ be an $\mathcal{F}%
_{t}$-progressively measurable $\mathscr{L}(H,E)$-valued process. Suppose
that $\int_{0}^{T}|\Phi ^{\star }(s)x^{\star }|_{H}^{2}ds<\infty $ a.s. for
all $x^{\star }\in E^{\star }$. For $t\in \lbrack 0,T]$, the stochastic
integral $Z=\int_{0}^{t}\Phi (s)dw(s)$ is defined to be the $E$-valued
random variable such that for any $x^{\star }\in E^{\star }$,
\begin{equation}
\left\langle Z,x^{\star }\right\rangle _{E,E^{\star }}=\left\langle
\int_{0}^{t}\Phi (s)dw(s),x^{\star }\right\rangle _{E,E^{\star
}}=\int_{0}^{t}\left\langle dw(s),\Phi ^{\star }(s)x^{\star }\right\rangle
_{H}.  \label{eq:defnbspsi}
\end{equation}
\end{definition}

The right-hand side of (\ref{eq:defnbspsi}) is well-defined for every $%
x^{\star }$ as discussed previously. According to Lemma \ref%
{lem:separable-dual-rv}, if such a $Z\in E$ exists, then it is unique.

Along with Banach space valued stochastic integrals we will also need to
consider Banach space valued Lebesgue integrals. These integrals will be
understood in the weak or Pettis sense, namely for a measurable $f:[0,T] \to
E$ and $t \in [0,T]$
\begin{equation*}
Z = \int_0^t f(s)ds
\end{equation*}
will denote the unique element of $E$ (provided it exists) such that for all
$x^\star \in E^\star$
\begin{equation*}
\left<Z,x^\star\right>_{E,E^\star} = \int_0^t
\left<f(s),x^\star\right>_{E,E^\star}ds.
\end{equation*}
%

Our first assumption is on the semigroup $\{S(t)\}$. See Definition \ref%
{def:c0semigroup} for the definitions of various terms in the assumption.

\begin{assumption}
\label{assum:semigroup} The unbounded linear operator $A: D(A)\subset E \to
E $ from \eqref{eq:intro-abstract} is the infinitesimal generator of a
compact $C_0$-semigroup $\{S(t)\}_{t\ge 0}$.
\end{assumption}

Next we introduce our assumption on the function $B$ that appears in %
\eqref{eq:intro-abstract}. Let $B:[0,T]\times D(B) \to E$ be a measurable
map where $D(B)$ is a measurable subset of $E$. Consider for $\Psi \in
L^{\infty}([0,T]:E)$, the equation
\begin{equation}  \label{eq:v-abs}
dv(t) = [Av(t) + B(t,v(t) + \Psi(t))]dt, \ \ \ \ v(0) = 0,
\end{equation}

\begin{assumption}
\label{assum:mathcal-M} For every $\Psi \in L^{\infty }([0,T]:E)$, there is
a unique mild solution $v\in L^{\infty }([0,T]:E)$ of \eqref{eq:v-abs},
namely,
\begin{equation} \label{eq:v-mild}
v(t)=\int_{0}^{t}S(t-s)[B(s,v(s)+\Psi (s))]ds,\;t\in \lbrack 0,T].
\end{equation}%
Define the map $\mathcal{M}$ from $L^{\infty }([0,T]:E)$ to itself by $%
\mathcal{M}(\Psi )\doteq v+\Psi $ if $v$ solves \eqref{eq:v-mild} for $\Psi $%
. The map $\mathcal{M}$ has the following properties.

\begin{enumerate}[(a)]

\item If $\Psi \in C([0,T]:E)$ then $\mathcal{M}(\Psi) \in C([0,T]:E)$.

\item There exists a nondecreasing function $\gamma:[0,\infty)\to[0,\infty)$
such that $\liminf_{s\to \infty} \gamma(s)/s >0$ and
for any $t \in [0,T]$,
\begin{equation}  \label{eq:M-growth}
|\mathcal{M}(\Psi)|_{L^\infty([0,t]:E)} \leq
\gamma(|\Psi|_{L^\infty([0,t]:E)}).
\end{equation}

\item For any $R\in (0,\infty )$, there exists a $C=C(R,T)\in (1,\infty )$
such that whenever $\Phi ,\Psi \in L^{\infty }([0,T]:E)$ satisfy $|\Phi
|_{L^{\infty }([0,T]:E)}\leq R$ and $|\Psi |_{L^{\infty }([0,T]:E)}\leq R$,
for any $t\in \lbrack 0,T]$
\begin{equation}
|\mathcal{M}(\Phi )-\mathcal{M}(\Psi )|_{L^{\infty }([0,t]:E)}\leq C|\Phi
-\Psi |_{L^{\infty }([0,t]:E)}.  \label{eq:M-Lipschitz}
\end{equation}

\item For any $R\in (0,\infty )$ and $2\leq p<\infty $, whenever $\Phi
_{n},\Phi \in L^{\infty }([0,T]:E)$ satisfy $|\Phi _{n}|_{L^{\infty
}([0,T]:E)}\leq R$ for all $n\in \mathbb{N}$, $|\Phi |_{L^{\infty
}([0,T]:E)}\leq R$, and $\lim_{n\rightarrow \infty }|\Phi _{n}-\Phi
|_{L^{p}([0,T]:E)}=0$,
\begin{equation}
\lim_{n\rightarrow \infty }|\mathcal{M}(\Phi _{n})-\mathcal{M}(\Phi
)|_{L^{p}([0,T]:E)}=0.  \label{eq:M-continuous}
\end{equation}
\end{enumerate}
\end{assumption}

\begin{remark}
	\begin{enumerate}[(a)] $\;$
		\item The superlinearity property of $\gamma$ in Assumption \ref{assum:mathcal-M}(b)
		 guarantees
		that the functional inverse of $\gamma$ satisfies for some $C>0$ and all $\xi \in [0,\infty)$,
		\begin{equation} \label{eq:gamma-inv}
		  \gamma^{-1}(\xi) \leq C(1 + |\xi|).
		\end{equation}
The function $\gamma^{-1}$ will  appear in Assumption \ref{assum:G-new} which specifies our main condition on the diffusion coefficient $G$.
\item Assumption \ref{assum:mathcal-M}(d) is stated in terms of continuity in $%
L^p([0,T]:E)$ rather than in $C([0,T]:E)$ because later in this work we will
focus on functions like $\Psi(t) = S(t) x^{\star\star}$ where $%
x^{\star\star} \in E^{\star\star}$. As long as the semigroup satisfies
Assumption \ref{assum:semigroup}, $\Psi \in L^p([0,T]:E)$ for any $p\in
[1,\infty]$, but $\Psi$ may not be in $C([0,T]:E)$ (See Lemma \ref%
{lem:U-compact}(1)). Assumption \ref{assum:mathcal-M}(a) gives in particular
that if $\Psi$ is a $E$-valued continuous progressively measurable process on some filtered
probability space then so is $\mathcal{M}(\Psi)$.
\end{enumerate}
\end{remark}

Assumption \ref{assum:mathcal-M} is satisfied in many general situations.
The simplest such example is when $B$ is a Lipschitz map on $E$.

\begin{example}
Suppose $B:[0,T]\times E\rightarrow E$ is a measurable map that satisfies
\begin{equation*}
\sup_{0\leq s\leq T}|B(s,x)-B(s,y)|_{E}\leq L_{B}|x-y|_{E},\quad \sup_{0\leq
s\leq T}|B(s,x)|_{E}\leq C_{B}(1+|x|_{E})\mbox{ for all }x,y\in E
\end{equation*}
for $L_{B},C_{B}<\infty $, and that Assumption \ref{assum:semigroup} holds.
Then Assumption \ref{assum:mathcal-M} is satisfied. Indeed, the unique
solvability of \eqref{eq:v-mild} and part (a) in Assumption \ref%
{assum:mathcal-M} follow by standard estimates. Part (b) follows from
Gronwall's lemma using the linear growth of $B$ and the boundedness of the semigroup.
Finally parts (c) and (d) are also easily seen using the Lipschitz property
of $B$ and Gronwall's lemma.
\end{example}

Global Lipschitz continuity of $B$ is not required. We present two examples
in Section \ref{S:examples} where $B$ is not globally Lipschitz continuous.%
%
%
%
%
%
%
%
%
%
%
%
%
%
%
%
%
%
%
%

For any $\mathcal{F}_{t}$-progressively measurable $\varphi \in L^{\infty
}([0,T]:E)$ a.s., consider
\begin{equation}
Z(\varphi )(t)\doteq \int_{0}^{t}S(t-s)G(s,\varphi (s))dw(s),\;t\in \lbrack
0,T]  \label{eq:Z-varphi-def}
\end{equation}%
and for any $\varphi \in L^{\infty }([0,T]:E)$ and $u\in L^{2}([0,T]:H)$,
\begin{equation}
(\mathcal{L}(\varphi )u)(t)\doteq \int_{0}^{t}S(t-s)G(s,\varphi
(s))u(s)ds,\;t\in \lbrack 0,T].  \label{eq:mathcal-L-def}
\end{equation}%
We now introduce our assumption on $G$ that in particular ensures that these
integrals are well-defined.

\begin{assumption}
\label{assum:G-new} There exists some Banach space $E_2 \supset E$ such that
the embedding is continuous and for any $s \in [0,T]$ and $x \in E$, $G(s,x)
\in \mathscr{L}(H,E_2)$. For every $t>0$, $S(t)$ can be extended as a
bounded linear operator from $E_2$ to $E$. There exists $p>2$ and $C \in
(0,\infty)$ such that the following conditions hold.

\begin{enumerate}[(a)]

\item For any $\mathcal{F}_{t}$-progressively measurable $\varphi \in
L^{\infty }([0,T]:E)$ satisfying ${\mathbb{E}}(\gamma ^{-1}(|\varphi
|_{L^{\infty }([0,T]:E)}))^{p}<\infty $, where $\gamma ^{-1}$ is the
functional inverse of $\gamma $ from \eqref{eq:gamma-inv}, $Z(\varphi )$ is a
well-defined $L^{p}(\Omega :C([0,T]:E))$-valued random variable.
Furthermore, for any $t\in \lbrack 0,T]$,
\begin{equation}
{\mathbb{E}}\left\vert Z(\varphi )\right\vert _{C([0,t]:E)}^{p}\leq C{%
\mathbb{E}}\int_{0}^{t}\left( \gamma ^{-1}(|\varphi (s)|_{E})\right) ^{p}ds.
\end{equation}

\item For any $\mathcal{F}_t$-progressively measurable $\varphi, \psi \in
L^p(\Omega:L^\infty([0,T]:E))$, and $t \in [0,T]$,
\begin{equation}
{\mathbb{E}} \left|Z(\varphi) - Z(\psi) \right|_{C([0,t]:E)}^p \leq C {%
\mathbb{E}} \int_0^t |\varphi(s) - \psi(s)|_E^pds.
\end{equation}

\item For any $R\in \lbrack 0,\infty )$, the collection
\begin{equation*}
\left\{ Z(\varphi ):\varphi \text{ is }\mathcal{F}_{t}\text{-progressively
measurable and }{\mathbb{E}}(\gamma ^{-1}(|\varphi |_{L^{\infty
}([0,T]:E)}))^{p}\leq R\right\}
\end{equation*}%
is tight in $C([0,T]:E)$.

\item For any $\varphi \in L^\infty([0,T]:E)$ and $u \in L^2([0,T]:H)$, $%
\mathcal{L}(\varphi)u$ is well-defined and $C([0,T]:E)$-valued. Furthermore,
for any $t \in [0,T]$,
\begin{equation}
\left|\mathcal{L}(\varphi)u \right|^p_{C([0,t]:E)} \leq
C|u|^p_{L^2([0,t]:H)}\left( \int_0^t \left(\gamma^{-1}(|\varphi(s)|_E)
\right)^p ds\right).
\end{equation}

\item For any $\varphi, \psi \in L^\infty([0,T]:E)$ and $u \in L^2([0,T]:H)$%
,
\begin{equation}
\left| \mathcal{L}(\varphi)u - \mathcal{L}(\psi)u\right|^p_{C([0,t]:E)} \leq
C|u|^p_{L^2([0,T]:H)} \left( \int_0^t |\varphi(s) - \psi(s)|_E^pds\right).
\end{equation}

\item For any $R,N\in \lbrack 0,\infty )$, the collection
\begin{equation*}
\left\{ \mathcal{L}(\varphi )u:|\varphi |_{L^{\infty }([0,T]:E)}\leq R,u\in
\mathcal{S}^{N}\right\}
\end{equation*}%
is pre-compact in $C([0,T]:E)$.
\end{enumerate}
\end{assumption}

In Section \ref{sec:suff-conds}, we provide a set of sufficient
conditions that imply Assumption \ref{assum:G-new}.

\section{Main Results}

\label{S:Laplace}

Theorem \ref{thm:uniqmilsoln} below says that under Assumptions \ref%
{assum:semigroup}, \ref{assum:mathcal-M}, and \ref{assum:G-new}, the SDE in %
\eqref{eq:intro-abstract} is well posed in the mild sense. These assumptions
will be in force throughout in Sections \ref{S:Laplace}-\ref{S:exit}. In
particular, throughout these sections $p$ will be as introduced in
Assumption \ref{assum:G-new}. By a mild solution we mean the following.

\begin{definition}[Mild Solution]
\label{def:mildsoln} Fix $x\in E$ and let $(\Omega ,{\mathcal{F}},{\mathbb{P}%
},\{{\mathcal{F}}_{t}\}_{0\leq t\leq T})$ and $w$ be as in Section \ref%
{S:notes-assums}. An $\{{\mathcal{F}}_{t}\}$-progressively measurable $E$-valued continuous
stochastic process $X_{x}^{\varepsilon }$ is said to be a mild solution to %
\eqref{eq:intro-abstract} if ${\mathbb{E}}(\gamma ^{-1}(|X_{x}^{\varepsilon
}|_{L^{\infty }([0,T]:E)}))^{p}<\infty $ and
\begin{equation}
X_{x}^{\varepsilon }=\mathcal{M}\left( S(\cdot )x+\sqrt{{\varepsilon }}%
\int_{0}^{\cdot }S(\cdot -s)G(s,X_{x}^{\varepsilon }(s))dw(s)\right) =%
\mathcal{M}\left( S(\cdot )x+\sqrt{{\varepsilon }}Z(X_{x}^{\varepsilon
})\right),   \label{eq:eq1012}
\end{equation}
where $Z(X_x^\e)$ is given by \eqref{eq:Z-varphi-def}.
\end{definition}


\begin{theorem}
\label{thm:uniqmilsoln}
There is a unique mild solution $X^{\varepsilon}_x$ of the equation in %
\eqref{eq:intro-abstract} for every $x \in E$ and ${\varepsilon}>0$.
\end{theorem}

The proof of Theorem \ref{thm:uniqmilsoln} is given in Section \ref%
{sec:uniqmildsolns}.

We now recall the notion of a uniform large deviations principle from \cite%
{F-W-book}. Let $\mathcal{E}$ be a Polish space with metric $\rho $, and let
$\mathcal{E}_{0}$ be some topological space that will be used for indexing.
In applications, frequently $\mathcal{E}_{0}$ will correspond to the space
for the initial condition of the SDE. We say a function $I:{\mathcal{E}}%
\rightarrow \lbrack 0,\infty ]$ is a rate function if it has compact level
sets, i.e., for all $M<\infty $, $\{\varphi :I(\varphi )\leq M\}$ is
compact. Recall that for $\varphi \in \mathcal{E}$ and $K\subset \mathcal{E}$,
\begin{equation*}
{\mathnormal{dist}}(\varphi ,K)\doteq \inf_{\psi \in K}\rho (\varphi ,\psi ).
\end{equation*}

\begin{definition}[Freidlin-Wentzell Uniform Large Deviations Principle]
A family of $\mathcal{E}$-valued processes $\{Z_{x}^{\varepsilon }\}_{{%
\varepsilon }>0}$ indexed by $x\in \mathcal{E}_{0}$ is said to satisfy a
large deviations principle with respect to the rate function $I_{x}:\mathcal{%
E}\rightarrow \lbrack 0,\infty ]$, $x\in \mathcal{E}_{0}$, uniformly in a
class $\mathscr{A}$ of subsets of $\mathcal{E}_{0}$, if the following hold.

\begin{enumerate}
\item LDP lower bound: For any $A_0 \in \mathscr{A}$, $\delta>0$, and $s_0>0$%
,
\begin{equation}  \label{eq:ldpunilowbd}
\liminf_{{\varepsilon} \to 0}\inf_{x \in A_0} \inf_{\varphi \in \Phi_x(s_0)}
\left\{ {\varepsilon} \log\left( {\mathbb{P}}\left(\rho(Z^{\varepsilon}_x,%
\varphi)<\delta \right)\right) + I_x(\varphi) \right\} \geq 0.
\end{equation}

\item LDP upper bound: For any $A_0 \in \mathscr{A}$, $\delta>0$, $s_0>0$,
\begin{equation}  \label{eq:ldpuniuppbd}
\limsup_{{\varepsilon} \to 0} \sup_{x \in A_0} \sup_{s \leq s_0} \left\{{%
\varepsilon} \log\left( {\mathbb{P}}\left({\mathnormal{dist}}%
(Z^{\varepsilon}_x, \Phi_x(s) )\geq \delta\right) \right) + s \right\} \leq
0,
\end{equation}
where for $s <\infty$ and $x \in {\mathcal{E}}_0$, $\Phi_x(s) \doteq
\{\varphi \in {\mathcal{E}}: I_x(\varphi) \le s\}$.
\end{enumerate}
\end{definition}

The main result of this section gives a uniform large deviations principle
for $\{X^{\varepsilon}_x\}$ given by Theorem \ref{thm:uniqmilsoln}. We begin
with a well-posedness result for certain deterministic controlled equations.

\begin{theorem}
\label{thm:uniqdetcont}
For every $x\in E$ and $u \in L^2([0,T]:H)$ there is a unique mild solution $%
X^{0,u}_x$ of the deterministic equation
\begin{equation}  \label{eq:eq1214ab1}
dX^{0,u}_x(t) = [AX^{0,u}_x(t) + B(t,X^{0,u}_x(t)) + G(t,X^{0,u}_x(t))u(t)
]dt ,\; X^{0,u}_x(0)=x.
\end{equation}
Namely $X^{0,u}_x$ is the unique element of $C([0,T]:E)$ that satisfies
\begin{equation}  \label{eq:eq1214ab}
X^{0,u}_x = \mathcal{M}\left(S(\cdot)x + \int_0^\cdot S(\cdot-s)
G(s,X^{0,u}_x(s))u(s)ds\right) = \mathcal{M}\left(S(\cdot)x + \mathcal{L}%
(X^{0,u}_x)u\right)
\end{equation}
where $\mathcal{L}(X^{0,u}_x)u$ is given by \eqref{eq:mathcal-L-def}.
\end{theorem}

Theorem \ref{thm:uniqdetcont} will be proved in Section \ref%
{sec:uniqmildsolns}.

The rate function for the large deviations principle for $%
\{X_{x}^{\varepsilon }\}$ is given as follows. For $x\in E$ and $\varphi \in
C([0,T]:E)$ let
\begin{equation}
I_{x}(\varphi )\doteq \inf \left\{ \frac{1}{2}%
\int_{0}^{T}|u(s)|_{H}^{2}ds:X_{x}^{0,u}=\varphi \right\} ,
\label{eq:rate-fct-X}
\end{equation}%
where the infimum is taken over all $u\in L^{2}([0,T]:H)$ for which $%
X_{x}^{0,u}$ given as the mild solution of \eqref{eq:eq1214ab1} equals $%
\varphi $. We set $I_{x}(\varphi )=\infty $ if there does not exist a $u\in
L^{2}([0,T]:H)$ such that $\varphi =X_{x}^{0,u}$.

The following is the main result of this work which gives a large deviation
principle for $\{X^{\varepsilon}_x\}$ in $C([0,T]:E)$, uniformly in the
class $\mathscr{A}$ of all bounded subsets of $E$.

\begin{theorem}
\label{thm:LDP}
For $x\in E$ and $\varepsilon >0$ let $X_{x}^{\varepsilon }$ be the unique
mild solution of \eqref{eq:intro-abstract} in $C([0,T]:E)$ as given by
Theorem \ref{thm:uniqmilsoln}. Then $\{X_{x}^{\varepsilon }\}_{{\varepsilon }%
>0}$ satisfies a large deviation principle in $C([0,T]:E)$ with respect to
the rate function $I_{x}$ in \eqref{eq:rate-fct-X}, uniformly in the class $%
\mathscr{A}$ of all bounded subsets of $E$. That is, for any bounded subset $%
E_{0}$ of $E$, the following uniform lower and upper bounds hold.

\begin{enumerate}
\item For any $\delta>0$, $0< s_0 <\infty$,
\begin{equation}  \label{eq:LDP-lower-unif}
\liminf_{{\varepsilon} \to 0}\inf_{x \in E_0} \inf_{\varphi \in \Phi(s_0)}
\left\{{\varepsilon} \log \left({\mathbb{P}} \left( |X^{\varepsilon}_x -
\varphi|_{C([0,T]:E)}<\delta \right)\right)+ I_{x}(\varphi)\right\} \geq 0.
\end{equation}

\item For any $\delta>0$, $0<s_0< \infty$,
\begin{equation}  \label{eq:LDP-upper-unif}
\limsup_{{\varepsilon} \to 0}\sup_{x \in E_0} \sup_{s \leq s_0} \left\{ {%
\varepsilon} \log\left( {\mathbb{P}} \left( \mathnormal{dist}%
_{C([0,T]:E)}\left(X^{\varepsilon}_x, \Phi_{x}(s) \right) \geq \delta
\right)\right)+s\right\} \leq 0.
\end{equation}
\end{enumerate}
\end{theorem}

In Section \ref{S:exit} we show how Theorem \ref{thm:LDP} can be used to
study asymptotic problems associated with exit of $X_{x}^{\varepsilon }$
from a bounded domain. Section \ref{S:examples} presents two examples, a
reaction-diffusion equation with multiplicative noise and a two-dimensional
Navier Stokes equation with multiplicative noise, where Theorem \ref{thm:LDP}
holds. The Navier-Stokes equation requires a modification to Assumption \ref%
{assum:mathcal-M} because the mapping $\mathcal{M}$ for the Navier-Stokes
operator is not well defined as a map from $C([0,T]:H)\rightarrow C([0,T]:H)$%
. 
We now proceed with the proof of Theorem \ref{thm:LDP}.

\section{Proof of Theorem \protect\ref{thm:LDP}}

\label{sec:secthmldp} The main ingredient in the proof of Theorem \ref%
{thm:LDP} will be a reformulation of the uniform large deviations principle in
terms of the uniform Laplace principle. We begin with the following general
definitions. Recall ${\mathcal{E}}$ is a Polish space and ${\mathcal{E}}_0$
is some topological space.

\begin{definition}
A family of rate functions $I_{x}$ on $\mathcal{E}$ and parameterized by $x \in
\mathcal{E}_0$, has compact level sets on compact sets of $\mathcal{E}_0$ if
for all compact $K \subset \mathcal{E}_0$ and for all $M \in (0,\infty)$,
\begin{equation}  \label{eq:eqlakm}
\Lambda_{K,M} \doteq \bigcup_{x \in K}\{\varphi \in \mathcal{E}:I_{x}(\varphi)
\leq M\}
\end{equation}
is a compact subset of $\mathcal{E}$.
\end{definition}

\begin{definition}[Uniform Laplace Principle]
\label{def:unilappri} Let $I_{x}:\mathcal{E}\rightarrow \lbrack 0,\infty ]$
be a family of rate functions with compact level sets on compact subsets of $%
\mathcal{E}_{0}$. The family $\{Z_{x}^{\varepsilon }\}_{{\varepsilon }>0}$
indexed by $x\in \mathcal{E}_{0}$ of ${\mathcal{E}}$-valued random variables
satisfies a Laplace principle on $\mathcal{E}$ with rate function $I_{x}$, $%
x\in {\mathcal{E}}_{0}$, uniformly on compact sets if for all compact $%
K\subset \mathcal{E}_{0}$ and all bounded continuous $h:\mathcal{E}%
\rightarrow \mathbb{R}$,
\begin{equation*}
\lim_{{\varepsilon }\rightarrow 0}\sup_{x\in K}\left\vert {\varepsilon }\log
{\mathbb{E}}\left( \exp \left( -\frac{h(Z_{x}^{\varepsilon })}{{\varepsilon }%
}\right) \right) +\inf_{\varphi \in \mathcal{E}}\{h(\varphi )+I_{x}(\varphi
)\}\right\vert =0.
\end{equation*}
\end{definition}

The following result shows that a uniform Laplace principle implies the
corresponding uniform large deviation principle. The proof is given in
Section \ref{sec:fin}.

\begin{theorem}
\label{thm:ULP-implies-LDP} Let $I_{x}:\mathcal{E}\rightarrow \lbrack
0,\infty ]$ be a family of rate functions with compact level sets on compact
subsets of $\mathcal{E}_{0}$. Suppose that the family $\{Z_{x}^{\varepsilon
}\}_{{\varepsilon }>0}$ indexed by $x\in \mathcal{E}_{0}$ of ${\mathcal{E}}$%
-valued random variables satisfies a Laplace principle on $\mathcal{E}$ with
rate function $I_{x}$, $x\in {\mathcal{E}}_{0}$, uniformly on compact sets
of ${\mathcal{E}}_{0}$. Then $\{Z_{x}^{\varepsilon }\}_{{\varepsilon }>0}$
satisfies a large deviation principle with respect to rate function $I_{x}$,
uniformly in class $\mathscr{A}$ of all compact subsets of ${\mathcal{E}}%
_{0} $.
\end{theorem}

In proving Theorem \ref{thm:LDP} we are specifically interested in the case
where $\mathcal{E}_0 = E$ and $\mathcal{E} = C([0,T]:E)$. To use a uniform
Laplace principle as in Definition \ref{def:unilappri} in order to prove
Theorem \ref{thm:LDP} we will need to consider level sets of the form in %
\eqref{eq:eqlakm} where $K$ is replaced by an arbitrary bounded set $E_0$.
However, since for infinite dimensional spaces $E$ bounded sets are
typically not relatively compact, the pre-compactness of such level sets
will generally not hold simply because for $\varphi \in \Lambda_{E_0,M}$, $%
\varphi(0)\in E_0$ is in a (in general) non-compact set. The key observation
in this work is that since the semigroup $S(t):E \to E$ is compact, the
initial condition is actually the only problem with compactness. In
particular, if ${\mathcal{E}}$ is replaced by $C([t_1,T]:E)$ for any $t_1\in
(0,T)$, the corresponding level sets will be compact (see Theorem \ref%
{thm:X-contin-weak}). To deal with the problem of initial conditions, we
introduce an associated collection of processes, $\{Y^{\varepsilon}_x\}$
that all have the same initial condition, namely $0$.

While the $X_{x}^{\varepsilon }$ processes take values in $C([0,T]:E)$ for
initial conditions $x\in E$, the theorem below shows that $%
Y_{x}^{\varepsilon }$ are well defined $C([0,T]:E)$-valued random variables
for all $x\in E^{\star \star }$, the double dual space of $E$. For $x\in E$,
$Y_{x}^{\varepsilon }$ is defined through \eqref{eq:intro-Y-mild} where $%
X_{x}^{\varepsilon }$ is given as in Theorem \ref{thm:uniqmilsoln}. We note
from \eqref{eq:M-growth} that if $\psi \in L^{\infty }([0,T]:E)$ is a
progressively measurable process satisfying ${\mathbb{E}}(|\psi |_{L^{\infty
}([0,T]:E)})^{p}<\infty $, then $\varphi =\mathcal{M}(\psi )$ satisfies ${%
\mathbb{E}}(\gamma ^{-1}(|\varphi |_{L^{\infty }([0,T]:E)}))^{p}<\infty $.

\begin{theorem}
\label{thm:Y-well-defined}
For every $x\in E^{\star \star }$ there is an a.s. unique $\{{\mathcal{F}}%
_{t}\}$-progressively measurable $E$-valued continuous process $Y_{x}^{\varepsilon }$ in $%
L^{p}(\Omega :C([0,T]:E))$ that solves
\begin{equation}
Y_{x}^{\varepsilon }(t)=\sqrt{{\varepsilon }}\int_{0}^{t}S(t-s)G(s,\mathcal{M%
}(S(\cdot )x+Y_{x}^{\varepsilon })(s))dw(s).  \label{eq:eqyext}
\end{equation}%
For $x\in E$, this is the unique mild solution of \eqref{eq:intro-Y-mild}.%
\end{theorem}

The above extension of the definition of $Y_{x}^{\varepsilon }$ to all $x\in
E^{\star \star }$ is useful because bounded subsets of $E$ have compact
closure in $E^{\star \star }$ in the weak-$\star $ topology. This weak-$%
\star $ compactness will enable us to prove a uniform (over $x$ in bounded
subsets of $E^{\star \star }$) Laplace principle for $\{Y_{x}^{\varepsilon
}\}$. We next present another unique solvability result. The proof, once
more, is given in Section \ref{sec:uniqmildsolns}.

\begin{theorem}
\label{prop:propyoux}
For any $u \in L^2([0,T]:H)$ and $x \in E^{\star\star}$ there is a unique $%
Y^{0,u}_x \in C([0,T]:E)$ that solves
\begin{equation}  \label{eq:youxt}
Y^{0,u}_x(t) = \int_0^t S(t-s)G(s,\mathcal{M}( S(\cdot)x+Y^{0,u}_x
)(s))u(s)ds, \; t \in [0,T].
\end{equation}
\end{theorem}

For $x\in E^{\star \star }$ and $\varphi \in C([0,T]:E)$ define%
\begin{equation}
\tilde{I}_{x}(\varphi )\doteq\inf \left\{ \frac{1}{2}%
\int_{0}^{T}|u(s)|_{H}^{2}ds:\varphi =Y_{x}^{0,u}\right\}
\label{eq:rate-fct-Y}
\end{equation}%
where $Y_{x}^{0,u}$ is the unique solution of \eqref{eq:youxt} and the
infimum is taken over all such $u\in L^{2}([0,T]:H)$. We set $\tilde{I}%
_{x}(\varphi )=\infty $ if there is no $u\in L^{2}([0,T]:H)$ such that $%
\varphi =Y_{x}^{0,u}$.

For any $x\in E^{\star \star }$, define the level sets of $\tilde{I}_{x}$ by
\begin{equation}
\tilde{\Phi}_{x}(M)\doteq \left\{ \varphi \in C([0,T]:E):\tilde{I}%
_{x}(\varphi )\leq M\right\} ,\;M\in (0,\infty ).  \label{eq:level-sets-Y}
\end{equation}

\begin{theorem}
\label{thm:compact-level-sets-tilde} For any $K \subset E^{\star\star}$
compact in the weak-$\star$ topology and $s \in (0,\infty)$, $\bigcup_{x \in
K} \tilde{\Phi}_x(s)$ is a compact subset of $C([0,T]:E)$.
\end{theorem}

Theorem \ref{thm:compact-level-sets-tilde} is proved in Section \ref%
{sec:secratefunprop}.

The following result gives a uniform Laplace principle for $%
\{Y_{x}^{\varepsilon }\}$. The proof will be given in Section \ref%
{sec:pfuniflap}.

\begin{theorem}
\label{thm:unif-Laplace-Y} The collection $\{\tilde I_x\}_{x \in
E^{\star\star}}$ is a family of rate functions with compact level sets on
compact subsets of $E^{\star\star}$ (endowed with the weak-$\star$
topology). The family $\{Y^{\varepsilon}_x\}$ solving \eqref{eq:eqyext}
satisfies a Laplace principle on $C([0,T]:E)$ with rate function $\tilde I_x
$, uniformly on compact (with the weak-$\star$ topology) subsets of $%
E^{\star\star}$. That is, for any weak-$\star$ compact $K\subset
E^{\star\star}$ and bounded continuous $h: C([0,T]:E) \to \mathbb{R}$,
\begin{equation}  \label{eq:unif-Laplace-Y}
\lim_{{\varepsilon} \to 0} \sup_{x \in K} \left|{\varepsilon} \log {\mathbb{E%
}}\left(\exp(-{\varepsilon}^{-1}h(Y^{\varepsilon}_x)) \right) +
\inf_{\varphi \in C([0,T]:E)}\{h(\varphi) + \tilde{I}_{x}(\varphi)\} \right|
=0.
\end{equation}
\end{theorem}

The following corollary is an immediate consequence of the fact that the
uniform Laplace principle implies a uniform large deviations principle
(Theorem \ref{thm:ULP-implies-LDP}).

\begin{corollary}
\label{cor:fw} Let $K$ be a weak-$\star$ compact subset of $E^{\star\star}$.

\begin{enumerate}
\item For any $\delta>0$, $0< s_0 <\infty$,
\begin{equation}  \label{eq:LDP-lower-Y}
\liminf_{{\varepsilon} \to 0}\inf_{x \in K} \inf_{\varphi \in \tilde{\Phi}%
_{x}(s_0)} \left( {\varepsilon} \log \left({\mathbb{P}} \left(
|Y^{\varepsilon}_x - \varphi|_{C([0,T]:E)}<\delta \right)\right)+ \tilde{I}%
_{x}(\varphi)\right) \geq 0.
\end{equation}

\item For any $\delta>0$, $0<s_0< \infty$,
\begin{equation}  \label{eq:LDP-upper-Y}
\limsup_{{\varepsilon} \to 0}\sup_{x \in K} \sup_{s \leq s_0} \left( {%
\varepsilon} \log\left( {\mathbb{P}} \left( \mathnormal{dist}%
_{C([0,T]:E)}\left(Y^{\varepsilon}_x, \tilde{\Phi}_{x}(s) \right) \geq
\delta \right)\right)+s\right) \leq 0.
\end{equation}
\end{enumerate}
\end{corollary}

We can now complete the proof of Theorem \ref{thm:LDP}.

\begin{proof}[Proof of Theorem \ref{thm:LDP}]
We begin by observing that if for $x\in E$, $Y^{\e}_x$ is the unique solution of \eqref{eq:eqyext}	then $X^{\e}_x = \mathcal{M}(S(\cdot)x + Y^{\e}_x)$ is the unique solution of \eqref{eq:eq1012}.
Similarly, if $\psi=Y^{0,u}_x$ solves \eqref{eq:youxt} for some $u\in L^2([0,T]:H)$ then $\varphi = \mathcal{M}(S(\cdot)x + \psi)=X^{0,u}_x$ solving \eqref{eq:eq1214ab} for the same $u$. Also,
if $\varphi =X^{0,u}_x$ solves (\ref{eq:eq1214ab}), then $\psi
(t)=\int_{0}^{t}S(t-s)G(s,\varphi (s))u(s)ds = Y^{0,u}_x$ solves (\ref{eq:youxt}) for
the same $u\in L^2([0,T]:H)$.

 Consider part 1 in Theorem \ref{thm:LDP}. Fix $\delta \in (0,1)$, $s_0\in (0,\infty)$ and a bounded set $E_0$ in $E$.
From Theorem \ref{thm:compact-level-sets-tilde}, $\bigcup_{x\in \tilde{E}_0} \tilde \Phi_x(s_0+1) \doteq K(E_0, s_0+1)$ is compact, where $\tilde{E}_0$ is the weak-$\star$ closure of $J_E(E_0)$ in $E^{\star\star}$, and $J_E$ is the canonical embedding of $E$ into $E^{\star\star}$ (see Appendix \ref{S:doubledual}).
Also since  $\{S(t)\}_{t\ge 0}$ is a $C_0$-semigroup, $\sup_{x\in \tilde{E}_0} \sup_{ t \in (0,T]}|S(t)x|_{E} <\infty$
(See Lemma \ref{lem:U-compact}).
Let
$$R \doteq \sup_{x\in \tilde{E}_0} \sup_{\psi \in \tilde \Phi_x(s_0+1)} (|\psi|_{C([0,T]:E)}+ |S(\cdot)x|_{L^\infty([0,T]:E)}) + 1.$$
With $C \doteq C(R,T)$ as in  Assumption \ref{assum:mathcal-M}(c), we have  that for any $\psi \in K(E_0, s_0+1)$,
and any $x \in E_0$,
$$\{|Y^\e_x - \psi|_{C([0,T]:E)} < \delta/C\} \subset \{|X^\e_x - \mathcal{M}(\psi + S(\cdot)x)|_{C([0,T]:E)} < \delta\}.$$
Fix $x \in E_0$ and $\varphi \in \Phi_x(s_0)$. Then for any $\kappa \in (0,1)$ there is a $\psi \in C([0,T]:E)$
such that
$$\tilde I_x(\psi) \le I_x(\varphi) + \kappa \le s_0+1, \mbox{ and } \varphi = \mathcal{M}(\psi + S(\cdot)x).$$
Thus
$$\e \log\Pro(|X^\e_x - \varphi|_{C([0,T]:E)} < \delta) + I_x(\varphi) \ge \e\log\Pro( |Y^\e_x - \psi|_{C([0,T]:E)} < \delta/C) +
\tilde I_x(\psi)-\kappa .$$

Noting that $x \in {E}_0$ and $\varphi \in  \Phi_x(s_0)$ were arbitrary, we have taking infimum over $x \in E_0$ and $\varphi \in \Phi_x(s_0)$
\begin{align*}
	&\inf_{x\in E_0}\inf_{\varphi \in \Phi_x(s_0)}
\left[\e\log\Pro(|X^\e_x - \varphi|_{C([0,T]:E)} < \delta) + I_x(\varphi)\right]\\
& \ge \inf_{x\in \tilde{E}_0}\inf_{\psi \in \tilde\Phi_x(s_0+1)}\left[\e \log\Pro( |Y^\e_x - \psi|_{C([0,T]:E)} < \delta/C) +
\tilde I_x(\psi)\right]-\kappa .
\end{align*}
The desired inequality in part 1 now follows on sending $\e \to 0$, using Corollary \ref{cor:fw}(1) and then sending $\kappa \to 0$.

Now we show the uniform upper bound in part 2 of the theorem. Fix $\delta \in (0,1)$, a bounded set $E_0$ in $E$, and $s_0 \in (0,\infty)$.
Let $C$ be as in the proof of part 1. Note that for any $s\le s_0$, $x\in E_0$ and
$\psi \in \tilde\Phi_x(s)$, if
$|Y^\e_x - \psi|_{C([0,T]:E)} < \delta/C$ then $\varphi \doteq \mathcal{M}(\psi + S(\cdot)x)$ satisfies
$\varphi \in \Phi_x(s)$ and by Assumption \ref{assum:mathcal-M}(c)
$|X^\e_x - \varphi|_{C([0,T]:E)} < \delta$. Thus
$\dist(Y^\e_x, \tilde{\Phi}_x(s)) < \delta/C$ implies
$\dist(X^\e_x, {\Phi}_x(s)) < \delta$. Thus
\begin{align*}
	&\sup_{x\in E_0}\sup_{s\le s_0}
	\left[\e \log\Pro(\dist(X^\e_x, \Phi_x(s)) \geq \delta)+s \right]\\
	&\le \sup_{x\in \tilde{E}_0}\sup_{s\le s_0}\left[\e \log\Pro(\dist(Y^\e_x, \tilde\Phi_x(s)) \geq \delta/C)+s \right].
\end{align*}
The result now follows on taking limit as $\e \to 0$ and applying Corollary \ref{cor:fw}(2).

 \end{proof}
{\ The rest of the paper is organized as follows. In Section \ref%
{sec:uniqmildsolns} we present the main wellposedness results that were used
in Sections \ref{S:Laplace} and \ref{sec:secthmldp}, namely Theorems \ref%
{thm:uniqmilsoln}, \ref{thm:uniqdetcont}, \ref{thm:Y-well-defined} and \ref%
{prop:propyoux}. 
In Section \ref{sec:pfuniflap} we give the proof of Theorem \ref%
{thm:unif-Laplace-Y}. Theorem \ref{thm:ULP-implies-LDP} is proved in Section %
\ref{sec:fin}. Next, as noted earlier, Section \ref{S:exit} is devoted to
the study of asymptotics of exit from a bounded domain. Section \ref%
{sec:suff-conds} presents a sufficient condition that is useful for
verifying Assumption \ref{assum:G-new}, and Section \ref{S:examples}
presents two examples of stochastic partial differential equations for which
Theorem \ref{thm:LDP} holds. Finally, Appendix A collects some auxiliary
results. }

\section{Wellposedness}

\label{sec:uniqmildsolns}

In this section we prove Theorems \ref{thm:uniqmilsoln}, \ref%
{thm:uniqdetcont}, \ref{thm:Y-well-defined} and \ref{prop:propyoux}.
Recall the collection $\mathcal{P}^N_2$ defined in \eqref{eq:PN2}. In order
to prove Theorems \ref{thm:Y-well-defined} and \ref{prop:propyoux} we will
consider the following controlled equation for ${\varepsilon}\ge 0$, $u \in
\mathcal{P}^N_2$ and $N \in \mathbb{N}$:
\begin{align}  \label{eq:Y-control-def}
Y^{{\varepsilon},u}_x(t) = & \sqrt{{\varepsilon}} \int_0^t S(t-s)G(s,
\mathcal{M}(Y^{{\varepsilon},u}_x + S(\cdot)x)(s))dw(s)  \notag \\
& + \int_0^t S(t-s)G(s,\mathcal{M}(Y^{{\varepsilon},u}_x +
S(\cdot)x)(s))u(s)ds.
\end{align}
We will show that \eqref{eq:Y-control-def} has a unique solution
for every $x \in E^{\star\star}$, ${\varepsilon}\ge 0$, and $u \in \mathcal{P}_2^N$.
That is, there is a unique continuous $\{\mathcal{F}_t\}$-progressively measurable process $%
Y^{{\varepsilon},u}_x \in L^p(\Omega :C([0,T]:E))$ that satisfies %
\eqref{eq:Y-control-def}. This result, taking either $u=0$ or $u$ to be a
deterministic function in $L^2([0,T]:H)$ with ${\varepsilon}=0$, will prove
Theorems \ref{thm:Y-well-defined} and \ref{prop:propyoux} respectively.



We begin by noting an elementary Lipschitz property of a certain cutoff
function defined on $E$. The proof is omitted.

\begin{lemma}
\label{lem:cutoff} For any $R \geq 0$, define the map $\mathcal{T}_R:E
\to E$ by
\begin{equation}  \label{eq:TR}
\mathcal{T}_R(x) \doteq
\begin{cases}
x & \text{ if } |x|_E \leq R \\
\frac{Rx}{|x|_E} & \text{ if } |x|_E>R.%
\end{cases}%
\end{equation}
Then for any $x \in E$, $|{\mathcal{T}}_R (x)|_E \leq R$ and ${\mathcal{T}}%
_R $ satisfies the following Lipschitz property.
\begin{equation*}
|{\mathcal{T}}_R(x) - {\mathcal{T}}_R(y)| \leq 3 |x-y|_E, \; \mbox{ for any }
x, y \in E.
\end{equation*}
\end{lemma}

For $R\in (0,\infty )$ define the subset $\mathcal{E}_{R}=\{\varphi \in
C([0,T]:E):|\varphi |_{C([0,T]:E)}\leq R\}$. Notice that by Assumption \ref%
{assum:mathcal-M}(c), $\mathcal{M}$ is Lipschitz continuous on $\mathcal{E}_{R}$
for every $R$. For ${\varepsilon }\geq 0$, $R\in (0,\infty )$, $x\in
E^{\star \star }$, and $u\in \mathcal{P}_{2}^{N}$, define the mapping $%
\mathscr{K}_{x,R}^{{\varepsilon },u}:L^{p}(\Omega :\mathcal{E}%
_{R})\rightarrow L^{p}(\Omega :\mathcal{E}_{R})$ by
\begin{align}
\mathscr{K}_{x,R}^{{\varepsilon },u}(\varphi )(t)=& {\mathcal{T}}_{R}\Bigg[%
\sqrt{{\varepsilon }}\int_{0}^{t}S(t-s)G(s,\mathcal{M}(S(\cdot )x+\varphi
)(s))dw(s)  \notag  \label{eq:mathscr-K-def} \\
& \qquad +\int_{0}^{t}S(t-s)G(s,\mathcal{M}(S(\cdot )x+\varphi )(s))u(s)ds%
\Bigg].
\end{align}%
Consider the $\mathcal{E}_{R}$-valued continuous process $%
Y_{x,R}^{\varepsilon ,u}$ given by
\begin{align}
Y_{x,R}^{{\varepsilon },u}(t)& =\mathscr{K}_{x,R}^{{\varepsilon }%
,u}(Y_{x,R}^{{\varepsilon },u})(t)  \notag \\
& ={\mathcal{T}}_{R}\Bigg[\sqrt{{\varepsilon }}\int_{0}^{t}S(t-s)G(s,%
\mathcal{M}(S(\cdot )x+Y_{x,R}^{{\varepsilon },u})(s))dw(s)  \notag \\
& \quad +\int_{0}^{t}S(t-s)G(s,\mathcal{M}(S(\cdot )x+Y_{x,R}^{{\varepsilon }%
,u})(s))u(s)ds\Bigg],t\in \lbrack 0,T].  \label{eq:eqyeuxr}
\end{align}%
The following result shows that there is a unique solution to %
\eqref{eq:eqyeuxr} and gives a uniform in $R$ moment bound on the solution
process.

\begin{lemma}
\label{lem:a-priori-bounds} For every ${\varepsilon }\geq 0$, $R\in
(0,\infty )$, $N \in \mathbb{N}$, $u\in \mathcal{P}_{2}^{N}$ and $x\in E^{\star \star }$, there is a unique $\mathcal{F}_{t}$%
-progressively measurable continuous $\mathcal{E}_{R}$-valued process $Y_{x,R}^{{\varepsilon }%
,u}$ that solves \eqref{eq:eqyeuxr}. Furthermore, there is a $C\in (0,\infty
)$ such that for all $|x|_{E^{\star \star }}\leq R$, ${\varepsilon }\geq 0$, $%
u\in \mathcal{P}_{2}^{N}$, and $R\in (0,\infty )$,
\begin{equation*}
{\mathbb{E}}|Y_{x,R}^{{\varepsilon },u}|_{C([0,T]:E)}^{p}\leq C|x|_{E^{\star
\star }}^{p}e^{C({\varepsilon }^{\frac{p}{2}}+N^{\frac{p}{2}})T}.
\end{equation*}
\end{lemma}

\begin{proof}
  Fix $\e \ge 0$, $N\geq 0$, $u \in  \mathcal{P}_2^N$, $R\in (0,\infty)$ and $x \in E^{\star\star}$.
	For $T_1 \in (0,T]$ to be chosen later and any progressively measurable processes $\varphi$ and $\psi \in L^p(\Omega: \mathcal{E}_R)$, by Lemma \ref{lem:cutoff} and Assumption \ref{assum:G-new} (b) and (e),
	\begin{align*}
	  &\E \left|\mathscr{K}^{\e,u}_{x,R}(\varphi) - \mathscr{K}^{\e,u}_{x,R}(\psi) \right|_{C([0,T_1]:E)}^p \\
	  &\leq  \kappa_1 (\e^{p/2} + N^{p/2}) \E\int_0^{T_1}\left|\mathcal{M}(S(\cdot)x + \varphi)(s) - \mathcal{M}(S(\cdot)x + \psi)(s) \right|_{E}^pds\\
	  &\leq \kappa_2 (\e^{p/2} + N^{p/2}) \E\int_0^{T_1}\left|\varphi-\psi \right|_{C([0,s]:E)}^pds\\
	  &\leq \kappa_2 (\e^{p/2} + N^{p/2})T_1 \E\left|\varphi - \psi \right|_{C([0,T_1]:E)}^p.
	\end{align*}
	The second to last line follows from Assumption \ref{assum:mathcal-M}(c) and the fact that $|\varphi + S(\cdot)x|_{L^\infty([0,T_1]:E)} \leq R + \kappa_3|x|_{E^{\star\star}}$ and the same is true for $\psi + S(\cdot)x$ .
	If we choose $T_1$ small enough so that $\kappa_2T_1(\e^{p/2} + N^{p/2})<1$, then $\mathscr{K}^\e_{x,R}$ is a contraction mapping on $L^p(\Omega:\mathcal{E}_R)$ (restricted to the interval $[0,T_1]$). The unique solvability of \eqref{eq:eqyeuxr} now follows by a standard iterative procedure.
	
	We now prove the second part of the lemma. By Assumption \ref{assum:G-new}(a) and (d), there is a $\kappa_1 \in (0,\infty)$ such that
for all $\mathcal{F}_t$-progressively measurable
 $\varphi \in L^\infty([0,T]:E)$ satisfying $\E(\gamma^{-1}(|\varphi|_{L^\infty([0,T]:E)}))^p<\infty$,
$u \in \mathcal{P}_2^N$, $t \in [0,T]$, and $\e\ge 0$,
  \begin{align}
     \E\sup_{t \in [0,T]}\left|\sqrt{\e}\int_0^t S(t-s)G(s,\varphi(s))dw(s) \right|_{E}^p
    &\leq \kappa_1 \e^{\frac{p}{2}}\E \int_0^t (\gamma^{-1}(|\varphi(s)|_E))^pds
\label{eq:eq856}
  \end{align}
  and
  \begin{align} 
    \E \sup_{t \in [0,T]} \left| \int_0^t S(t-s)G(s,\varphi(s))u(s)ds \right|_E^p
     &\leq \kappa_1 N^{\frac{p}{2}}\E \int_0^t (\gamma^{-1}(|\varphi(s)|_E))^pds.\label{eq:eq904}
  \end{align}
From Lemma \ref{lem:U-compact}(1) and the previous contraction mapping argument we see that, 	for every $R\in (0,\infty)$,
	$x \in E^{\star\star}$,  $S(\cdot)x + Y^{\e,u}_{x,R} \in L^p(\Omega: L^\infty([0,T]:E))$   and
by Assumption \ref{assum:mathcal-M}(b), with probability one
  \[|\mathcal{M}(S(\cdot)x + Y^{\e,u}_{x,R})|_{L^\infty([0,t]:E)} \leq \gamma\left(|S(\cdot)x + Y^{\e,u}_{x,R}|_{L^\infty([0,t]:E)}\right).\]
Taking $\varphi = \mathcal{M}(S(\cdot)x + Y^{\e,u}_{x,R})$ in \eqref{eq:eq856} and \eqref{eq:eq904} and using the above estimate
 we see that for any $ t \in [0,T]$, $Y^{\e,u}_{x,R}$ satisfies
  \begin{align*}
    \E|Y^{\e,u}_{x,R}|_{L^\infty([0,t]:E)}^p
    \leq \kappa_1(\e^{\frac{p}{2}} + N^{\frac{p}{2}})  \E \int_0^t |S(\cdot)x + Y^{\e,u}_{x,R}|_{L^\infty([0,s]:E)}^p ds .
  \end{align*}
  From Lemma \ref{lem:U-compact}(1), for some $\kappa_2 \in (0,\infty)$,
 $$\E | S(\cdot)x + Y^{\e,u}_{x,R}|_{L^\infty([0,t]:E)}^p \leq \kappa_2(|x|_{E^{\star\star}}^p +  \E |Y^{\e,u}_{x,R}|_{L^\infty([0,t]:E)}^p)$$ for all $t\in [0,T]$ and $R\in (0,\infty)$. Combining this estimate with the previous display we conclude the second statement in the lemma by an application of  Gr\"onwall's inequality.
\end{proof}


\begin{theorem}
\label{thm:Yeux-uniqueness} For any ${\varepsilon}\ge 0$, $N\geq 0$, $u \in
\mathcal{P}_2^N$, and $x \in E^{\star\star}$, there exists a unique solution
$Y^{{\varepsilon},u}_x \in L^p(\Omega:C([0,T]:E))$ for %
\eqref{eq:Y-control-def}.
\end{theorem}

\begin{proof}
Fix $\e\ge0$, $u \in {\mathcal{P}^N_2}$ and $x \in E^{\star\star}$. For $R\in (0,\infty)$, let $Y^{\e,u}_{x,R}$ be as in \eqref{eq:eqyeuxr}. Then the unique solvability of \eqref{eq:eqyeuxr} for every $R$ implies that for any $0<R_1<R_2$, $Y^{\e,u}_{x,R_1}(t) = Y^{\e,u}_{x,R_2}(t)$ for all $t\leq \tau_{R_1,R_2} \wedge T$, where
\[\tau_{R_1,R_2} \doteq \inf\{t>0: |Y^{\e,u}_{x,R_2}(t)|_E \geq R_1\}.\]
From the uniform moment bound in Lemma \ref{lem:a-priori-bounds},
$\lim_{R \to \infty} Y^{\e,u}_{x,R}(t) \doteq Y^{\e,u}_{x}(t)$ exists a.s. for every $t \in [0,T]$ and by an application of
Fatou's lemma we see that
\begin{equation}
	\E |Y^{\e,u}_{x}|_{C([0,T]:E)}^p \leq C|x|_{E^{\star\star}}^p   e^{C(\e^{\frac{p}{2}} + N^{\frac{p}{2}})T},
	\label{eq:Yeux-uniqueness}
\end{equation}
where $C$ is as in the statement of Lemma \ref{lem:a-priori-bounds}.
It is easy to check that this $Y^{\e,u}_{x}$ is the unique solution to \eqref{eq:Y-control-def}.
\end{proof}

\begin{proof}[Proof of Theorems \ref{thm:uniqmilsoln} and \ref{thm:uniqdetcont}]
  Fix $x \in E$, $\e \ge 0$ and $u \in \mathcal{P}^N_2$.
Let $Y^{\e,u}_x \in L^p(\Omega:C([0,T]:E))$ be the unique solution of \eqref{eq:Y-control-def}. Define
$X^{\e,u}_x \doteq \mathcal{M}(S(\cdot)x + Y^{\e,u}_x ).$
Then clearly, $X^{\e,u}_x$ solves
  \begin{align} \label{eq:Xeux}
    X^{\e,u}_x(t) = &\mathcal{M} \left(S(\cdot)x
    +\sqrt{\e} \int_0^\cdot S(\cdot-s)G(s,X^{\e,u}_x(s))dw(s)
    + \int_0^\cdot S(\cdot-s)G(s,X^{\e,u}_x(s))u(s)ds\right).
  \end{align}
Suppose $\tilde X^{\e,u}_x$ is another $E$-valued continuous $\mathcal{F}_t$-progressively measurable process
with \\$\E(\gamma^{-1}(|\tilde X^{\e,u}_x|_{L^\infty([0,T]:E)}))^p < \infty$ that satisfies
\eqref{eq:Xeux}. Then by Assumption \ref{assum:G-new}(a) and (d)
\begin{align*}
  \tilde Y^{\e,u}_x(t) \doteq  \sqrt{\e} \int_0^t S(t-s)G(s, \tilde X^{\e,u}_x(s))dw(s)
   + \int_0^t S(t-s)G(s,\tilde X^{\e,u}_x(s))u(s)ds
\end{align*}
is in $L^p(\Omega:C([0,T]:E))$ and is a solution of \eqref{eq:Y-control-def}. By Theorem \ref{thm:Yeux-uniqueness},
$\tilde Y^{\e,u}_x=  Y^{\e,u}_x$. The unique solvability of \eqref{eq:Xeux} now follows on observing that
$$\tilde X^{\e,u}_x = \mathcal{M}(S(\cdot)x + \tilde Y^{\e,u}_x ) = \mathcal{M}(S(\cdot)x +  Y^{\e,u}_x )
= X^{\e,u}_x.$$
Finally, Theorems \ref{thm:uniqmilsoln} and \ref{thm:uniqdetcont} follow on taking $u=0$ and $\e=0$ respectively.
\end{proof}

\section{Proof of Theorem \protect\ref{thm:unif-Laplace-Y}}

\label{sec:pfuniflap} In this section we provide the proof of Theorem \ref%
{thm:unif-Laplace-Y}. The proof is based on a variational formula \cite[%
Theorem 2]{bdm-2008}. Because bounded subsets of $E^{\star\star}$ endowed with
the weak topology are in general not metrizable, we need to make some
changes to the proof. For example, if $K$ is a weak-$\star$ compact subset
of $E^{\star\star}$, it is not generally true that every sequence $%
\{x_n\}_{n \in {\mathbb{N}}} \in K$ has a convergent subsequence. We
overcome this limitation by using Lemma \ref{lem:U-compact}(3), which says
that every $\{x_n\}_{n \in {\mathbb{N}}} \in K$ has a subsequence (relabled $%
x_n$) and $x^{\star\star} \in E^{\star\star}$ such that $S(\cdot)x_n \to
S(\cdot)x^{\star\star}$ in $L^p([0,T]:E)$.

From Theorem \ref{thm:Y-well-defined}, for every ${\varepsilon }\in (0,1)$
there is a measurable map ${\mathcal{G}}^{{\varepsilon }}:E^{\star \star
}\times C([0,T]:{\mathbb{R}}^{\infty })\rightarrow C([0,T]:E)$ such that for
every $x\in E^{\star \star }$, $Y_{x}^{{\varepsilon }}={\mathcal{G}}^{{%
\varepsilon }}(x,\sqrt{{\varepsilon }}\beta )$. Recall the CONS $\{e_{i}\}$
used to define $\beta $ from $w$ in Section \ref{S:notes-assums}. For $u\in {%
\mathcal{P}}_{2}$ and $t\in \lbrack 0,T]$, we define the ${\mathbb{R}}%
^{\infty }$-valued random variable $\mathbf{u}(t)\doteq (\langle
u(t),e_{i}\rangle _{H})_{i\geq 1}$. By a standard argument based on
Girsanov's theorem, for every $u\in {\mathcal{P}}_{2}^{N}$ and $x\in
E^{\star \star }$, $Y_{x}^{{\varepsilon },u}\doteq {\mathcal{G}}^{{%
\varepsilon }}(x,\sqrt{{\varepsilon }}\beta +\int_{0}^{\cdot }\mathbf{u}%
(s)ds)$ is the unique solution of \eqref{eq:Y-control-def}.
We will need to study convergence properties of $\{Y_{x}^{{\varepsilon }%
,u}\} $. 

\subsection{Convergence Properties of $Y^{{\protect\varepsilon},u}_x$}

We now give two convergence results for controlled processes that will be
needed in the proof of Theorem \ref{thm:unif-Laplace-Y}.

\begin{theorem}
\label{thm:unif-conv} Let $N\in (0,\infty)$. Assume that $\{x_n\}_{n \in {%
\mathbb{N}}}\subset E^{\star\star}$ is a bounded sequence and $x \in
E^{\star\star}$ is such that $S(\cdot)x_n \to S(\cdot)x$ in $L^p([0,T]:E)$.
Assume that $\{{\varepsilon}_n\}_{n \in {\mathbb{N}}}\subset [0,1]$ converges
to $0$, and $\{u_n\}_{n \in {\mathbb{N}}}\subset {\mathcal{P}}_2^N$
converges in distribution (with $\mathcal{S}^N$ equipped with the weak topology) to $u$%
. Then $Y_n \doteq Y^{{\varepsilon}_n, u_n}_{x_n}$ converges in distribution
in $C([0,T]:E)$ to $Y^{0,u}_x$.
\end{theorem}

\begin{proof}
Define $Z_{1,n}$ and $Z_{2,n}$ by
\begin{align*}
  Z_{1,n}(t) &\doteq \sqrt{\e_n} \int_0^t S(t-s) G(s, \mathcal{M}(Y_n + S(\cdot)x_n)(s))dw(s)\\
  Z_{2,n}(t) &\doteq \int_0^t S(t-s) G(s,\mathcal{M}(Y_n + S(\cdot)x_n)(s))u_n(s) ds
\end{align*}
Note that $Y_n = Z_{1,n} + Z_{2,n}$.
By \eqref{eq:Yeux-uniqueness} in the proof of Theorem \ref{thm:Yeux-uniqueness}, for some $\kappa_1 \in (0,\infty)$,
$\sup_n\E|Y_n|_{C([0,T]:E)}^p\leq \kappa_1.$
Thus by Assumption \ref{assum:mathcal-M}(b), for some $\kappa_2 \in (0,\infty)$,
$$\sup_n \E(\gamma^{-1}(|\mathcal{M}(Y_n+ S(\cdot)x_n)|_{L^\infty([0,T]:E)}))^p\leq \kappa_2.$$
 By Assumption \ref{assum:G-new} (c) and (f), $\{Y_n\}_{n\in \NN}$ is tight in $C([0,T]:E)$. Consequently, there is a subsequence (indexed again by $n \in \NN$) converging in distribution to
 $\tilde Y$ for some $C([0,T:E])$ valued random variable $\tilde Y$. As a consequence of Assumption \ref{assum:G-new}(a) and
the bound in the above display, $Z_{1,n} \to 0$ in probability in $C([0,T:E])$.
Thus $(Y_n, Z_{2,n}, u_n)$ converges to $(\tilde Y, \tilde Y, u)$ in distribution and by appealing to the Skorohod representation theorem we can assume without loss of generality that the convergence holds a.s.

 Let
  \begin{align*}
    &\tilde{Z}(t) \doteq \int_0^t  S(t-s)G(s,\mathcal{M}(\tilde{Y}+S(\cdot)x)(s))u(s)ds, \; t\in [0,T].
  \end{align*}
  To complete the proof, in view of the uniqueness given by Theorem \ref{prop:propyoux}, it suffices to show that $\tilde Y = \tilde Z$ a.s.
Note that
\begin{align*}
    |Z_{2,n}(t)  - \tilde{Z}(t)|_E
    &\leq \left| \int_0^t S(t-s)(G(s,\mathcal{M}(Y_n +S(\cdot)x_n)(s)) -G(s,\mathcal{M}(\tilde{Y}+S(\cdot)x)(s)))u_n(s)ds \right|_E \\
    &\quad+\left|\int_0^t  S(t-s)G(s,\mathcal{M}(\tilde{Y}+S(\cdot)x)(s))(u_n(s)-u(s))ds \right|_E\\
    &\doteq |J_{1,n}(t)|_E + |J_{2,n}(t)|_E.
  \end{align*}
In the notation of Assumption \ref{assum:G-new}, $J_{2,n} = \mathcal{L}(\varphi)(u_n -u)$ for $\varphi = \mathcal{M}(\tilde Y + S(\cdot)x)$. For a fixed $\varphi \in L^\infty([0,T]:E)$, Assumption \ref{assum:G-new}(f) guarantees that $\mathcal{L}(\varphi)$ is a compact operator from $L^2([0,T]:H)$ to $C([0,T]:E)$.  In particular, $\mathcal{L}(\varphi)$ is continuous from the weak topology on $L^2([0,T]:H)$ to the norm topology on $C([0,T]:E)$ (see \cite[Proposition VI.3.3(a)]{conway}). Therefore, the fact that $u_n \to u$ weakly in $L^2([0,T]:H)$ with probability one guarantees that $J_{2,n}$ goes to zero uniformly in $t\in[0,T]$.

As for $J_{1,n}$, by Assumption \ref{assum:G-new}(e),
\begin{equation} \label{eq:Y5}
  \sup_{t \in [0,T]} |J_{1,n}(t)|_E \leq \kappa N^{\frac{1}{2}}|\mathcal{M}(Y_n + S(\cdot)x) -\mathcal{M}(\tilde{Y} + S(\cdot)x_n)|_{L^p([0,T]:E)}
\end{equation}
Because $Y_n \to \tilde{Y}$ in $C([0,T]:E)$ with probability 1, $|Y_n|_{C([0,T]:E)}$ is eventually less than $|\tilde{Y}|_{C([0,T]:E)} + 1$. We assumed that $\sup_{n \in \mathbb{N}}|x_n|_{E^{\star\star}}$ is finite. Therefore, Assumption \ref{assum:mathcal-M}(d) guarantees that with probability one
\begin{equation} \label{eq:Y6}
  \lim_{n \to \infty} \sup_{0\le t \le T} |J_{1,n}(t)|_E =0.
\end{equation}
This proves that $\tilde Y = \tilde Z$ a.s. and completes the proof.
\end{proof}
An immediate consequence is the following result for the deterministic controlled equations.

\begin{theorem}
\label{thm:convlimcont} Assume that $\{x_n\}_{n \in {\mathbb{N}}}\subset
E^{\star\star}$ is a bounded sequence such that $S(\cdot)x_n$ converges to $%
S(\cdot)x$ in $L^p([0,T]:E)$, and suppose that for some $N \in {\mathbb{N}}$%
, $\{u_n\}_{n \in {\mathbb{N}}}\subset \mathcal{S}^N$ converges in the weak topology
to $u$. Then $Y^{0,u_n}_{x_n}$ converges in $C([0,T]:E)$ to $Y^{0,u}_x$.
\end{theorem}

\subsection{Proof of Theorem \protect\ref{thm:compact-level-sets-tilde}}

\label{sec:secratefunprop} Recall the rate functions $\tilde I_x$ introduced
in \eqref{eq:rate-fct-Y}. In this section we prove Theorem \ref%
{thm:compact-level-sets-tilde}.

\begin{proof}[Proof of Theorem \ref{thm:compact-level-sets-tilde}]
Let  $K \subset E^{\star\star}$ be compact in the weak-$\star$ topology and let $N\in \NN$.
By the definition of $\tilde{I}_x$ (see \eqref{eq:rate-fct-Y}) and Theorem \ref{thm:convlimcont}
\[\bigcup_{x \in K} \tilde{\Phi}_x(N) = \left\{Y^{0,u}_x: u \in \mathcal{S}^{2N}, x \in K \right\}.\]
Take an arbitrary sequence $(x_n, u_n) \in K \times \mathcal{S}^{2N}$.
By the compactness of $K$ and Lemma \ref{lem:U-compact}(3), we can find $x\in K$ and a subequence (relabeled $x_n$) such that $S(\cdot) x_n \to S(\cdot)x$ in $L^p([0,T]:E)$. Since $\mathcal{S}^{2N}$ is compact in the weak toplogy, we can find a (further) subsequence such that for some $u \in \mathcal{S}^{2N}$, $u_n \to u$ weakly in $\mathcal{S}^{2N}$. From Theorem \ref{thm:convlimcont} we now have that $Y^{0,u_n}_{x_n} \to Y^{0,u}_x$ which completes the proof.
\end{proof}

{\ }

\subsection{Completing the proof of Theorem \protect\ref{thm:unif-Laplace-Y}}

We more or less follow the program of \cite{bd-2000,bdm-2008}, except for
the fact that $E^{\star\star}$ endowed with the weak-$\star$ topology is in
general not metrizable, and therefore not a Polish space. Despite this,
Theorem \ref{thm:unif-Laplace-Y} will follow from Theorem \ref{thm:unif-conv}%
.

Let $h:C([0,T]:E)\rightarrow \mathbb{R}$ be bounded and continuous. For $({%
\varepsilon },x)\in (0,\infty )\times E^{\star \star }$ we define
\begin{equation}
F({\varepsilon },x)\doteq -{\varepsilon }\log {\mathbb{E}}\left( e^{-\frac{%
h(Y_{x}^{\varepsilon })}{{\varepsilon }}}\right) ,
\end{equation}%
and for $x\in E^{\star \star }$%
\begin{equation}
F(0,x)\doteq \inf_{\varphi \in C([0,T]:E)}\{h(\varphi )+\tilde{I}_{x}(\varphi )\}.
\end{equation}%
With this notation, the uniform Laplace principle \eqref{eq:unif-Laplace-Y},
can be stated as
\begin{equation}
\lim_{{\varepsilon }\rightarrow 0}\sup_{x\in K}\left\vert F({\varepsilon }%
,x)-F(0,x)\right\vert =0.  \label{eq:unif-Laplace-F}
\end{equation}%

 Because for any fixed $\e>0$ and $x \in E^{\star\star}$ there is a measurable mapping that sends the infinite dimensional Wiener process $w \mapsto Y^\e_x$, there is a variational representation for $F(\e,x)$ given by  \cite[Theorem 2]{bdm-2008},

   \begin{equation} \label{eq:variational}
     F(\e,x) = \inf_{u \in \mathcal{P}_2} \E \left[\frac{1}{2} \int_0^T |u(s)|_H^2ds + h(Y^{\e,u}_x) \right].
   \end{equation}

 Following the localization method from \cite[Theorem 4.4]{bd-2000} for any bounded and continuous $h: C([0,T]:E) \to \mathbb{R}$ and $\delta>0$, there exists $N>0$ such that for any $\e>0$ and $x \in E^{\star\star}$ there exists $u = u^\e_x \in \mathcal{P}_2^N$ satisfying
 \begin{equation} \label{eq:variational-N}
   F(\e,x) \geq \E \left[\frac{1}{2} \int_0^T |u(s)|_H^2ds + h(Y^{\e,u}_x)\right]-\delta.
 \end{equation}

We prove Theorem \ref{thm:unif-Laplace-Y} in two steps.

\begin{lemma}[Uniform Laplace Principle upper bound]
For any $K \subset E^{\star\star}$ that is compact in the weak-$\star$
topology,
\begin{equation*}
\limsup_{{\varepsilon} \to 0} \sup_{x \in K} \left(-F({\varepsilon},x) +
F(0,x)\right) \leq 0.
\end{equation*}
\end{lemma}

\begin{proof}
  Let $\e_n \downarrow 0$ be arbitrary. For each $n\in \mathbb{N}$, let $\{x_n\}_{n \in \NN} \subset K$ be such that
  \begin{equation} \label{eq:Laplace-upper-sup}
    \sup_{x \in K} \left(-F(\e_n,x) + F(0,x) \right) \leq -F(\e_n,x_n) + F(0,x_n) + \frac{1}{n}.
  \end{equation}
  Fix $\delta >0$. By 
  \eqref{eq:variational-N}, we can find $N\in \NN$ and $\{u_n\}_{n \in \NN} \subset \mathcal{P}_2^N$ such that for all $n\in \NN$
  \begin{equation} \label{eq:Laplace-upper-e}
    F(\e_n,x_n) \ge \E \left[\frac{1}{2} \int_0^T |u_n(s)|_H^2ds + h(Y^{\e_n,u_n}_{x_n}) \right] -\delta.
  \end{equation}
  It is also clear by the definition of $\tilde{I}_x$  that
  \begin{equation} \label{eq:Laplace-upper-0}
    F(0,x_n)  = \inf_{u \in L^2([0,T]:H)} \left[\frac{1}{2} \int_0^T |u(s)|_H^2 ds + h(Y^{0,u}_{x_n}) \right] \leq \E \left[\frac{1}{2} \int_0^T |u_n(s)|_H^2ds + h(Y^{0,u_n}_{x_n}) \right].
  \end{equation}

  Using the fact that $\mathcal{S}^N$ is a compact metric space under the weak topology, we can find a subsequence (which we relabel as $\{n\}$) such that $u_n$ converges in distribution to some $\tilde{u}\in \mathcal{P}_2^N$. By Lemma \ref{lem:U-compact}(3), there exists a further subsequence (which we again relabel as $\{n\}$) and an $x \in K$ such that $S(\cdot)x_n \to S(\cdot)x$ in $L^p([0,T]:E)$. It follows from Theorems \ref{thm:unif-conv} and \ref{thm:convlimcont} that $Y^{\e_n,u_n}_{x_n}$ and $Y^{0,u_n}_{x_n}$ both converge in distribution in $C([0,T]:E)$ to $Y^{0,\tilde u}_{x}$.

  Using the estimates in  \eqref{eq:Laplace-upper-sup}, \eqref{eq:Laplace-upper-e}, and \eqref{eq:Laplace-upper-0},
  \begin{align*}
    &\limsup_{n \to \infty} \sup_{x \in K} (-F(\e_n,x) + F(0,x)) \\
&\quad\leq \limsup_{n \to \infty}(-F(\e_n,x_n) + F(0,x_n)) + \frac{1}{n}\\
    &\quad\leq \limsup_{n \to \infty}\left( -\E \left[\frac{1}{2} \int_0^T |u_n(s)|_H^2ds + h(Y^{\e_n,u_n}_{x_n}) \right] + \E \left[\frac{1}{2} \int_0^T |u_n(s)|_H^2ds + h(Y^{0,u_n}_{x_n}) \right] \right)+ \frac{1}{n}+\delta\\
    &\quad\leq \delta.
  \end{align*}
  The final line is a consequence of the fact that
  \[\lim_{n \to \infty}\E h(Y^{\e_n,u_n}_{x_n}) =\lim_{n \to \infty}\E h(Y^{0,u_n}_{x_n}) = \E(h(Y^{0,\tilde u}_{x})). \]
  Since $\delta>0$ and $\e_n \downarrow 0$ were arbitrary, the result follows.
\end{proof}

\begin{lemma}[Uniform Laplace Principle lower bound]
For any $K \subset E^{\star\star}$ that is compact in the weak-$\star$
topology,
\begin{equation*}
\liminf_{{\varepsilon} \to 0} \inf_{x \in K} \left(-F({\varepsilon},x) +
F(0,x)\right) \geq 0.
\end{equation*}
\end{lemma}

\begin{proof}
  The proof is very similar to the upper bound proof. Let $\e_n\downarrow 0$ be arbitrary. For each $n\in \mathbb{N}$, let $\{x_n\}_{n \in \NN} \subset K$ such that
  \begin{equation} \label{eq:Laplace-lower-inf}
    \inf_{x \in K} \left(-F(\e_n,x) + F(0,x) \right) \geq -F(\e_n,x_n) + F(0,x_n) - \frac{1}{n}.
  \end{equation}
  By the definitions of $F(0,x)$ and $\tilde{I}_x$, we can find a sequence $\{u_n\}_{n \in \NN} \subset \mathcal{S}^N$, where $N\doteq 2\|h\|_{L^\infty(C([0,T]):\mathbb{R})}+1$, such that
  \begin{equation} \label{eq:Laplace-lower-0}
    F(0,x_n) = \inf_{u \in \mathcal{S}^N}\left[ \frac{1}{2} \int_0^T |u(s)|_H^2ds + h(Y^{0,u}_{x_n})\right] >  \frac{1}{2} \int_0^T |u_n(s)|_H^2ds + h(Y^{0,u_n}_{x_n}) - \frac{1}{n}.
  \end{equation}
  By \eqref{eq:variational} and the fact that the chosen $u_n \in L^2([0,T]:H)\subset \mathcal{P}_2$,
  \begin{equation} \label{eq:Laplace-lower-e}
    F(\e_n,x_n)  = \inf_{u \in \mathcal{P}_2} \E\left[\frac{1}{2} \int_0^T |u(s)|_H^2 ds + h(Y^{\e_n,u}_{x_n}) \right] \leq \E \left[\frac{1}{2} \int_0^T |u_n(s)|_H^2ds + h(Y^{\e_n,u_n}_{x_n}) \right].
  \end{equation}

  From the weak compactness of $\mathcal{S}^N$, we can find a subsequence such that $u_n$ converges weakly to some $\tilde{u}$ in $\mathcal{S}^N$. By Lemma \ref{lem:U-compact}(3), there exists a further subsequence such that $S(\cdot)x_n \to S(\cdot)x$ in $L^p([0,T]:E)$ for some $x \in E^{\star\star}$. It follows from Theorem \ref{thm:unif-conv} and \ref{thm:convlimcont}
 that $Y^{\e_n,u_n}_{x_n}$ and $Y^{0,u_n}_{x_n}$ both converge in distribution in $C([0,T]:E)$ to $Y^{0,\tilde u}_{x}$ ($Y^{0,u_n}_{x_n}$ are actually not random).

  Applying estimates  \eqref{eq:Laplace-lower-inf}, \eqref{eq:Laplace-lower-0}, and \eqref{eq:Laplace-lower-e},
  \begin{align*}
    &\liminf_{n \to \infty} \inf_{x \in K} (-F(\e_n,x) + F(0,x))\\
&\quad \geq \liminf_{n \to \infty}(-F(\e_n,x_n) + F(0,x_n)) - \frac{1}{n}\\
    &\quad\geq \liminf_{n \to \infty}\left( -\E \left[\frac{1}{2} \int_0^T |u_n(s)|_H^2ds + h(Y^{\e_n,u_n}_{x_n}) \right] +  \left[\frac{1}{2} \int_0^T |u_n(s)|_H^2ds + h(Y^{0,u_n}_{x_n}) \right] \right)- \frac{2}{n}\\
    &\quad\geq 0.
  \end{align*}
  Since $\e_n\downarrow 0$ was arbitrary, the result follows.
\end{proof}

\section{Proof of Theorem \protect\ref{thm:ULP-implies-LDP}}

\label{sec:fin} Recall that $\Phi _{x}(s)\doteq \{\varphi \in \mathcal{E}%
:I_{x}(\varphi )\leq s\}$ are the level sets of the rate function. Let $\rho
$ be the metric on $\mathcal{E}$ and recall that ${\mathnormal{dist}}(y,B)$ denotes
the distance between a point $x\in \mathcal{E}$ and a set $B\subset \mathcal{%
E}$, i.e., ${\mathnormal{dist}}(y,B)=\inf_{z\in B}\rho (y,z)$.

We begin with the LDP lower bound, namely we prove \eqref{eq:ldpunilowbd}
with $A_0$ replaced by an arbitrary compact set $K$ in ${\mathcal{E}}_0$.
For $\varphi \in \mathcal{E}$, $j\geq0$ and $\delta>0$, define the bounded
continuous function $h_{j,\delta,\varphi}: \mathcal{E} \to \mathbb{R}$
\begin{equation*}
h_{j, \delta,\varphi}(\psi) \doteq j \left(\frac{\rho(\psi,\varphi)}{\delta}
\wedge 1 \right).
\end{equation*}
This function is nonnegative and it is equal to $j$ if $\rho(\varphi, \psi)\ge
\delta$. Therefore,
\begin{equation*}
{\mathbb{E}}\left(\exp\left(-\frac{h_{j,\delta,\varphi}(Z^{\varepsilon}_x)}{{%
\varepsilon}}\right) \right) \leq {\mathbb{P}}(\rho(Z^{\varepsilon}_x,%
\varphi)<\delta) + e^{-\frac{j}{{\varepsilon}}}.
\end{equation*}
Furthermore, because $h_{j,\delta,\varphi}(\varphi) = 0$,
\begin{equation}  \label{eq:h-leq-I}
\inf_{\psi \in \mathcal{E}} \left(h_{j,\delta,\varphi}(\psi) + I_x(\psi)
\right) \leq I_x(\varphi).
\end{equation}
Therefore,
\begin{align}  \label{eq:LDP-geq-ULP}
&\left.{\varepsilon} \log\left({\mathbb{P}}\left(\rho(Z^{{\varepsilon}}_x,
\varphi)< \delta \right) +e^{-\frac{j}{{\varepsilon}}}\right) + I_x(\varphi)
\right.  \notag \\
&\geq \left.{\varepsilon} \log{\mathbb{E}} \left(\exp\left(-\frac{%
h_{j,\delta,\varphi}(Z^{\varepsilon}_x)}{{\varepsilon}}\right) \right) +
\inf_{\psi \in \mathcal{E}} \left(h_{j,\delta,\varphi}(\psi) + I_x(\psi)
\right). \right.
\end{align}

Fix $s_{0}>0$ and take arbitrary sequences ${\varepsilon }_{n}\downarrow 0$,
$x_{n}\in K$, and $\varphi _{n}\in \Phi _{x_{n}}(s_{0})$. By the compactness
of $\bigcup_{x\in K}{\Phi _{x}(s_{0})}$ there exists a convergent
subsequence (relabled as $\varphi _{n}$) such that $\varphi _{n}\rightarrow
\varphi $ in $\mathcal{E}$. Because $\varphi _{n}\rightarrow \varphi $, for
any fixed $\delta >0$, the functions $h_{j,\delta ,\varphi _{n}}$ converge
to $h_{j,\delta ,\varphi }$ uniformly.

We note that for any bounded, continuous functions $g,h:\mathcal{E}%
\rightarrow \mathbb{R}$ and any $\mathcal{E}$-valued random variable $Y$,
\begin{equation}
\left\vert {\varepsilon }\log {\mathbb{E}}\left( \exp \left( -\frac{g(Y)}{{%
\varepsilon }}\right) \right) -{\varepsilon }\log {\mathbb{E}}\left( \exp
\left( -\frac{h(Y)}{{\varepsilon }}\right) \right) \right\vert \leq
|g-h|_{L^{\infty }(\mathcal{E})}  \label{eq:log-continuity}
\end{equation}%
and
\begin{equation}
\left\vert \inf_{\psi \in \mathcal{E}}\left( I_{x}(\psi )+g(\psi )\right)
-\inf_{\psi \in \mathcal{E}}\left( I_{x}(\psi )+h(\psi )\right) \right\vert
\leq |g-h|_{L^{\infty }(\mathcal{E})}.  \label{eq:inf-continuous}
\end{equation}%
Therefore
\begin{align*}
& \liminf_{n\rightarrow \infty }\left( {\varepsilon }_{n}\log \left( {%
\mathbb{P}}(\rho (Z_{x_{n}}^{{\varepsilon }_{n}},\varphi _{n})<\delta )+e^{-%
\frac{j}{{\varepsilon }_{n}}}\right) +I_{x_{n}}(\varphi _{n})\right) \\
& \geq \liminf_{n\rightarrow \infty }\left( {\varepsilon }_{n}\log {\mathbb{E%
}}\left( \exp \left( -\frac{h_{j,\delta ,\varphi _{n}}(Z_{x_{n}}^{{%
\varepsilon }_{n}})}{{\varepsilon }_{n}}\right) \right) +\inf_{\psi \in
\mathcal{E}}\left( h_{j,\delta ,\varphi _{n}}(\psi )+I_{x_{n}}(\psi )\right)
\right) \\
& \geq \liminf_{n\rightarrow \infty }\left( {\varepsilon }_{n}\log {\mathbb{E%
}}\left( \exp \left( -\frac{h_{j,\delta ,\varphi }(Z_{x_{n}}^{{\varepsilon }%
_{n}})}{{\varepsilon }_{n}}\right) \right) +\inf_{\psi \in \mathcal{E}%
}\left( h_{j,\delta ,\varphi }(\psi )+I_{x_{n}}(\psi )\right) \right) \\
& \hspace{1cm}-2\liminf_{n\rightarrow \infty }|h_{j,\delta ,\varphi
_{n}}-h_{j,\delta ,\varphi }|_{L^{\infty }(\mathcal{E})} \\
& =0,
\end{align*}%
where the first inequality is from \eqref{eq:LDP-geq-ULP}, the second from %
\eqref{eq:log-continuity} and \eqref{eq:inf-continuous}, and the last from
the fact that uniform Laplace principle holds and recalling that $%
h_{j,\delta ,\varphi _{n}}$ converges uniformly to $h_{j,\delta ,\varphi }$.

Choosing $j>s_0$ and noting that $I_{x_n}(\varphi_n) \leq s_0$, we now have
\begin{align*}
&\liminf_{n \to \infty} \left( {\varepsilon}_n \log\left( {\mathbb{P}}%
(\rho(Z^{{\varepsilon}_n}_{x_n}, \varphi_n)<\delta)\right) +
I_{x_n}(\varphi_n)\right) \geq 0.
\end{align*}
Since our sequences were arbitrary, the lower bound \eqref{eq:ldpunilowbd}
(with $A_0$ replaced by $K$) follows.

We now prove the upper bound in \eqref{eq:ldpuniuppbd} with $A_{0}$ replaced
by an arbitrary compact $K$ in ${\mathcal{E}}_{0}$. For $\delta >0$, $j\geq
0 $, $x\in {\mathcal{E}}_{0}$, $s\in (0,\infty )$ and $\psi \in {\mathcal{E}}
$, define $h_{j,\delta ,x,s}(\psi )\doteq j-j\left( \frac{{\mathnormal{dist}}(\psi
,\Phi _{x}(s))}{\delta }\wedge 1\right) $. Then for any ${\varepsilon }>0$,
\begin{equation}
{\mathbb{P}}({\mathnormal{dist}}(Z_{x}^{\varepsilon },\Phi _{x}(s))\geq
\delta )\leq {\mathbb{E}}\left( \exp \left( -\frac{h_{j,\delta
,x,s}(Z_{x}^{\varepsilon })}{{\varepsilon }}\right) \right) .
\label{eq:LDP-leq-ULP}
\end{equation}%
If $j>s_{0}\geq s$, then
\begin{equation}
\inf_{\psi \in \mathcal{E}}\left( h_{j,\delta ,x,s}(\psi )+I_{x}(\psi
)\right) \geq s  \label{eq:h-geq-s}
\end{equation}%
because $I_{x}(\psi )>s$ if $\psi \not\in \Phi _{x}(s)$ and $h_{j,\delta
,x,s}(\psi )=j$ if $\psi \in \Phi _{x}(s)$.

By assumption, $\Lambda _{K,s_0}=\bigcup_{x\in K}\Phi _{x}(s_0)$ is a compact
Polish space. We define the Hausdoff metric between closed subsets of $%
B_{1},B_{2}\subset \Lambda _{K,s_0}$ by
\begin{equation*}
\lambda (B_{1},B_{2})\doteq \inf \{\gamma >0:B_{1}\subset B_{2}^{\gamma }\text{
and }B_{2}\subset B_{1}^{\gamma }\},
\end{equation*}%
where $B_{i}^{\gamma }=\{\psi\in \Lambda _{K,s_0}:{\mathnormal{dist}}%
(\psi,B_{i})\leq \gamma \}$. The closed subsets of $\Lambda _{K,s_0}$ form a
compact metric space under this metric. In particular, for any sequence $%
x_{n}\in K$, and $s_{n}\leq s_{0}$, there exists a subsequence (relabeled $%
s_{n}$, $x_{n}$) and a closed subset $B$ such that $\lambda (\Phi
_{x_{n}}(s_{n}),B)\rightarrow 0$. Note that for all closed $%
B_{1},B_{2}\subset \Lambda _{K,s_0}$ and $\psi \in {\mathcal{E}}$
\begin{equation*}
|{\mathnormal{dist}}(\psi ,B_{1})-{\mathnormal{dist}}(\psi ,B_{2})|\leq
\lambda (B_{1},B_{2}).
\end{equation*}%
Thus $h_{j,\delta ,x_{n},s_{n}}$ converges uniformly in ${\mathcal{E}}$ to
\begin{equation*}
h_{j,\delta ,B}(\psi )\doteq j-j\left( \frac{{\mathnormal{dist}}(\psi ,B)}{\delta }%
\wedge 1\right) .
\end{equation*}%
Consequently, by \eqref{eq:LDP-leq-ULP} and \eqref{eq:h-geq-s},
\begin{align*}
& \limsup_{n\rightarrow \infty }\left( {\varepsilon }_{n}\log {\mathbb{P}}({%
\mathnormal{dist}}(Z_{x_{n}}^{{\varepsilon }_{n}},\Phi _{x_{n}}(s_{n}))\geq
\delta )+s_{n}\right) \\
& \leq \limsup_{n\rightarrow \infty }\left( {\varepsilon }_{n}\log {\mathbb{E%
}}\left( \exp \left( -\frac{h_{j,\delta ,x_{n},s_{n}}(Z_{x_{n}}^{\varepsilon
})}{{\varepsilon }_{n}}\right) \right) +\inf_{\psi \in \mathcal{E}}\left(
h_{j,\delta ,x_{n},s_{n}}(\psi )+I_{x_{n}}(\psi )\right) \right) .
\end{align*}%
By \eqref{eq:log-continuity} and \eqref{eq:inf-continuous}, the above
display is bounded above by
\begin{align*}
& \leq \limsup_{n\rightarrow \infty }\left( {\varepsilon }_{n}\log {\mathbb{E%
}}\left( \exp \left( -\frac{h_{j,\delta ,B}(Z_{x_{n}}^{\varepsilon })}{{%
\varepsilon }_{n}}\right) \right) +\inf_{\psi \in \mathcal{E}}\left(
h_{j,\delta ,B}(\psi )+I_{x_{n}}(\psi )\right) \right) \\
& \hspace{1cm}+2|h_{j,\delta ,x_{n},s_{n}}-h_{j,\delta ,B}|_{L^{\infty }(%
\mathcal{E})}.
\end{align*}%
By the uniform Laplace principle and the uniform convergence of $h_{j,\delta
,x_{n},s_{n}}$ to $h_{j,\delta ,B}$, this expression converges to 0. Since ${%
\varepsilon }_{n}>0$, $x_{n}\in K$, $s_{n}\leq s_{0}$ are arbitrary, we have
the desired bound in \eqref{eq:ldpuniuppbd} (with $A_{0}$ replaced by $K$).
\hfill \qed

\section{Exit time and exit place from a domain of attraction}

\label{S:exit}

A motivation for studying  uniform large deviations principles of the form in Theorem %
\ref{thm:LDP} is to prove exit time and exit place asymptotics. Let $%
X^{\varepsilon}_x$ be the solution to a time homogeneous version of %
\eqref{eq:intro-abstract}
\begin{equation*}
dX^{\varepsilon}_x(t) = [AX^{\varepsilon}_x(t) + B(X^{\varepsilon}_x(t))]dt
+ \sqrt{{\varepsilon}} G(X^{\varepsilon}_x(t))dw(t), \ \
X^{\varepsilon}_x(0) =x.
\end{equation*}
We assume that Assumptions \ref{assum:semigroup}, \ref{assum:mathcal-M}, and %
\ref{assum:G-new} hold for any time horizon $T \in (0,\infty)$ so that Theorem \ref%
{thm:LDP} is valid in $C([0,T]:E)$ for any fixed $T$. We additionally
assume that the function $\gamma$ from \eqref{eq:M-growth} is independent of
the time horizon $T$.

Let $D\subset E$ be a bounded open subset, which is a basin of attraction
for the unperturbed system $X^0_x$, by which we mean that Assumption \ref{assum:attraction} given below is satisfied. We use the notation  $\bar{D}$
for the closure of $D$ in $E$, $\partial D$ for the boundary of $D$
and $D^c$ for the complement of $D$.



\begin{assumption}
\label{assum:attraction} The open set $D$ is a basin of attraction for an $O \in E$ for the noiseless dynamical system
$\{X^0_x\}$, namely

\begin{enumerate}[(a)]

\item If $x \in D$, then $X^0_x(t) \in D$ for all $t>0$.

\item For any $\rho >0$, there exists $T_{0}>0$ such that
\begin{equation*}
\sup_{x\in D}|X_{x}^{0}(T_{0})-O|_{E}<\rho .
\end{equation*}

\item There exists $C\in (0,\infty )$ and $\rho _{0}>0$ such that for any $%
\rho \in (0,\rho _{0})$, if $|x-O|_{E}\leq \rho $, then $|X_{x}^{0}(t)-O|\leq
C\rho $ for all $t>0$.
\end{enumerate}
\end{assumption}

Let
\begin{equation}
\tau _{x}^{\varepsilon }\doteq\inf \{t>0:X_{x}^{\varepsilon }(t)\not\in D\}
\label{eq:exit-time}
\end{equation}%
be the exit time from the set $D$. Because the unperturbed system is
attracted to $O$, the stochastic system is unlikely to leave $D$ over short
time periods if ${\varepsilon }$ is small. We are interested in the
asymptotic growth rate of $\tau _{x}^{\varepsilon }$ as ${\varepsilon }%
\rightarrow 0$ as well as the exit behavior $X_{x}^{\varepsilon }(\tau
_{x}^{\varepsilon })$. The results of this section generalize the
Freidlin-Wentzell theory \cite{dz,F-W-book} to an infinite dimensional
setting.

In Theorem \ref{thm:LDP}, we proved that $\{X_{x}^{\varepsilon }\}$
satisfies a uniform large deviations principle with
respect to initial conditions in bounded subsets of $E$. The rate function
was denoted as $I_{x}:C([0,T]:E)\rightarrow \lbrack 0,\infty ]$ in %
\eqref{eq:rate-fct-X}. To emphasize the dependence of the rate function on
the time interval, for the rest of this section we denote the rate function
as $I_{x,0,T}:C([0,T]:E)\rightarrow \mathbb{R}$ and the associated level
sets by $\Phi _{x,0,T}(M)\subset C([0,T]:E)$. The rate function $I_{x,T_1,T_2}$ for $0\le T_1\le T_2$ is defined in a similar manner.

We next introduce the quasipotential (cf. \cite{F-W-book}) associated with the collection
of rate functions $\{I_{x,0,T}\}$ defined as
\begin{equation*}
V(x,y)\doteq \inf \left\{ I_{x,0,T}(\varphi ):\varphi (0)=x,\varphi (T)=y,T\in
(0,\infty )\right\},\;  x,y\in E.
\end{equation*}%
Without ambiguity, we can also define $V$ on subsets of $E$. If $%
D_{1},D_{2}\subset E$, then
\begin{equation*}
V(D_{1},D_{2}) \doteq \inf_{x\in D_{1}}\inf_{y\in D_{2}}V(x,y).
\end{equation*}%
If the initial condition is $O$ we use the notation $V(D_{1})=V(O,D_{1}).$

In many finite dimensional examples, one can easily prove that $V(x,y)$ is
jointly continuous in $x$ and $y$ (see for example \cite{F-W-book}). This is
unlikely to be true in the infinite dimensional setting. In fact, in many
examples, the quasipotential is equal to $\infty$ on a dense subset of $%
E\times E$.

We make some internal and external regularity assumptions about $V$.

\begin{assumption}
\label{assum:V-regularity} The domain $D$ and the quasipotential $V$ satisfy

\begin{enumerate}[(a)]

\item Outer regularity:
\begin{equation}  \label{eq:V-outer-reg}
V(\partial D) = V(\bar{D}^c)
\end{equation}

\item Inner regularity: If $\gamma_\rho \doteq \{x \in E: |x-O|_E \leq \rho\}$
and $N \subseteq \partial D$ , then
\begin{equation}  \label{eq:V-inner-reg}
\lim_{ \rho \to 0} V(\gamma_\rho, N) = V(O,N).
\end{equation}
\end{enumerate}
\end{assumption}


The main result of this section is that under Assumptions \ref%
{assum:attraction} and \ref{assum:V-regularity}, and also the assumptions
stated at the beginning of the section, the following exit time asymptotics
hold.

\begin{theorem}
\label{thm:exit-problems}

\begin{enumerate}
\item For any $x \in D$,
\begin{equation}  \label{eq:exit-rate-expected}
\lim_{{\varepsilon} \to 0} {\varepsilon} \log {\mathbb{E}}
\tau^{\varepsilon}_x = V(\partial D)
\end{equation}

\item For any $x \in D$ and $\eta>0$,
\begin{equation}  \label{eq:exit-rate-in-prob}
\lim_{{\varepsilon} \to 0} {\mathbb{P}} \left(\exp\left({\varepsilon}%
^{-1}(V(\partial D)-\eta) \right) \leq \tau^{\varepsilon}_x \leq \exp\left({%
\varepsilon}^{-1}(V(\partial D) + \eta) \right) \right) = 1.
\end{equation}

\item For any $x \in D$ and any closed $N \subset \partial D$ for which $V(N) > V(\partial D)$%
,
\begin{equation}  \label{eq:exit-place}
\lim_{{\varepsilon} \to 0} {\mathbb{P}}(X^{\varepsilon}_x(\tau^{%
\varepsilon}_x) \in N) = 0.
\end{equation}
\end{enumerate}
\end{theorem}

\subsection{ Properties of $X^{0,u}_x$}

The following two results record some useful properties of $X^{0,u}_x$ which
is the unique solution of \eqref{eq:eq1214ab} for a given $u \in
L^2([0,T]:H) $.

\begin{theorem}
\label{thm:X-continuity} For any $N>0$, the mapping $(x,u) \mapsto X^{0,u}_x
$ is continuous as a map from $E\times \mathcal{S}^N \to C([0,T]:E)$.
\end{theorem}

\begin{proof}
  If $x_n \to x$ in $E$  and $u_n \to u$ in $\mathcal{S}^N$ (endowed with the weak topology), then by Theorem \ref{thm:convlimcont}, $Y^{0,u_n}_{x_n} \to Y^{0,u}_x$ in $C([0,T]:E)$. Furthermore, $S(\cdot)x_n \to S(\cdot)x$ in $C([0,T]:E)$. Therefore, by the continuity of $\mathcal{M}$ given in Assumption \ref{assum:mathcal-M}(c),
  \[X^{0,u_n}_{x_n} = \mathcal{M}(S(\cdot)x_n + Y^{0,u_n}_{x_n}) \to \mathcal{M}(S(\cdot)x + Y^{0,u}_x) = X^{0,u}_x.\]
\end{proof}

If we have a general bounded sequence $|x_n|_E \leq R$, instead of a
convergent sequence, then we cannot in general say that a subsequence of $%
X^{0,u_n}_{x_n}$ converges to $X^{0,u}_x$ in $C([0,T]:E)$ for some $x \in E$ (since $%
X^{0,u_n}_{x_n}(0)=x_n$). The next theorem demonstrates that the only
difficulty is at $t=0$. In particular, there always exists a subsequence
that converges in $C([t_1,T]:E)$ for any $t_1 \in (0,T)$.

\begin{theorem}
\label{thm:X-contin-weak} For any $R\in \lbrack 0,\infty )$, $N\in \lbrack
0,\infty )$, and $0<t_{1}<T$ the set
\begin{equation*}
\left\{ X_{x}^{0,u}:|x|_{E}\leq R,u\in \mathcal{S}^{N}\right\}
\end{equation*}%
is pre-compact in $C([t_{1},T]:E)$.
\end{theorem}

\begin{proof}
  Fix $t_1 \in (0,T)$. Let $|x_n|_E \leq R$ and $u_n \in \mathcal{S}^N$ be arbitrary sequences. It suffices to show that a subsequence of $\{X^{0,u_n}_{x_n}\}$ converges in $C([t_1,T]:E)$.
 Because $\{S(t)\}$ is a compact semigroup, Lemma \ref{lem:U-compact}(3) guarantees that there exists a subsequence (relabeled $x_n$) such that $S(\cdot)x_n$ converges to $S(\cdot)x^{\star\star}$ in $L^p([0,T]:E)$ for some
$x^{\star\star} \in E^{\star\star}$ with $|x^{\star\star}|_{E^{\star\star}} \leq R$ and any $p \in (1,\infty)$. By the compactness of $\mathcal{S}^N$ (in the weak topology), there is a further subsequence such that $u_n \to u$ weakly. By Theorem \ref{thm:convlimcont}, $Y^{0,u_n}_{x_n} \to Y^{0,u}_{x^{\star\star}}$ in $C([0,T]:E)$.

  By Assumption \ref{assum:mathcal-M}(d)  for any $p \in [2,\infty)$,
  \[X^{0,u_n}_{x_n} = \mathcal{M}(S(\cdot)x_n + Y^{0,u_n}_{x_n}) \to \mathcal{M}(S(\cdot)x^{\star\star} + Y^{0,u}_x) \text{ in }L^p([0,T]:E). \]
  Because the convergence is in $L^p([0,T]:E)$, we can find a further subsequence such that
  $X^{0,u_n}_{x_n}(t)$ converges in $E$ for almost all $t \in [0,T]$. In particular, we can find a $t_0 \in (0,t_1)$, such that $X^{0,u_n}_{x_n}(t_0) \to y$ for some $y \in E$. Let $y_n = X^{0,u_n}_{x_n}(t_0)$.
Then for all $t \in [t_0,T]$,
    $X^{0,u_n}_{x_n}(t) =X^{0,\tilde{u}_n}_{y_n}(t-t_0)
 $
  where $\tilde{u}_n(t) \doteq u_n(t+t_0)1_{[0,T-t_0]}(t)$ is a translated version of $u_n$. Note that $\tilde{u}_n$ converges in $\mathcal{S}^N$ to $\tilde{u}$, where $\tilde u(t) \doteq u(t+t_0)1_{[0,T-t_0]}(t)$. Because $y_n \to y \in E$, it follows from Theorem \ref{thm:X-continuity} that
  \[X^{0,\tilde{u}_n}_{y_n} \to X^{0,u}_y \text{ in } C([0,T-t_0]:E).\]
  Therefore,
  \[X^{0,u_n}_{x_n}(\cdot) \to X^{0,u}_y(\cdot -t_0) \text{ in } C([t_0,T]:E).\]
  Because $t_0<t_1$, the convergence is also valid in $C([t_1,T]:E)$.
\end{proof}

\subsection{Proof of Theorem \protect\ref{thm:exit-problems}}

We more or less follow the proof in Chapter 5.7 of \cite{dz}, with
modifications to deal with the infinite dimensionality of the system. For
example, $\bar{D}$ is not compact, so all  proofs require extra care
when proving uniformity with respect to initial condition. For the rest of
this section, let $\gamma _{\rho }\doteq \left\{ x\in E:|x-O|_{E}\leq \rho
\right\} $. Let $\rho_0>0$  be small enough so that $\gamma_{C \rho_0} \subset D$ where $C$ is the constant from Assumption \ref{assum:attraction}(c).

By using the
following lemmas, the proof of Theorem \ref{thm:exit-problems} can be
completed by using the same arguments as in the proof of \cite[Theorem 5.7.11]{dz}.

\begin{lemma}
\label{lem:exit-1} For any $\eta>0$, there exists $T>0$ such that
\begin{equation*}
\liminf_{{\varepsilon} \to 0} \inf_{x \in D} {\varepsilon}\log  {\mathbb{P}}%
(\tau^{\varepsilon}_x \leq T) > - V(\partial D) -\eta.
\end{equation*}
\end{lemma}

\begin{proof}
  Fix $\eta>0$. By the outer regularity of $V$ (Assumption \ref{assum:V-regularity}(a)), there exists $y \in E \setminus \bar{D}$ such that $V(O,y) \leq V(O,\partial D) + \frac{\eta}{3}$. Then by \eqref{eq:rate-fct-X} we can find a $T_1>0$ and a control $\tilde{u} \in L^2([0,T_1]:H)$ such that the controlled trajectory $X^{0,\tilde{u}}_O$ satisfies $X^{0,\tilde{u}}_O(T_1) = y$ and $\frac{1}{2}\int_0^{T_1} |\tilde{u}(s)|_H^2 ds \leq V(O,\partial D) + \frac{2\eta}{3}$.
  Let $a = \dist_E(y,D)>0$. By the continuity property given in Theorem \ref{thm:X-continuity}, there exists a $\rho>0$ such that for all $x$ satisfying $|x-O|_E < \rho$, $|X^{0,\tilde{u}}_O - X^{0,\tilde{u}}_x|_{C([0,T_1]:E)} \leq \frac{a}{2}$.

  By Assumption \ref{assum:attraction}, the unperturbed system $X^{0}_x$ is uniformly attracted to the equilibrium point $O$ and thus there exists a $T_2>0$ such that $\sup_{x \in D} |X^{0}_x(T_2) - O|_E < \rho$. We build a new control $u$ by appending $0$ control with $\tilde{u}$ as follows
  \[u(t) \doteq \begin{cases}
    0 & \text{ if } 0 \leq t \leq T_2 \\
    \tilde{u}(t -T_2) & \text{ if } T_2 < t \leq T_2 + T_1
  \end{cases}.\]
  Let $T = T_1 + T_2$. Notice that $\frac{1}{2}\int_0^T|u(s)|_H^2 ds = \frac{1}{2}\int_0^{T_1}|\tilde{u}(s)|_H^2 ds$ and for any $x \in D$, $\dist(X^{0,u}_x(T), D) > \frac{a}{2}$.

   We have the containment
  \[\left\{\tau^\e_x \leq T \right\} \supset \left\{|X^\e_x - X^{0,u}_x|_{C([0,T]:E)} < \frac{a}{4} \right\}.\]
Then, by the uniform large deviations principle lower bound  (Theorem \ref{thm:LDP}(1)),
  \begin{align*} 
    &\liminf_{\e \to 0} \inf_{x \in D} \e \log \Pro \left(\tau^\e_x \leq T \right) \\
    &\quad \geq \liminf_{\e \to 0} \inf_{x \in D} \e \log \Pro\left(|X^\e_x-X^{0,u}_x|_{C([0,T]:E)} < \frac{a}{4} \right)\\
    &\quad \geq -\frac{1}{2} \int_0^T|u(s)|_H^2 ds \geq -V(O,\partial D) - \frac{2\eta}{3}.
  \end{align*}
  Notice that the previous step is exactly where it is required that the large deviations principle be uniform over the bounded set $D$, and not merely over compact subsets. The result follows because $\eta>0$ is arbitrary.
\end{proof}
Recall $\gamma_{\rho}$ from Assumption \ref{assum:V-regularity}(b).
\begin{lemma}
\label{lem:hard-to-stay-outside} Fix $\rho \in (0,\rho_0)$ and let $\sigma_{x,\rho}^\e = \inf \{t>0:
X^{\varepsilon}_x (t) \in \gamma_\rho \cup \partial D\}$. Then
\begin{equation}
\lim_{ t \to \infty} \limsup_{{\varepsilon} \to 0} \sup_{x \in D}{\varepsilon} \log
 {\mathbb{P}}(\sigma_{x,\rho}^\e>t) = -\infty.
\end{equation}
\end{lemma}

\begin{proof}
  Let $\rho_0$ be as in Assumption \ref{assum:V-regularity}(c). Fix $\rho\in(0,\rho_0)$ so that $\gamma_{C\rho} \subset D$. By Assumption \ref{assum:attraction}(b), there exists $T_0\in (0,\infty)$ such that
  \begin{equation} \label{eq:T0-choice}
    \sup_{x \in D} |X^0_x(T_0)-O|_E \leq \frac{\rho}{2}.
  \end{equation}
  Let $u \in L^2([0,T_0]:H)$ be any control and $x \in D$ be any initial condition such that $X^{0,u}_x(t) \in D$ for all $t \in [0,T_0]$. Let $Y^{0,u}_x(t) = \int_0^t S(t-s)G(X^{0,u}_x(s))u(s)ds$ and $Y^0_x = 0$. With these definitions,
  $X^{0,u}_x = \mathcal{M}(S(\cdot)x + Y^{0,u}_x)$ and $X^0_x = \mathcal{M}(S(\cdot)x + Y^0_x)$.

   By Assumption \ref{assum:G-new}(d),
  \[|Y^{0,u}_x|_{C([0,T_0]:E)} \leq \kappa_1\left( \int_0^{T_0} (\gamma^{-1}(|X^{0,u}_x(s)|_{E}))^p ds \right)^{\frac{1}{p}}|u|_{L^2([0,T_0]:H)}.\]
  The initial condition $x$ and control $u$ were chosen so that $X_x^{0,u}$ does not leave the bounded  set $D$ over $[0,T_0]$. Therefore $|X^{0,u}_x|_{C([0,T_0]:E)}$ is bounded and,
  \[|Y^{0,u}_x|_{C([0,T_0]:E)} \leq \kappa_2 T_0^{\frac{1}{p}}|u|_{L^2([0,T_0]:H)}.\]
By the Lipschitz continuity of $\mathcal{M}$ (Assumption \ref{assum:mathcal-M}(c)) and the fact that $Y^{0,u}_x$ and $Y^0_x$ are bounded in $C([0,T_0]:E)$, there exists  $\kappa_3 \in (0,\infty)$ such that
  \begin{equation} \label{eq:Xu-X0}
    |X^{0,u}_x - X^0_x|_{C([0,T_0]:E)} \leq \kappa_3 |Y^{0,u}_x - Y^0_x|_{C([0,T_0]:E)} = \kappa_3|Y^{0,u}_x|_{C([0,T_0]:E)} \leq  \kappa_4 T_0^{\frac{1}{p}} |u|_{L^2([0,T_0]:H)}
  \end{equation}
where $\kappa_3, \kappa_4$ depend on $T_0$.

  If $|X^{0,u}_x(T_0)-O|_E > \frac{3\rho}{4}$, then because of \eqref{eq:T0-choice} and \eqref{eq:Xu-X0},
  \[\frac{\rho}{4} \leq |X^{0,u}_x(T_0) - X^0_x(T_0)|_E \leq \kappa_4 T_0^{\frac{1}{p}}|u|_{L^2([0,T_0]:H)} .\]
Thus there exists some  $a \in (0,\infty)$ such that whenever $x \in D$, $X^{0,u}_x(t) \in D$ for $t \in [0, T_0]$ and $|X^{0,u}_x(T_0)-O|_E >\frac{3\rho}{4}$,
  \[a < \frac{1}{2} \int_0^{T_0}|u(s)|_H^2ds.\]

  This means that if $x \in D$, then the set of trajectories
  \[\left\{\varphi \in C([0,T_0]:E) : \varphi(0)=x, \varphi(t) \in D \text{ for } t \in  [0,T_0],  |\varphi(T_0)-O|_E> \frac{3\rho}{4} \right\} \cap \Phi_{x,0,T_0}(a) =\emptyset , \]
  and therefore
  \begin{align*}
    &\left\{\varphi \in C([0,T_0]:E) :\varphi(0)=x, \varphi(t) \in D,  |\varphi(t)-O|_E> \rho, \ t \in [0, T_0] \right\} \\ &\subset \left\{\varphi \in C([0,T_0]:E): \dist_{C([0,T_0]:E)}(\varphi,\Phi_{x,0,T_0}(a)) \geq \frac{\rho}{4} \right\}.
  \end{align*}
  Note that $a$ does not depend on the initial condition $x \in D$.

  By the uniform large deviations principle upper bound (Theorem \ref{thm:LDP}(2)),
  \begin{align*}
    &\limsup_{\e \to 0} \sup_{x \in D} \e \log \Pro(\sigma_{x,\rho}^\e> T_0)\\
    &\quad =\limsup_{\e \to 0} \sup_{x \in D} \e \log \Pro(X^\e_x(t) \in D\setminus \gamma_\rho \text{ for } t \in [0,T_0])\\
    &\quad \leq \limsup_{\e\to 0} \sup_{x \in D} \e \log \Pro\left( \dist_{C([0,T_0]:E)}(X^\e_x,\Phi_{x,0,T_0}(a)) \geq \frac{\rho}{4} \right) \leq -a.
  \end{align*}

  We can show that the likelihood of staying in $D \setminus \gamma_\rho$ for longer time periods becomes exponentially less likely by use of Markov property as follows. For $k \in \mathbb{N}$
  \begin{align*}
    \sup_{x \in D} \e \log \Pro(\sigma_{x,\rho}^\e> kT_0)
    & =\sup_{x \in D}\e \log\Pro\left(X^\e_x(t) \in D \setminus \gamma_\rho \text{ for } t \in [0, kT_0] \right)\\
    &\leq  \e \log\left(\sup_{x \in D} \Pro(X^\e_x(t) \in D\setminus \gamma_\rho \text{ for } t \in [0, T_0]) \right)^k.
  \end{align*}
  Therefore,
  \[\limsup_{\e \to 0} \sup_{x \in D} \e \log  \Pro(\sigma_{x,\rho}^\e>k T_0) = -ka\]
  and since $k \in \mathbb{N}$ is arbitrary, the result follows.
\end{proof}

Let $\Gamma _{\rho }\doteq\left\{ x\in E:|x-O|_{E}=2C\rho \right\} $
where $C$ is from Assumption \ref{assum:attraction}(c).
\begin{lemma} \label{lem:Pro-leq-VN}
For any closed set $N \subset \partial D$,
\begin{equation*}
\lim_{\rho \to 0} \limsup_{{\varepsilon} \to 0}\sup_{x \in \Gamma_\rho} {%
\varepsilon} \log {\mathbb{P}}(X^{\varepsilon}_x(\sigma_{x,\rho}^\e) \in N) \leq
-V(N).
\end{equation*}
\end{lemma}

\begin{proof}
  Fix $a < V(N)$ . By the inner regularity of Assumption \ref{assum:V-regularity}(b), there exists
$\rho\in(0,\frac{\rho_0}{2C})$ such that
  \begin{equation}\label{eq:vgam2cr}
	V(\Gamma_{2C\rho},N) > a.\end{equation}
	By Lemma \ref{lem:hard-to-stay-outside}, we can choose $T_0>0$ large enough such that
  \begin{equation} \label{eq:T0-choice-2}
    \limsup_{\e \to 0} \sup_{x \in \Gamma_\rho} \e \log \Pro(\sigma_{x,\rho}^\e > T_0) \leq -a .
  \end{equation}
  The probability $\Pro(X^\e_x(\sigma_{x,\rho}^\e) \in N)$ can be decomposed as
  \begin{align*}
    &\Pro(X^\e_x(\sigma_{x,\rho}^\e) \in N) = \Pro(X^\e_x(\sigma_{x,\rho}^\e) \in N, \ \sigma_{x,\rho}^\e\leq T_0)
    + \Pro(X^\e_x(\sigma_{x,\rho}^\e) \in N, \ \sigma_{x,\rho}^\e>T_0).
  \end{align*}

  Now we show that there exists  $\delta>0$ such that any trajectory $\varphi$ that starts at $\varphi(0)=x \in \Gamma_\rho$ and exits $D$ through $N$ before time $T_0$ has the property that
  \begin{equation} \label{eq:dist-from-level}
    \dist_{C([0,T_0]:E)}(\varphi, \Phi_{x,0,T_0}(a)) >\delta.
  \end{equation}
  This claim says that any trajectory that starts at $x\in\Gamma_\rho$ and satisfies $I_{x,0,T_0}(\varphi) \leq a$, cannot come within distance $\delta$ of a trajectory exiting $D$ through the set $N \subset \partial D$ before time $T_0$. In the important case where $N=\partial D$, the claim says that such a trajectory cannot come within distance $\delta$ of $\partial D$ before time $T_0$.

  To prove \eqref{eq:dist-from-level}, suppose by contradiction that there exist sequences $\varphi_n \in C([0,T_0]:E)$ and $t_n \in [0,T_0]$ such that $\varphi_n(0)=x_n \in \Gamma_{\rho}$, $\varphi_n(t) \in D$ for $t \in [0,t_n]$, $\dist_E(\varphi_n(t_n),  N) \leq \frac{1}{n}$, and $I_{x_n,0,T_0}(\varphi_n) \leq a$. In other words, $\{\varphi_n\}_{n \in \NN}$ is a sequence in $\bigcup_{x \in \Gamma_\rho}\Phi_{x,0,T_0}(a)$ that gets arbitrarily close to exiting $D$ through $N$. There must exist controls $u_n \in \mathcal{S}^{2a}$ such that $\varphi_n = X^{0,u_n}_{x_n}$.

  Assumption \ref{assum:attraction}(c) guarantees that $|X^0_{x_n}(t)-O|_E \leq 2C^2 \rho$ for $t>0$. By \eqref{eq:Xu-X0}, for any $t\in [0,t_n]$,
  \[|X^{0,u_n}_{x_n}(t) - O|_E \leq |X^{0,u_n}_{x_n}(t)-X^0_{x_n}(t)|_E + |X^0_{x_n}(t) - O|_E \leq \kappa_3 t^{\frac{1}{p}} \sqrt{2a} + 2C^2\rho. \]
  Choose $T_1<T_0$ small enough so that the above equation guarantees that for all $n \in \NN$ and $t\in [0,T_1]$,
  \[|X^{0,u_n}_{x_n}(t) - O|_E \leq 4C^2\rho.\]
  Notice that these equations imply that $t_n \in (T_1,T_0]$ for each $n \in \NN$.

  By the compactness of $[T_1,T_0]$ there is a subsequence $\{t_n\} \subset [T_1,T_0]$ that converges to some $t_\star$. By Theorem \ref{thm:X-contin-weak}, there is a further subsequence of $\{\varphi_n\}$ that converges to a limit $\varphi$ in $C([T_1,T_0]:E)$. The above estimates imply that $|\varphi(T_1)-O| \leq 4C^2\rho$, $\varphi(t^\star) \in N$, and $I_{x,T_1,T_0}(\varphi) \leq a$, where $x = \varphi(T_1)$. By continuity, there must exist a time $T_2\in[T_1,T_0]$ such that $\varphi(T_2) \in \Gamma_{2C\rho}$.
  The existence of such a $\varphi$ says that $V(\Gamma_{2 C\rho}, N) \leq a$ which is a contradiction to \eqref{eq:vgam2cr}. Therefore, \eqref{eq:dist-from-level} must hold.

  It follows from \eqref{eq:T0-choice-2}, \eqref{eq:dist-from-level}, and Theorem \ref{thm:LDP}(2) that
  \begin{align*}
    &\limsup_{\e\to 0} \sup_{x \in \Gamma_\rho} \e \log \Pro(X^\e_x(\sigma_{x,\rho}^\e) \in N) \\
    &\quad \leq \limsup_{\e \to 0} \sup_{x \in \Gamma_\rho} \e \log \Big(\Pro(X^\e_x(\sigma_{x,\rho}^\e) \in N, \sigma_{x,\rho}^\e \leq T_0) + \Pro(X^\e_x(\sigma_{x,\rho}^\e) \in N, \sigma_{x,\rho}^\e> T_0)  \Big)\\
    &\quad \leq \max \left\{\limsup_{\e \to 0} \sup_{x \in \Gamma_\rho} \e \log \Pro(X^\e_x(\sigma_{x,\rho}^\e) \in N, \sigma_{x,\rho}^\e \leq T_0), \ \ \limsup_{\e \to 0} \sup_{x \in \Gamma_\rho} \e \log \Pro(\sigma_{x,\rho}^\e>T_0)  \right\}\\
    &\quad \leq \max\left\{ \limsup_{\e \to 0} \sup_{x \in \Gamma_\rho} \e \log \Pro(\dist_{C([0,T_0]:E)}(X^\e_x, \Phi_{x,0,T_0}(a))\geq \delta), -a\right\}\\
    & \quad\leq -a.
  \end{align*}

  The result follows because $a< V(N)$ was arbitrary.
\end{proof}

\begin{lemma}
For every $\rho>0$ such that $\gamma_\rho \subset D$, and $x \in D$,
\begin{equation*}
\lim_{ {\varepsilon} \to 0} {\mathbb{P}}(X^{\varepsilon}_x(\sigma_{x,\rho}^\e) \in
\gamma_\rho) =1
\end{equation*}
\end{lemma}

\begin{proof}
  This is just a consequence of the fact that $X^\e_x$ converges to $X^0_x$ in $C([0,T]:E)$, with probability one, for every $T$ as $\e \to 0$  and by Assumption \ref{assum:attraction}(a) and (b), $X^0_x$ hits $\gamma_\rho$ and never hits $\partial D$.
\end{proof}

The  next lemma says that it is exponentially unlikely for $%
X^{\varepsilon}_x$ to cross from $\gamma_\rho$ to $\Gamma_\rho$ quickly.

\begin{lemma}
\label{lem:not-too-fast} Let $C>0$ be from Assumption \ref{assum:attraction}%
(c). For $\rho\in(0,\rho_0)$ and any $c>0$ there exists $T=T(c,\rho)>0$ such
that
\begin{equation*}
\limsup_{{\varepsilon} \to 0} \sup_{x \in \gamma_\rho} \e \log {\mathbb{P}} \left( |
X^{\varepsilon}_x - O|_{C([0,T]:E)} \geq 2 C\rho \right) < -c.
\end{equation*}
\end{lemma}

\begin{proof}
   Let $c>0$ and $\rho \in (0,\rho_0)$. By Assumption \ref{assum:attraction}(c), if the unperturbed process $X^0_x$ starts at $x \in \gamma_\rho$, then $|X^0_x(t)-O|_E \leq C\rho$ for all $t>0$. If $X^{0,u}_x$ were to reach $\{x \in E: |x-O|_E = \frac{3C\rho}{2}\}$ at time $t$, then by \eqref{eq:Xu-X0},
  \[\frac{C\rho}{2} \leq |X^{0,u}_x(t) - X^0_x(t)|_E \leq \kappa_4 t^{\frac{1}{p}}  |u|_{L^2([0,t]:H)}.\]
  It follows that we can find $T_0 \in (0,1)$ small enough such that
  \[\frac{1}{2}\int_0^{T_0} |u(s)|_H^2 ds > c\]
  for any $u$ such that $x \in \gamma_\rho$ and $|X^{0,u}_x(t)-O|_E = \frac{3C\rho}{2}$ for some $t \in [0,T_0]$. Therefore, for any $x \in \gamma_\rho$, the set
  \begin{align*}
    &\left\{\varphi \in C([0,T_0]:E): \varphi(0) =x, \ \ |\varphi - O|_{C([0,T_0]:E)} \geq 2C\rho\right\} \\
    &\subset \left\{\varphi \in C([0,T_0]:E): \dist_{C([0,T_0]:E)}(\varphi, \Phi_{x,0,T_0}(c)) \geq \frac{C\rho}{2}\right\}.
  \end{align*}
  By the uniform large deviations upper bound (Theorem \ref{thm:LDP}(2)),
  \begin{align*}
    &\limsup_{ \e \to 0} \sup_{x \in \gamma_\rho} \e \log \Pro(|X^\e_x - O|_{C([0,T]:E)} \geq 2C\rho)\\
    &\leq \limsup_{\e \to 0} \sup_{x \in \gamma_\rho} \e \log \Pro\left(\dist_{C([0,T_0]:E)}(X^\e_x, \Phi_{x,0,T_0}(c))\geq \frac{C\rho}{2}\right) \leq -c.
  \end{align*}

\end{proof}

Using Lemmas \ref{lem:exit-1}-\ref{lem:not-too-fast}, Theorem \ref%
{thm:exit-problems} now follows from the proof of \cite[Theorem 5.7.11]{dz}. We note that Lemma \ref{lem:not-too-fast}
and Lemma 5.7.23 of \cite{dz} are not exactly the same. The role of these lemmas are to prove that it is exponentially unlikely for the process $X^\e_x$ to transition from $\gamma_\rho$ to $\Gamma_\rho$ in a short period of time. We remark that
in Lemma 5.7.23 of \cite{dz}, it is proven that it is exponentially unlikely for the process $X^\e_x$ to  deviate from its initial condition over short periods of time uniformly for all $x \in D$. However
that is
 more than what is needed to prove Theorem \ref{thm:exit-problems} and in particular the estimate in  Lemma \ref{lem:not-too-fast} suffices for the proof.

\section{Sufficient conditions for Assumption \protect\ref{assum:G-new}}

\label{sec:suff-conds} In this section we provide a sufficient condition
under which Assumption \ref{assum:G-new} holds. The results of the section hold for any  fixed $T<\infty$.
Many of the arguments in
this section are based on the stochastic factorization method of \cite[%
Section 5.3.1]{DaP-Z}. The sufficient conditions will be verified for
particular models of interest in the next section.


\begin{assumption}
\label{assum:G-sufficient} There exist Banach spaces $E_{1}$, $E_{2}$ such
that (i) $E\subset E_{1}\subset E_{2}$ and the embeddings are continuous,
(ii) for every $t>0$, $S(t)$ can be extended to a bounded operator from $%
E_{2}$ to $E$, and (iii) for some $M\in (0,\infty )$, $r\in \lbrack 0,1/2)$
we have that for every $t\in (0,T]$ and $x\in E_{1}$,
\begin{equation}
|S(t)x|_{E}\leq Mt^{-r}|x|_{E_{1}}.  \label{eq:E1-E-semigroup}
\end{equation}%
For any $s,t>0$ and $x\in E$, $G(s,x)\in \mathscr{L}(H,E_{2})$ and $%
S(t)G(s,x)\in \mathscr{L}(H,E)$. There exists $\alpha \in (r,1/2)$ and $p>%
\frac{1}{\alpha -r}$ such that for every $E$-valued progressively measurable
process $\varphi $ satisfying ${\mathbb{E}}(\gamma ^{-1}(|\varphi
|_{L^{\infty }([0,T]:E)}))^{p}<\infty $, where $\gamma $ is the
nondecreasing function from $\mathbb{R}_{+}$ to $\mathbb{R}_{+}$ satisfying %
\eqref{eq:M-growth}, and any $t\in \lbrack 0,T]$, the stochastic integral
\begin{equation}
Z_{\alpha }(\varphi )(t)\doteq \int_{0}^{t}(t-s)^{-\alpha }S(t-s)G(s,\varphi
(s))dw(s)  \label{eq:eqzalvarp}
\end{equation}%
is well-defined as an $E_{1}$-valued random variable for every $t\in \lbrack
0,T]$. There exists $\zeta \in L_{\text{loc}}^{1}([0,T]:\mathbb{R}_{+})$ such
that for any $E$-valued progressively measurable process satisfying $\varphi
$ satisfying $${\mathbb{E}}(\gamma ^{-1}(|\varphi |_{L^{\infty
}([0,T]:E)}))^{p}<\infty ,$$ and any $t\in \lbrack 0,T]$,
\begin{equation}
\label{eq:G-suff-bound}
{\mathbb{E}}\left\vert Z_\alpha(\varphi)(t)\right\vert _{E_{1}}^{p}\leq {\mathbb{E}}\left( \int_{0}^{t}\zeta
(t-s)(\gamma ^{-1}(|\varphi (s)|_{E}))^{2}ds\right) ^{p/2},
\end{equation}%
and for all $E$-valued progressively measurable processes $\varphi ,\psi \in
L^{p}(\Omega :L^{\infty }([0,T]:E))$ and any $t\in \lbrack 0,T]$,
\begin{equation}
\label{eq:G-suff-Lip}
{\mathbb{E}}\left\vert Z_\alpha(\varphi)(t) - Z_\alpha(\psi)(t)\right\vert _{E_{1}}^{p}\leq {\mathbb{E}}\left(
\int_{0}^{t}\zeta (t-s)|\varphi (s)-\psi (s)|_{E}^{2}ds\right) ^{p/2}.
\end{equation}

\end{assumption}

The following theorem shows that Assumption \ref{assum:G-sufficient} along
with Assumption \ref{assum:mathcal-M} on the map $\mathcal{M}$ and
the compactness of the semigroup $S(t):E\rightarrow E$ (Assumption \ref%
{assum:semigroup}) implies Assumption \ref{assum:G-new}.

\begin{theorem}
\label{thm:suff-conds} If Assumptions \ref{assum:semigroup}, \ref{assum:mathcal-M},
and \ref%
{assum:G-sufficient} are satisfied, then Assumption \ref{assum:G-new} is
also satisfied with $\gamma$ and $p $ as in Assumptions \ref{assum:mathcal-M} and \ref{assum:G-sufficient}.%
\end{theorem}

Assumption \ref{assum:G-sufficient} merely assumes that certain stochastic
integrals are well-defined as $E_{1}$-valued random variables and satisfy
certain moment bounds for each fixed $t>0$. Theorem \ref{thm:suff-conds}
shows this assumption, together with Assumptions \ref{assum:semigroup}, is
enough to prove that the stochastic integrals $Z(\varphi )$ defined in %
\eqref{eq:Z-varphi-def} are $E$-valued continuous stochastic processes.
Furthermore, these assumptions also suffice for the Lebesgue integrals
defined in \eqref{eq:mathcal-L-def} to be well-defined and continuous in $t$%
. Finally, they imply that $\{Z(\varphi ):{\mathbb{E}}(\gamma ^{-1}(|\varphi
|_{L^{\infty }([0,T]:E)}))^{p}\leq R\}$ is tight in $C([0,T]:E)$ and $\{%
\mathcal{L}(\varphi )u:|\varphi |_{L^{\infty }([0,T]:E)}\leq R,u\in \mathcal{S}^{N}\}$
is relatively compact in $C([0,T]:E)$.

The rest of this section is devoted to the proof of Theorem \ref%
{thm:suff-conds}. We assume throughout the section that Assumptions \ref%
{assum:semigroup}, \ref{assum:mathcal-M} and \ref{assum:G-sufficient} are satisfied and $p$, $%
\alpha $, $E_{1},E_{2}$, $\gamma $ are as in Assumption \ref%
{assum:G-sufficient}. Consider the convolution mapping
\begin{equation}
\mathscr{F}_{\alpha }(\varphi )(t)\doteq \int_{0}^{t}(t-s)^{\alpha
-1}S(t-s)\varphi (s)ds,\;t\in \lbrack 0,T].  \label{def:F-alpha}
\end{equation}%
The following result from \cite{DaP-Z} says that $\mathscr{F}_{\alpha }$ is a well-defined
bounded linear map on $L^{p}([0,T]:E_{1})$.

\begin{lemma}[Proposition 5.9 in \protect\cite{DaP-Z}]
\label{lem:F-alpha-regularity} For any $t \in [0,T]$, $\mathscr{F}_\alpha$
is a bounded linear operator from {$L^p([0,t]:E_1) \to C([0,t]:E)$ .}
\end{lemma}

The following lemma is a consequence of the fact that for any $\alpha \in
(0,1)$ and $\sigma <t$,
\begin{equation*}
\int_{\sigma }^{t}(t-s)^{\alpha -1}(s-\sigma )^{-\alpha }ds=\frac{\pi }{\sin
(\alpha \pi )}.
\end{equation*}%
The lemma shows that under Assumptions \ref{assum:semigroup} and \ref%
{assum:G-sufficient}, Assumption \ref{assum:G-new}(a) is satisfied.

\begin{lemma}[Stochastic factorization]
\label{lem:stoch-fact} For any $\{\mathcal{F}_t\}$-progressively measurable $%
E$-valued process $\varphi\in L^\infty([0,T]:E)$ satisfying ${\mathbb{E}}%
(\gamma^{-1}(|\varphi|_{L^\infty([0,T]:E)}))^p<\infty$, the stochastic
integral $Z(\varphi)$ given by \eqref{eq:Z-varphi-def} is a well-defined $%
\mathcal{F}_t$-progressively measurable process in $C([0,T]:E)$ and
\begin{equation}
Z(\varphi)(t) = \frac{\sin(\alpha \pi)}{\pi} \mathscr{F}_\alpha(Z_\alpha(%
\varphi))(t)=\frac{\sin(\alpha\pi)}{\pi} \int_0^t (t-s)^{\alpha-1}
S(t-s)Z_\alpha(\varphi)(s)ds
\end{equation}
where $Z_\alpha(\varphi)$ is as in \eqref{eq:eqzalvarp}. Furthermore, $%
Z(\varphi) \in L^p(\Omega :C([0,T]:E))$ and for any $t \in [0,T]$,
\begin{equation}  \label{eq:Z-Lp-bound}
{\mathbb{E}}|Z(\varphi)|_{C([0,t]:E)}^p \leq \frac{\|\mathscr{F}_\alpha\|^p
}{\pi^p}|\zeta|_{L^1([0,T]:\mathbb{R})}^{\frac{p}{2}}{\mathbb{E}}\int_0^t
(\gamma^{-1}(|\varphi(s)|_{E}))^pds
\end{equation}
where $\|\mathscr{F}_\alpha\| \doteq \|\mathscr{F}_\alpha\|_{\mathscr{L}%
(L^p([0,T]:E_1),C([0,T]:E))}$ denotes the operator norm of $\mathscr{F}%
_\alpha$. Consequently, Assumption \ref{assum:G-new}(a) is satisfied.
\end{lemma}

\begin{proof}
  Let $\varphi$ be as in the statement of the lemma and
$\tilde{Z} \doteq \frac{\sin(\pi\alpha)}{\pi} \mathscr{F}_\alpha(Z_\alpha(\varphi))$. By Assumption \ref{assum:G-sufficient}
(see \eqref{eq:G-suff-bound}) and Lemma \ref{lem:F-alpha-regularity}, for $t \in [0,T]$,
  \[\pi^p\E|\tilde{Z}|_{C([0,t]:E)}^p \leq \|\mathscr{F}_\alpha\|^p\, \E\int_0^t |Z_\alpha(\varphi)(\sigma)|_{E_1}^pd\sigma \leq  \|\mathscr{F}_\alpha\|^p \E \int_0^t \left(\int_0^\sigma \zeta(\sigma - s) (\gamma^{-1}(|\varphi(s)|_E))^2 ds \right)^{\frac{p}{2}}d\sigma.\]
  By Young's inequality for convolutions,
  \[\E|\tilde{Z}|_{C([0,t]:E)}^p \leq \frac{\|\mathscr{F}_\alpha\|^p}{\pi^p}\,|\zeta|_{L^1([0,T]:\mathbb{R})}^{\frac{p}{2}} \E\int_0^t(\gamma^{-1}(|\varphi(s)|_{E}))^pds. \]

   By Definition \ref{def:stoch-int}, in order to show that the stochastic integral in \eqref{eq:Z-varphi-def} is well defined, it suffices to show that
for any $x^\star \in E^\star$ and $t \in [0,T]$,
  \begin{equation} \int_0^t \left<dw(s), [S(t-s)G(s,\varphi(s))]^{\star} x^\star\right>_{H} =\left<\tilde{Z}(t),x^\star \right>_{E,E^\star} . \label{eq:eq1058}\end{equation}
 This will in fact show that $Z(\varphi)=\tilde Z$ and  in view of the above bounds on $\E|\tilde{Z}|_{C([0,T]:E)}^p$ complete the proof of the lemma.
Using the semigroup property of $S$ and the stochastic Fubini Theorem \cite[Theorem 4.33]{DaP-Z}
\begin{align*}
& \int_{0}^{t}\left\langle (t-s)^{\alpha -1}S(t-s)Z_{\alpha }(\varphi
)(s),x^{\ast }\right\rangle _{E,E^{\ast }}ds \\
& \quad =\int_{0}^{t}\int_{0}^{s}\left\langle (t-s)^{\alpha
-1}S(t-s)(s-\sigma )^{-\alpha }S(s-\sigma )G(\sigma ,\varphi (\sigma
))dw(\sigma ),x^{\ast }\right\rangle _{E,E^{\ast }}ds \\
& \quad =\int_{0}^{t}\int_{\sigma }^{t}(t-s)^{\alpha -1}(s-\sigma )^{-\alpha
}ds\left\langle S(t-\sigma )G(\sigma ,\varphi (\sigma ))dw(\sigma ),x^{\ast
}\right\rangle _{E,E^{\ast }} \\
& \quad =\frac{\pi }{\sin (\alpha \pi )}\int_{0}^{t}\left\langle S(t-\sigma
)G(\sigma ,\varphi (\sigma ))dw(\sigma ),x^{\ast }\right\rangle _{E,E^{\ast
}},
\end{align*}%
and thus \eqref{eq:eq1058} follows.
To justify the use of the Fubini Theorem we will show that
  \begin{equation} \label{eq:Fubini-cond}
    \int_0^t (t-\sigma)^{\alpha-1} \left(\E \int_0^\sigma (\sigma-s)^{-2\alpha} |[S(t-s)G(s,\varphi(s))]^\star x^\star|_{H}^2ds  \right)^{1/2}d\sigma<\infty.
  \end{equation}
  The last inequality can be verified by Assumption \ref{assum:G-sufficient}. Indeed, by the Ito isometry, \eqref{eq:E1-E-semigroup}, and Jensen's inequality,
  \begin{align*}
    &\E \int_0^\sigma (\sigma-s)^{-2\alpha}|(S(t-s)G(s,\varphi(s)))^\star x^\star|_H^2 ds
    =\E\left<S(t-\sigma)Z_\alpha(\varphi)(\sigma),x^\star\right>_{E,E^\star}^2 \\
    &\leq \E |S(t-\sigma)Z_\alpha(\varphi)(\sigma)|_E^2 |x^{\star}|_{E^\star}^2
    \leq M^2(t-\sigma)^{-2r}\E |Z_\alpha(\varphi)(\sigma)|_{E_1}^2 |x^{\star}|_{E^\star}^2\\
    &\leq M^2 (t-\sigma)^{-2r} \left(\E |Z_\alpha(\varphi)(\sigma)|_{E_1}^p\right)^{\frac{2}{p}} | x^\star|_{E^\star}^2.
  \end{align*}
  By Assumption \ref{assum:G-sufficient}, the above expression is bounded above by
  \[ M^2(t-\sigma)^{-2r}|x^\star|_{E^\star}^2 \left( \E \left[\int_0^\sigma \zeta(\sigma-s) (\gamma^{-1}(|\varphi(s)|_E))^2 ds\right]^{p/2}\right)^{2/p}.\]
  Thus, by  H\"older's inequality, \eqref{eq:Fubini-cond} is now bounded by
  \begin{align*}
    & \kappa_1 |x^\star|_{E^\star} \int_0^t (t-\sigma)^{\alpha-r-1}
    \left(\E\left[\int_0^\sigma \zeta(\sigma -s) (\gamma^{-1}(|\varphi(s)|_E))^2 ds\right]^{\frac{p}{2}} \right)^{1/p} d\sigma\\
    &\leq  \kappa_1 |x^\star|_{E^\star} \left(\int_0^t (t-\sigma)^{\frac{(\alpha-r-1)p}{p-1}}d\sigma \right)^{\frac{p-1}{p}}
\left(\E \left( \int_0^t\left[\int_0^\sigma \zeta(\sigma-s) (\gamma^{-1}(|\varphi(s)|_E))^2ds\right ]^{\frac{p}{2}} d\sigma\right)\right)^{\frac{1}{p}}.
  \end{align*}
  By Assumption \ref{assum:G-sufficient}, $p>\frac{1}{\alpha-r}$. Therefore, $\frac{(\alpha-r-1)p}{p-1}>-1$ and so the first integral is finite.
  Thus, by Young's inequality for convolutions, the last display is bounded above by
  \[ \kappa_2|\zeta|_{L^1([0,T]: \mathbb{R})}^{\frac{1}{2}} \left(\E\int_0^t  (\gamma^{-1}(|\varphi(s)|_{E}))^pds\right)^{\frac{1}{p}}.\]
  This proves \eqref{eq:Fubini-cond}
and, as discussed earlier, completes the proof of the lemma.
  \end{proof}

We now show that under Assumption \ref{assum:G-sufficient}, Assumption \ref%
{assum:G-new} (b) is satisfied.

\begin{lemma}
\label{lem:stoch-conv-Lipschitz} For any $\{\mathcal{F}_{t}\}$-progressively
measurable $\varphi ,\psi \in L^{p}(\Omega :L^{\infty }([0,T]:E))$ and $t\in
\lbrack 0,T]$,
\begin{equation*}
{\mathbb{E}}\left\vert Z(\varphi )-Z(\psi )\right\vert _{C([0,t]:E)}^{p}\leq
\frac{\Vert \mathscr{F}_{\alpha }\Vert ^{p}}{\pi ^{p}}|\zeta |_{L^{1}([0,T]:%
\mathbb{R})}^{\frac{p}{2}}{\mathbb{E}}\int_{0}^{t}|\varphi (s)-\psi
(s)|_{E}^{{p}}ds.
\end{equation*}%
Consequently, Assumption \ref{assum:G-new}(b) is satisfied.
\end{lemma}

\begin{proof}
  By the stochastic factorization formula (Lemma \ref{lem:stoch-fact}),
  \begin{equation}Z(\varphi) - Z(\psi) = \frac{\sin(\pi\alpha)}{\pi} \mathscr{F}_\alpha(Z_\alpha(\varphi) - Z_\alpha(\psi)),\end{equation}
  where $Z_\alpha$ is defined in \eqref{eq:eqzalvarp}.
   By Lemma \ref{lem:F-alpha-regularity},
  \[\E|\mathscr{F}_\alpha(Z_\alpha(\varphi) - Z_\alpha(\psi))|_{C([0,t]:E)}^p \leq \|\mathscr{F}_\alpha\|^p \E\int_0^t|Z_\alpha(\varphi)(\sigma) - Z_\alpha(\psi)(\sigma)|_{E_1}^p d\sigma,\]
where we recall that $\|\mathscr{F}_\alpha\|$ is the norm of the linear operator $\mathscr{F}_\alpha$.
   Combining the above display with \eqref{eq:G-suff-Lip},
  \begin{align*}
	\E|\mathscr{F}_\alpha(Z_\alpha(\varphi) - Z_\alpha(\psi))|_{C([0,t]:E)}^p &\leq \|\mathscr{F}_\alpha\|^p \E \int_0^t \left(\int_0^\sigma \zeta(\sigma-s)|\varphi(s)-\psi(s)|_E^2ds \right)^{p/2}d\sigma\\
	& \leq \|\mathscr{F}_\alpha\|^p |\zeta|_{L^1([0,T]:\mathbb{R})}^{\frac{p}{2}} {\mathbb{E}}\int_{0}^{t}|\varphi (s)-\psi
(s)|_{E}^{{p}}ds,\end{align*}
  where the last inequality once more uses  Young's inequality for convolutions.
This completes the proof.
\end{proof}
{\
The following lemma shows that the assumption that the stochastic integral
in \eqref{eq:eqzalvarp} is well-defined in fact says that certain Lebesgue
integrals are well-defined in the Pettis sense. }

\begin{lemma}
\label{lem:defined-from-Girsanov} For any $t\in \lbrack 0,T]$, $\varphi \in
L^{\infty }([0,T]:E)$ and $u\in L^{2}([0,T]:H)$, the integral
\begin{equation}
L_{t}^{\alpha }(\varphi )u\doteq \int_{0}^{t}(t-s)^{-\alpha
}S(t-s)G(s,\varphi (s))u(s)ds.  \label{eq:L-alpha-t-def}
\end{equation}%
is well-defined as an element of $E_{1}$ in the Pettis sense, namely $L_{t}^{\alpha }(\varphi )u$ is the unique element of $E_{1}$  that for
every $x^{\star }\in E_{1}^{\star }$ satisfies
\begin{equation*}
\left\langle L_t^\alpha(\varphi)u,x^{\star }\right\rangle _{E_{1},E_{1}^{\star
}}=\int_{0}^{t}\left\langle (t-s)^{-\alpha }S(t-s)G(s,\varphi
(s))u(s),x^{\star }\right\rangle _{E_{1},E_{1}^{\star }}ds.
\end{equation*}
\end{lemma}

\begin{proof}
  By Girsanov theorem (cf. \cite[Theorem 10.14]{DaP-Z}),
  \[\hat{w}(t) \doteq w(t) - \int_0^t u(s)ds\]
  is a cylindrical Wiener process on $H$ under the measure $\hat\Pro$ given by
  \[\frac{d\hat{\Pro}}{d\Pro} \doteq \exp \left(\int_0^T \left<u(s),dw(s)\right>_H - \frac{1}{2}\int_0^T |u(s)|_H^2 ds \right).\]
  Assumption \ref{assum:G-sufficient} guarantees that both
  \[\int_0^t (t-s)^{-\alpha}S(t-s)G(s,\varphi(s))dw(s) \text{ and } \int_0^t (t-s)^{-\alpha}S(t-s)G(s,\varphi(s))d\hat w(s)\]
  are $E_1$-valued random variables with probability one under the measures $\Pro$ and $\hat\Pro$ respectively. Since $\Pro$ and $\hat\Pro$ are mutually absolutely continuous, both stochastic integrals are $E_1$-valued with probability one with respect to either measure.

  Let
  \[\tilde{L} \doteq  \int_0^t (t-s)^{-\alpha}S(t-s)G(s,\varphi(s))dw(s) - \int_0^t (t-s)^{-\alpha}S(t-s)G(s,\varphi(s))d\hat w(s).\]
  Then $\tilde{L}$ is $E_1$ valued. It suffices to show that for any $x^\star \in E_1^\star$,
  \begin{equation} \label{eq:Pettis-L}
    \left<\tilde{L},x^\star\right>_{E_1,E_1^\star} = \int_0^t \left<(t-s)^{-\alpha}S(t-s)G(s,\varphi(s))u(s),x^\star\right>_{E_1,E_1^\star}ds,
  \end{equation}
  As a consequence of Definition 2.3 and the definition of $\hat w$,
  \begin{align*}
    \left<\tilde{L},x^\star\right>_{E_1,E_1^\star} &= \int_0^t \left<dw(s), (t-s)^{-\alpha}[S(t-s)G(s,\varphi(s))]^\star x^\star \right>_H \\
    &\quad- \int_0^t \left< d\hat w(s), (t-s)^{-\alpha}[S(t-s)G(s,\varphi(s))]^\star x^\star \right>_H\\
    &= \int_0^t  \left< u(s), (t-s)^{-\alpha}[S(t-s)G(s,\varphi(s))]^\star x^\star\right>_H ds\\
    &= \int_0^t \left< (t-s)^{-\alpha}S(t-s)G(s,\varphi(s))u(s), x^\star\right>_{E,E^\star}ds.
  \end{align*}
This proves \eqref{eq:Pettis-L} and completes the proof.
\end{proof}

The following lemma shows that under Assumption \ref{assum:G-sufficient},
Assumption \ref{assum:G-new} (d) and (e) are satisfied.

\begin{lemma}
\label{lem:control-term-Lipschitz} For $\varphi \in L^{\infty }([0,T]:E)$
and $u\in L^{2}([0,T]:H)$, $\mathcal{L}(\varphi )u$ from %
\eqref{eq:mathcal-L-def} is well-defined and in $C([0,T]:E)$. Also, for any $%
t\in \lbrack 0,T]$,
\begin{equation}
\left\vert \mathcal{L}(\varphi )u\right\vert _{C([0,t]:E)}\leq \frac{\Vert %
\mathscr{F}_{\alpha }\Vert }{\pi }|\zeta |_{L^{1}([0,T]:\mathbb{R})}^{\frac{1%
}{2}}\left( \int_{0}^{t}(\gamma ^{-1}(|\varphi (s)|_{E}))^{p}ds\right) ^{%
\frac{1}{p}}|u|_{L^{2}([0,t]:H)}  \label{eq:determ-int-bound}
\end{equation}%
and for $\varphi ,\psi \in L^{\infty }([0,T]:E)$ and $u \in L^2([0,T]:H)$.,
\begin{equation}
\left\vert \mathcal{L}(\varphi )u-\mathcal{L}(\psi )u\right\vert
_{C([0,t]:E)}\leq \frac{\Vert \mathscr{F}_{\alpha }\Vert }{\pi }|\zeta
|_{L^{1}([0,T]:\mathbb{R})}^{\frac{1}{2}}|\varphi -\psi
|_{L^{p}([0,t]:E)}|u|_{L^{2}([0,T]:H)}.  \label{eq:control-conv-bound}
\end{equation}%
Consequently Assumption \ref{assum:G-new}(d) and (e) are satisfied.
\end{lemma}

\begin{proof}
  Proofs of the two inequalities are similar, so we only prove \eqref{eq:determ-int-bound}. By Assumption \ref{assum:G-sufficient} and Jensen's inequality, for any $\varphi\in L^\infty([0,T]:E)$ the stochastic integral
  $Z_\alpha(\varphi)(t)$
  is well defined as an $E_1$-valued random variable and
  \begin{equation}
	\E\left|Z_\alpha(\varphi)(t) \right|_{E_1}^2 \leq \int_0^t \zeta(t-s) (\gamma^{-1}(|\varphi(s)|_E))^2 ds. \label{eq:eq451}\end{equation}
  For any $x^\star \in E_1^\star$,
  \begin{align*}
    &\left<Z_\alpha(\varphi)(t),x^\star \right>_{E_1,E_1^\star}
     = \int_0^t (t-s)^{-\alpha} \left< dw(s), [S(t-s)G(s,\varphi(s))]^\star x^\star \right>_H
  \end{align*}
  and so
  by \eqref{eq:stoch-int-L2-norm},
  \begin{align*}
    &\E\left<Z_\alpha(\varphi)(t),x^\star \right>_{E_1,E_1^\star}^2
     = \int_0^t (t-s)^{-2\alpha}\Big|[S(t-s)G(s,\varphi(s))]^\star x^\star \Big|_H^2ds.
  \end{align*}
  The above equalities imply
  \begin{align*}
    &\int_0^t (t-s)^{-2\alpha}\left|[S(t-s)G(s,\varphi(s))]^\star x^\star \right|_H^2ds\\
    &\quad \leq\E \left|Z_\alpha(\varphi)(t) \right|_{E_1}^2|x^\star|_{E_1^\star}^2\\
    &\quad \leq |x^\star|_{E_1^\star}^2 \int_0^t \zeta(t-s)(\gamma^{-1}(|\varphi(s)|_E)^2ds,
  \end{align*}
  where the last inequality is from \eqref{eq:eq451}.
Thus $\tilde{L}_t^\alpha(\varphi) : E_1^\star \to L^2([0,t]:H)$ defined by
  \[(\tilde{L}^\alpha_t(\varphi)x^\star)(s) \doteq (t-s)^{-\alpha}[S(t-s)G(s,\varphi(s))]^\star x^\star\]
  is a bounded linear operator with norm bounded by $(\int_0^t \zeta(t-s)(\gamma^{-1}(|\varphi(s)|_E))^2ds)^{1/2}$.

  Let $L^\alpha_t(\varphi): L^2([0,t]:H) \to E_1$  be the linear operator
  defined in \eqref{eq:L-alpha-t-def}.
%
We claim that $L^\alpha_t(\varphi)$ is a bounded operator and $(L^\alpha_t(\varphi))^\star = \tilde{L}^\alpha_t(\varphi)$.
Indeed, for
 $x^\star \in E_1^\star$ and $u \in L^2([0,T]:H)$,
  \begin{align*}
    &\left<u, \tilde{L}^\alpha_t(\varphi)x^\star \right>_{L^2([0,t]:H)}
    =\int_0^t \left<u(s), (t-s)^{-\alpha}[S(t-s)G(s,\varphi(s))]^\star x^\star \right>_Hds\\
    &= \int_0^t\left<(t-s)^{-\alpha}S(t-s) G(s,\varphi(s))u(s), x^\star\right>_{E_1,E_1^\star}ds
    =\left<L^\alpha_t(\varphi)u,x^\star \right>_{E_1,E_1^\star} .
  \end{align*}
 Consequently,
 \begin{equation} \label{eq:L-alpha-t-bound}
   \left\|L^\alpha_t(\varphi) \right\|_{\mathscr{L}(L^2([0,t]:H),E_1)} \leq \left(\int_0^t \zeta(t-s)(\gamma^{-1}(|\varphi(s)|_E))^2 ds\right)^{1/2}.
  \end{equation}

 From \eqref{eq:L-alpha-t-bound} and an application of Young's inequality we see that $t \mapsto L^\alpha_t(\varphi)u \in L^p([0,T]:E_1)$.
A similar argument as in the proof of Lemma \ref{lem:stoch-fact} now shows that the following factorization formula holds.
  \[(\mathcal{L}(\varphi)u)(t) = \int_0^t S(t-s)G(s,\varphi(s))u(s)ds = \frac{\sin(\pi \alpha)}{\pi} \mathscr{F}_\alpha(L^\alpha_\cdot(\varphi)u)(t).\]
In particular, the left side is well defined in $C([0,T]:E)$ and
  by Lemma \ref{lem:F-alpha-regularity} and \eqref{eq:L-alpha-t-bound},
  \begin{align*}
    &\left|\mathcal{L}(\varphi)u\right|_{C([0,t]:E)}
    \leq \frac{\|\mathscr{F}_\alpha\|}{\pi} \left(\int_0^t |L_\sigma^\alpha(\varphi)u|_{E_1}^p d\sigma\right)^{\frac{1}{p}}\\
    &\leq \frac{\|\mathscr{F}_\alpha\|}{\pi}  \left(\int_0^t \left|\int_0^\sigma \zeta(\sigma-s)(\gamma^{-1}(|\varphi(s)|_E))^2 ds \right|^{\frac{p}{2}}|u|_{L^2([0,\sigma]:H)}^{{p}}d\sigma\right)^{\frac{1}{p}},
  \end{align*}
  which is bounded above by the right side of \eqref{eq:determ-int-bound} by Young's inequality for convolutions.
\end{proof}

Finally we prove that Assumption \ref{assum:G-sufficient} implies the
tightness properties in Assumption \ref{assum:G-new} (c) and (f). We begin by
establishing an important property of the convolution mapping $\mathscr{F}%
_{\alpha}$ introduced in \eqref{def:F-alpha}. By Lemma \ref%
{lem:F-alpha-regularity}, $\mathscr{F}_\alpha$ is a bounded linear operator
from $L^p([0,T]:E_1)$ to $C([0,T]:E)$. In the next theorem, we show that the
fact that $S(t)$ is a compact semigroup implies that $\mathscr{F}_\alpha$ is
a compact operator.

\begin{theorem}
\label{thm:F-alpha} The map $\mathscr{F}_\alpha$ from $L^p([0,T]:E_1)$ to $%
C([0,T]:E)$ is a compact linear operator.
\end{theorem}

\begin{proof}
  It suffices to show that $\{\mathscr{F}_\alpha(\varphi): |\varphi|_{L^p([0,T]:E_1)} \leq 1\}$ is a relatively compact subset of $C([0,T]:E)$.
 We proceed in two steps, using an infinite dimensional Arzela-Ascoli argument. First, we show that for any fixed $ t \in[0, T]$, the set $\{\mathscr{F}_\alpha(\varphi)(t): |\varphi|_{L^p([0,T]:E_1)} \leq 1\}\subset E$ is a relatively compact subset of $E$. Then  we show that $\{\mathscr{F}_\alpha(\varphi): |\varphi|_{L^p([0,T]:E_1)} \leq 1\}\subset C([0,T]:E)$ is equicontinuous.
For the first step, let $M$ and $r$ be as in \eqref{eq:E1-E-semigroup} and for $\delta>0$ define
  \[q(\delta) \doteq M\left(\int_0^\delta s^{\frac{p(\alpha-r-1)}{p-1}} ds\right)^{\frac{p-1}{p}}.\]
  This is finite since $p> \frac{1}{\alpha-r}$ and therefore $\frac{p(\alpha-r-1)}{p-1}>-1$.
By  \eqref{eq:E1-E-semigroup}, \eqref{def:F-alpha} and  H\"older's inequality, for any  $\varphi \in L^p([0,T]:E)$,
  \begin{equation} \label{eq:F-alpha-bound}
    \left|\mathscr{F}_\alpha (\varphi) \right|_{C([0,T]:E)} \leq q(T) |\varphi|_{L^p([0,T]:E_1)}.
  \end{equation}
Next, for any $\delta \in (0,t)$,
  \begin{align}\label{eq:F-alpha-decompose}
   \mathscr{F}_\alpha(\varphi)(t) = &S(\delta)\int_0^{t-\delta} (t-s)^{\alpha-1} S(t-s-\delta)\varphi(s)ds\\ \nonumber
   & + \int_{t-\delta}^t (t-s)^{\alpha-1} S(t-s)\varphi(s)ds.
  \end{align}

  By \eqref{eq:E1-E-semigroup}, \eqref{eq:F-alpha-decompose}, and the fact that $(t-s)^{\alpha-r-1} \leq (t-s-\delta)^{\alpha-r-1}$ for $0\le s< t-\delta <t\le T$
\begin{equation}\left|\int_0^{t-\delta}(t-s)^{\alpha-1}S(t-s-\delta)\varphi(s)ds\right|_E \leq q(T) |\varphi|_{L^p([0,T]:E_1)}\label{eq:eq1110}\end{equation}
and by another application of \eqref{eq:E1-E-semigroup} and
 H\"older's inequality,
  \begin{equation}\left| \int_{t-\delta}^t (t-s)^{\alpha - 1} S(t-s)\varphi(s)ds \right|_E \leq q(\delta) |\varphi|_{L^p([0,T]:E_1)}.\label{eq:eq1114}\end{equation}
From the compactness assumption on the semigroup (see Assumption \ref{assum:semigroup}), for every $R>0$,
   $K^{\delta,R} \doteq \{S(\delta)x: |x|_E \leq R\}$ is a relatively compact subset of $E$.
Thus from \eqref{eq:eq1110}
whenever $|\varphi|_{L^p([0,T]:E_1)} \leq 1$,
  \[S(\delta) \int_0^{t-\delta} (t-s)^{\alpha-1} S(t-s-\delta)\varphi(s)ds \in K^{\delta,q(T)}.\]
  From \eqref{eq:F-alpha-decompose} and \eqref{eq:eq1114}  we now see that if
$|\varphi|_{L^p([0,T]:E_1)} \leq 1$ and $0 \leq t \leq T$, then
  \[\dist_E( \mathscr{F}_\alpha(\varphi)(t), K^{\delta,q(T)})\leq q(\delta),\]
  which converges to zero as $\delta \to 0$.
  It now follows that
  \begin{equation} \label{eq:K-alpha-p-T}
    {K}_{T}=\{\mathscr{F}_\alpha(\varphi)(t): |\varphi|_{L^p([0,T]:E_1)}\leq 1, \  0\leq t\leq T\}
  \end{equation}
   is a relatively compact subset of $E$. This completes the first step.

  We now show that $\{\mathscr{F}_\alpha(\varphi):|\varphi|_{L^p([0,T]:E_1)}\leq 1 \}$ is equicontinuous. Note that for $\delta >0$ and $\sigma<t \le \sigma+\delta$,
  \begin{align*}
    \mathscr{F}_\alpha(\varphi)(t) - \mathscr{F}_\alpha(\varphi)(\sigma) &= \int_\sigma^t (t-s)^{\alpha-1}S(t-s)\varphi(s)ds \\&\quad + (S(t-\sigma)-I)\int_0^\sigma(\sigma-s)^{\alpha-1} S(\sigma-s)\varphi(s)ds\\
    &\quad + \int_0^\sigma \left((t-s)^{\alpha-1} - (\sigma-s)^{\alpha -1} \right)S(t-s)\varphi(s)ds \\
    &\doteq J_1 + J_2 + J_3
  \end{align*}
By \eqref{eq:eq1114},
  \[|J_1|_E \leq q(\delta) |\varphi|_{L^p([0,T]:E)} \le q(\delta).\]
  Next, from \eqref{eq:K-alpha-p-T}
  \[|J_2|_E \leq \sup_{x \in K_{T}} |(S(t-\sigma) - I)x|_E \le \sup_{x \in K_{T}, u \in [0,\delta]} |(S(u) - I)x|_E
\doteq \theta_1(\delta),\]
  where $K_{T}$ is the relatively compact set defined in \eqref{eq:K-alpha-p-T}.
  Also, by H\"older's inequality,
  \[|J_3|_E \leq \left(\int_0^\sigma \left( (\sigma-s)^{\alpha-r-1} - (t-s)^{\alpha-r-1}  \right)^{\frac{p}{p-1}}ds \right)^{\frac{p-1}{p}} |\varphi|_{L^p([0,T]:E_1)}. \]
Note that
$$\int_0^\sigma  \left( (\sigma-s)^{\alpha-r-1} - (t-s)^{\alpha-r-1}  \right)^{\frac{p}{p-1}}ds
\le \int_0^T [u^{\alpha-r-1} - (u+\delta)^{\alpha-r-1}]^{\frac{p}{p-1}}du \doteq \theta_2(\delta).$$
Combining the above estimates we have for every $\varphi$ with $|\varphi|_{L^p([0,T]:E_1)}\leq 1$,
$$\sup_{(t,\sigma)\in [0,T]: |t-\sigma|\le \delta}|\mathscr{F}_\alpha(\varphi)(t) - \mathscr{F}_\alpha(\varphi)(\sigma)|
\le q(\delta) + \theta_1(\delta) + (\theta_2(\delta))^{\frac{p-1}{p}}.$$
By definition, $q(\delta) \to 0$ as $\delta \to 0$. Also by strong continuity of the semigroup and compactness of
$K_T$, $\theta_1(\delta)\to 0$ as $\delta \to 0$. Finally $\theta_2(\delta)\to 0$ as $\delta \to 0$ by an application
of dominated convergence theorem. Using these observations in the above display we have the desired equicontinuity.
This completes the proof.
\end{proof}

The following lemma shows that under Assumption \ref{assum:G-sufficient},
Assumption \ref{assum:G-new} (c) and (f) are satisfied.

\begin{lemma}
\label{lem:integral-tightness} Let $\mathcal{L}(\varphi )$ and $Z(\varphi )$
be given by \eqref{eq:mathcal-L-def} and \eqref{eq:Z-varphi-def}
respectively. Then for any $R\in (0,\infty )$, the collection of random
variables
\begin{equation*}
\left\{ Z(\varphi ):\varphi \text{ is }\mathcal{F}_{t}\text{- progressively
measurable and }{\mathbb{E}}(\gamma ^{-1}(|\varphi |_{L^{\infty
}([0,T]:E)}))^{p}\leq R\right\}
\end{equation*}%
is tight in $C([0,T]:E)$ and for any $R\in (0,\infty )$ and $N\in (0,\infty )
$,
\begin{equation*}
\left\{ \mathcal{L}(\varphi )u:|\varphi |_{L^{\infty }([0,T]:E)}\leq R,u\in
\mathcal{S}^{N}\right\}
\end{equation*}%
is a pre-compact subset of $C([0,T]:E)$. Therefore Assumption \ref%
{assum:G-new}(c) and (f) are satisfied.
\end{lemma}

\begin{proof}
  From Lemma \ref{lem:stoch-fact} we  have that
  \[Z(\varphi) = \frac{\sin(\pi\alpha)}{\pi} \mathscr{F}_\alpha(Z_\alpha(\varphi)),\;\;
  \mathcal{L}(\varphi)u = \frac{\sin(\pi\alpha)}{\pi}\mathscr{F}_\alpha(L_\cdot^\alpha(\varphi)u)\]
  where $Z_\alpha(\varphi)$ is given in \eqref{eq:eqzalvarp} and $L_t^\alpha(\varphi)$ is given in \eqref{eq:L-alpha-t-def}.

  We showed in the proof of Lemma \ref{lem:stoch-fact} that
  \begin{align*}
    &\E\int_0^T |Z_\alpha(\varphi)(t)|_{E_1}^pdt \leq  \kappa|\zeta|_{L^1([0,T]: \mathbb{R})}^{\frac{p}{2}}\E\int_0^T(\gamma^{-1}(|\varphi(s)|_E))^pds.
  \end{align*}
  Because $\mathscr{F}_\alpha$ is a compact operator by Theorem \ref{thm:F-alpha} and
  \[\sup\left\{ \E|Z_\alpha(\varphi)|_{L^p([0,T]:E_1)}^p : \varphi \text{ is }\mathcal{F}_t\text{-progressively measurable and }\E(\gamma^{-1}(|\varphi|_{L^{\infty}([0,T]:E)}))^p\leq R\right\}< \infty,\]
  it follows that
  \[\left\{Z(\varphi):  \varphi \text{ is }\mathcal{F}_t\text{-progressively measurable and }\E(\gamma^{-1}(|\varphi|_{L^{\infty}([0,T]:E)}))^p \leq R\right\}\]
  is tight in $C([0,T]:E)$.
%
%
  Along the same lines, \eqref{eq:L-alpha-t-bound} implies that if $|\varphi|_{L^\infty([0,T]:E)} \leq R$ and $|u|_{L^2([0,T]:H)} \leq N$,
  \[\sup\left\{\left(\int_0^T |L_t^\alpha(\varphi)u|_{E_1}^pdt \right): |\varphi|_{L^\infty([0,T]:E)} \leq R, u \in \mathcal{S}^N\right\}<\infty. \]
  Therefore because $\mathscr{F}_\alpha$ is a compact operator,
  \[\left\{\mathcal{L}(\varphi)u : |\varphi|_{L^\infty([0,T]:E)}\leq R, u \in \mathcal{S}^N \right\}\]
  is pre-compact in $C([0,T]:E)$.
\end{proof}
We can now complete the proof of Theorem \ref{thm:suff-conds}.

\emph{Proof of Theorem \ref{thm:suff-conds}.} Suppose that Assumptions \ref%
{assum:semigroup} and \ref{assum:G-sufficient} are satisfied. Then Lemma \ref%
{lem:stoch-fact} shows that Assumption \ref{assum:G-new} (a) is satisfied;
Lemma \ref{lem:stoch-conv-Lipschitz} verifies part (b) of the assumption;
Lemma \ref{lem:control-term-Lipschitz} verifies parts (d) and (e); and Lemma %
\ref{lem:integral-tightness} verifies parts (c) and (f). \hfill \qed

\section{Examples}

\label{S:examples}

In this section we discuss two examples of Banach space valued stochastic differential equations that can be treated using the methods developed in this work. The first example considers a class of reaction-diffusion equations that have been studied in \cite{c-2003,cr-2004}.  To keep the presentation simple we make several simplifying assumptions while keeping  the two main technically challenging ingredients in the model, namely a multiplicative degenerate noise, and reaction term that is locally Lipschitz rather than Lipschitz.
Using results from Sections  \ref{S:Laplace} and \ref{sec:suff-conds} we give a different proof of the uniform large deviation principle established in \cite{cr-2004}.

Our second example considers a two-dimensional Navier-Stokes equation exposed to small multiplicative noise. We study the solutions in $C([0,T]:H)$ where $H$ is a Hilbert-space of divergence free $L^2$ vector fields. The nonlinearity in the Navier-Stokes equation, $B(u) =(u \cdot \nabla)u$, is not locally Lipschitz when $u$ is in $H$; in fact, it requires extra regularity in order to be well-defined. This complication requires that we formulate a modified version of Assumption \ref{assum:G-new} in order to prove that the solutions to the two-dimensional Navier-Stokes equation satisfy a large deviations principle uniformly over initial conditions in bounded subsets of $H$.

\subsection{Reaction-Diffusion Equation with multiplicative colored noise}
Consider the reaction diffusion equation with Dirichlet boundary conditions on a bounded open set $\mathcal{O} \subset \mathbb{R}^d$ with a smooth boundary. This was studied by Sowers \cite{s-1992}, Cerrai and Rockner \cite{cr-2004}, and others. 
In particular the uniform large deviation principle for this model follows from results in \cite{cr-2004}. However the goal of this section is to provide a simpler proof using results of Sections \ref{S:Laplace} and \ref{sec:suff-conds}.

Consider the stochastic reaction diffusion equation of the form
\begin{equation} \label{eq:RDE}
  \begin{cases}
    \frac{\partial X^\e_x}{\partial t}(t,\xi) = \mathcal{A} X^\e_x(t,\xi) + b(t,\xi,X^\e_x(t,\xi)) + \sqrt{\e}g(t,\xi, X^\e_x(t,\xi))Q\frac{\partial w}{\partial t }(t,\xi), \ \ t\ge 0, \xi \in \bar{\mathcal{O}} \\
    X^\e_x(t,\xi) = 0, \ \ \ (t,\xi) \in [0,\infty)\times \partial \mathcal{O}, \ \
    X^\e_x(0,\xi) = x(\xi), \; \xi \in \bar{\mathcal{O}},
  \end{cases}
\end{equation}
Here $\mathcal{A}$ is a second order differential operator of the form
$$\mathcal{A}(\xi) = \sum_{h,k=1}^d a_{hk}(\xi) \frac{\partial^2}{\partial \xi_h \partial \xi_k} + \sum_{h=1}^d \vartheta_h(\xi) \frac{\partial}{\partial \xi_h}, \; \xi \in \bar{\mathcal{O}},$$
where $a_{hk}$ are in $C^1(\bar{\mathcal{O}})$ and $\vartheta_h$ are in $C(\bar{\mathcal{O}})$. The $[a_{hk}]$ matrix is non-negative, symmetric and uniformly elliptic. Conditions on real valued functions $b$ and $g$ will be specified below.
The noise is given as a cylindrical Brownian motion $w$ on the Hilbert space $H=L^2(\mathcal{O})$. $Q$ is a bounded linear operator on $H$, additional conditions on which will be specified below.
The solution $X^\e_x(t,\cdot)$ for equation \eqref{eq:RDE} will take values in the Banach space
$E = C_0(\bar{\mathcal{O}})$, namely the space of continuous functions on $\bar{\mathcal{O}}$ vanishing at the boundary. Thus $X^\e_x$ will have sample paths in $C([0,T]:E)$.


We will denote by $A$  the realization of $\mathcal{A}$ endowed with the Dirichlet boundary condition. Then $A$ generates an analytic semigroup $\{S(t)\}_{t\ge 0}$ in $H$ whose restriction
to $E = C_0(\bar{\mathcal{O}})$  defines a $C_0$-semigroup on $E$ (see \cite{cr-2004} and references therein). Furthermore for every $t>0$, $S(t)$ is a compact operator on $E$ and in particular $\{S(t)\}$ satisfies Assumption \ref{assum:semigroup}.


We now introduce the assumption on the coefficients $b$ and $g$. Define for $t\ge 0$, $\xi \in \bar{\mathcal{O}}$ and $x:\bar{\mathcal{O}}\to \mathbb{R}$, $B(t,x)(\xi) \doteq b(t,\xi, x(\xi))$.
\begin{assumption} \label{assum:b-f-Lip}
   There is a locally bounded $\Phi: [0,\infty) \to \mathbb{R}_+$ such that the following hold (see Hypothesis 4 of \cite{cr-2004}).
  \begin{enumerate}[(a)]
    \item For every $t\ge 0$, $B(t): E\to E$ is locally Lipschitz continuous, locally uniformly in $t$
        and there is an $m \in \mathbb{N}$ such that for all $x \in E$
$$|B(t,x)|_E \le \Phi(t) (1+ |x|^m_E).$$
    \item For all $t\ge 0$ and $x,h \in E$, there exists some
    $$\delta_h \in \partial |h|_E \doteq \{h^\star \in E^\star: |h^\star|_{E^\star}=1, \left<h,h^\star\right>_{E,E^\star} = |h|_E \},$$ such that
        $$\left<B(t,x+h) - B(t,x), \delta_h \right>_{E,E^\star} \leq \Phi(t)(1 + |h|_E+|x|_E).$$
\item For all $t\ge 0$ and $\sigma \in \mathbb{R}$
  \[\sup_{\xi \in \bar{\mathcal{O}}} |g(t,\xi,\sigma)| \leq \Phi(t)(1 + |\sigma|^{\frac{1}{m}}).\]
%
    \item The function $g: [0,\infty)\times \bar{\mathcal{O}} \times \mathbb{R} \to \mathbb{R}$ is  continuous and the following Lipschitz property
holds: For all $t\ge 0$ and $\sigma, \rho \in \mathbb{R}$
  \[\sup_{\xi \in \bar{\mathcal{O}}} |g(t,\xi,\sigma) - g(t,\xi,\rho)|  \leq \Phi(t)|\sigma-\rho|.\]
 \end{enumerate}
\end{assumption}

With $B$ as defined before Assumption \ref{assum:b-f-Lip}, it can be checked that Assumption \ref{assum:mathcal-M} is satisfied with $\gamma(x) \doteq \kappa_1(1+x^m)$ for some
$\kappa_1 \in (0,\infty)$.
Indeed, Assumption \ref{assum:mathcal-M}(a) and (b) are proven as in \cite[Lemma 5.4]{c-2003}.   See also \cite{cr-2004}. Assumption \ref{assum:mathcal-M}(c) and (d) are consequences of the fact that $B(t): E \to E$ is locally Lipschitz continuous, locally uniformly in $t$, and the Gr\"onwall inequality.

We now specify our condition on the operator $Q$.

\begin{assumption} \label{assum:Q}
  $Q$ is a bounded linear operator on $H=L^2(\mathcal{O})$ with complete orthonormal system of eigenfunctions $\{f_j\}_{j \in \NN} \subset H$ such that
  $Qf_j = \lambda_j f_j.$
  If $d \geq 2$, there exists a number $2< q < \frac{2d}{d-2}$
  \begin{equation}
    \sum_{j=1}^\infty \lambda_j^{q} < \infty.
  \end{equation}
  If $d=1$ we require $\sup_j \lambda_j<\infty$.
%
\end{assumption}
With $E_2=H$, define $G: [0,\infty)\times E \to \mathscr{L}(H,E_2)$ as follows. For $t\ge 0$, $x\in E$
$$(G(t,x)h)(\xi) \doteq g(t,\xi, x(\xi))(Qh)(\xi), \;\; h \in H, \xi \in \mathcal{O}.$$

With $\mathcal{M}$ as in Assumption \ref{assum:mathcal-M}, let $X^{\e}_x$ be the mild solution of \eqref{eq:intro-abstract} in the sense of Definition \ref{def:mildsoln}. In view of Theorem \ref{thm:uniqmilsoln},
in order to check that \eqref{eq:intro-abstract} has a unique mild solution, it suffices to verify Assumption \ref{assum:G-new}. For that, due to Theorem \ref{thm:suff-conds}, it is enough to show that Assumption \ref{assum:G-sufficient} is satisfied.  Recall that $E_2$ is taken to be $H$. With this choice of $E_2$ it is shown in \cite[Proof of Theorem 4.2]{c-2003} that, with $G$ and $A$ as defined in this section, Assumption \ref{assum:G-sufficient} is satisfied with $E_1 = L^p(\mathcal{O})$ for some $p>2$ and $\zeta(t) = t^{-(2\alpha + \frac{d}{2\zeta_*})}$, for some $\alpha \in (0, 1/2)$ and $\zeta_*>0$ which are specified in \cite{c-2003}.

Thus we have verified Assumptions \ref{assum:semigroup}, \ref{assum:mathcal-M}, and \ref{assum:G-new}. Therefore from Theorem \ref{thm:LDP}
	 $\{X^{\varepsilon}_x\}_{\e>0}$ satisfies a large deviation principle in $C([0,T]:E)$ with respect to the rate function
	$I_x$ in \eqref{eq:rate-fct-X}, uniformly in the class  of all bounded subsets of $E$.

\subsection{Two-dimensional stochastic Navier-Stokes equation}
 The goal of this section is to present an example that illustrates that the general program for establishing a uniform large deviations principle developed in Sections 1-7 can be implemented even if Assumptions
\ref{assum:semigroup}, \ref{assum:mathcal-M} and \ref{assum:G-new} do not hold exactly in the form presented in Section \ref{S:notes-assums}. For this we will consider a
2-dimensional Navier-Stokes equation on a smooth, open, bounded spatial domain $ \mathcal{O} \subset \mathbb{R}^2$ of the form

\begin{equation} \label{eq:NS}
  \begin{cases}
    \frac{\partial u}{\partial t}(t,\xi) = \Delta u(t,\xi) -(u\cdot \nabla)u(t,\xi) - \nabla p(t,\xi) + \sqrt{\e}g(u(t,\xi))Q\frac{\partial w}{\partial t}(t,\xi), \ \ (t,\xi) \in [0,\infty)\times \bar{\mathcal{O}} \\
    \textnormal{div} u(t,\xi) = 0, \ \ (t,\xi) \in [0,\infty)\times \bar{\mathcal{O}},\ \ u(t,\xi) = 0, \ \ (t,\xi) \in [0,\infty)\times \partial \mathcal{O}, \ \
    u(0,\xi) = x(\xi), \; \xi \in \bar{\mathcal{O}},
  \end{cases}
\end{equation}
where  the first requirement in the second line is the incompressibility condition, the second a Dirichlet boundary condition and the third an initial condition.
The noise $w(t,\xi)$ in the above equation is a cylindrical Brownian motion on the Hilbert space  $L^2\doteq L^2(\mathcal{O}:\RR^2)$ given on some filtered probability space $(\Omega, \mathcal{F}, \Pro, \{\mathcal{F}_t\})$. Conditions on the operator $Q$ will be specified below.
The solution is a pair $u(t,\xi)\in \mathbb{R}^2$ and $p(t,\xi) \in \mathbb{R}$. The equation can be interpreted as an abstract evolution equation with $u$ and $p$ in appropriate function spaces.
A uniform large deviations principle for the 2-D Navier-Stokes equation with additive noise was recently established in \cite{bcf-2015} and  a large deviations result for the multiplicative noise case was studied in \cite{ss-2006}. The latter paper  did not consider  a uniform large deviations principle.

We now introduce our main assumptions on the model. The first assumption is on the operator $Q$ and the second on the diffusion coefficient $g$.
\begin{assumption} \label{assum:NS-noise-reg}
   $Q:L^2 \to L^2$ is a nonnegative linear  operator with  a complete orthonormal system $\{f_k\}$ of eigenfunctions, i.e., $Q f_k = \lambda_k f_k$, $\lambda_k\geq0$ for $k \in \mathbb{N}$. We also assume that the $\{f_k\}$ are uniformly bounded, i.e., $\sup_k|f_k|_{L^\infty(\mathcal{O}:\RR^2)}<\infty$ . Furthermore, we assume that  $Q^2$ is trace class, namely
  \[\Tr(Q^2) \doteq \sum_{k=1}^\infty \lambda_k^2 <\infty.\]
\end{assumption}

\begin{assumption} \label{assum:NS-g-reg}
  The function $g:\mathbb{R}^2 \to \mathbb{R}$ is bounded and  Lipschitz continuous, namely
  \[|g|_{L^\infty(\RR^2:\RR)} \doteq \sup_{\eta \in \RR^2} |g(\eta)|_\RR <\infty, \text{ and } \ \ \ \ |g|_{\textnormal{Lip}} \doteq \sup_{\eta_1 \not= \eta_2} \frac{|g(\eta_1) - g(\eta_2)|_\RR}{|\eta_1 - \eta_2|_{\RR^2}} <\infty.\]
\end{assumption}

\begin{remark}
  We assume that the noise is trace class and that $g$ is uniformly bounded only for the purpose of simplifying the proofs of the section. Wellposedness of the  two dimensional Navier-Stokes equation can be established under more general conditions (see for example Assumption 5.1 of \cite{bcf-2015}) and we expect that with additional work similar techniques can be used for establishing uniform large deviation principles in such general settings as well. Also, although we do not include a deterministic forcing term in the equation in order to keep the presentation simple, treatment of the equation with a deterministic forcing term proceeds along similar lines.
\end{remark}

We will use the following definitions and notation from Temam \cite[Section 2.5]{t-1995}.
\begin{itemize}
	\item Let
$\mathcal{V} \doteq \{\varphi \in C^\infty_0(\mathcal{O}:\mathbb{R}^2) : \textnormal{div}\varphi =0 \}$.
That is, $\mathcal{V}$ is the set of infinitely differentiable $\mathbb{R}^2$ valued divergence-free vectors on $\mathcal{O}$ with zero boundary conditions.
\item Let
$H$ be the closure of $\mathcal{V}$ in $L^2(\mathcal{O}:\mathbb{R}^2)$ and
let $P\in \mathscr{L}(L^2,H)$ be the Leray-Helmholtz projection onto $H$.
\item Let $V$ be the closure of $\mathcal{V}$ in $H_0^1$ (the Hilbert space of functions in $L^2(\mathcal{O})$ with derivatives of order one in $L^2(\mathcal{O})$ as well, and vanishing at the boundary).
Let $V'$ be the dual of $V$. Then we have the following dense, continuous embedding $V \hookrightarrow H=H' \hookrightarrow V'$.
\item Let $Au \doteq P\Delta u$ for $u \in D(A) \doteq \{u \in H_0^1 \cap H^2, \textnormal{div} u =0\}$ where $H^2$ is the space of functions in $L^2(\mathcal{O})$ with derivatives of order one and two in $L^2(\mathcal{O})$ as well.
Since $\mathcal{O}$ is a bounded domain, and the pseudo-inverse $(-A)^{-1}$ is a compact operator in $H$, there exists a complete orthonormal basis $\{e_k\}$ of $H$ made up of eigenfunctions for $A$. We order the sequence in such a way that $A e_k = -\gamma_k e_k$ for $0< \gamma_1 \leq \gamma_2 \leq \gamma_3 \leq...$.
 The semigroup generated by $A$ denoted as $\{S(t)\}$ is  $C_0$, analytic, and therefore also compact for $t>0$ as an operator from $H \to H$.
\item For any $\delta \in \mathbb{R}$ we define the fractional powers of $(-A)$ by $(-A)^\delta e_k = \gamma_k^{\delta} e_k$. The Hilbert spaces $H^\delta$ are defined as the completion of $\mathcal{V}$ under the norm $$|x|_{H^\delta} \doteq |(-A)^{\delta/2} x|_H = \left[\sum_{k=1}^\infty \gamma_k^{\delta}\left<x,e_k\right>_H^2\right]^{1/2}.$$ When $\delta=2$ this coincides with the Hilbert space $H^2$ introduced earlier. When $\delta = 1$, $H^1 = V$. When $\delta=0$, $H^0=H$.
\item For $p\geq1$ let $L^p \doteq L^p(\mathcal{O}:\mathbb{R}^2)$.
\item Define the trilinear form $b: V\times V\times V \to \mathbb{R}$ as
\[b(u,v,w) \doteq  \sum_{i,j=1}^2 \int\limits_\mathcal{O} u_i \partial_i v_j w_jdx, \; u,v,w\in V.\]
Using the form, define $B:V\times V \to V'$ as the continuous bilinear operator
$$\langle w, B(u,v) \rangle_{V,V'} \doteq b(u,v,w).$$
We write for $u\in V$, $B(u,u)$ as $B(u)$.
\end{itemize}

The following properties of the above trilinear form will be used in various estimates.
\begin{lemma}[\cite{t-1995} Lemma 2.1] \label{lem:b-regularity}
  $b$ is well defined and trilinear as a map from $H^{m_1}\times H^{1+m_2}\times H^{m_3} \to \mathbb{R}$ where $m_i\geq 0$ and either $m_1 + m_2 +m_3 \geq 1$ and none of them is equal to $1$ or $m_1 + m_2 + m_3>1$ if at least one of the $m_i=1$.
  We will use the estimates, due to \eqref{eq:antisymm}
  \begin{align} \label{eq:b-reg-L4}
    &|b(u,v,w)| \leq \kappa |u|_{L^4} |v|_{L^4} |w|_{H^1},\\
    \label{eq:b-reg-H32}
    &|b(u,v,w)| \leq \kappa |u|_H |v|_{L^4} |w|_{H^{\frac{3}{2}}},\\
    \label{eq:b-reg-H32-switch}
    &|b(u,v,w)| \leq \kappa |u|_{L^4} |v|_{H} |w|_{H^{\frac{3}{2}}}
  \end{align}
\end{lemma}
It can be checked that $b$ is antisymmetric in the sense that for $u,v,w \in V$,
\begin{equation} \label{eq:antisymm}
  b(u,v,w) = -b(u,w,v).
\end{equation}
 Consequently $b(u,v,v) = 0$ for $u \in H$ and $v \in V$.

For $x \in H$, let $\tilde{G}(x)\in \mathscr{L}(L^2)$ be defined as,
  \[[\tilde{G}(x)h](\xi) \doteq g(x(\xi))[Qh](\xi), h  \in L^2.\]
  Let $G(x) \doteq P\tilde{G}(x) \in \mathscr{L}(L^2,H)$.


Applying $P$ to equation \eqref{eq:NS} we get the following abstract evolution equation for $X^\e_x = Pu$
\begin{equation} \label{eq:N-S-abstract}
  dX^\e_x(t) = [A X^\e_x(t) - B(X^\e_x(t))]dt + \sqrt{\e}G(X^\e_x(t))dw(t), \ \  X^\e_x(0) = x \in H.
\end{equation}
Theorem \ref{thm:uniqxeuxNS} will give that there is a unique solution of \eqref{eq:N-S-abstract} (in the sense made precise
in Section \ref{sec:wellposNavSto}) in $C([0,T]:H) \cap L^4([0,T]:L^4)$.
The goal of this section is to establish a   LDP for $\{X^\e_x\}_{\e>0}$  in $C([0,T]:H)$ that is uniform with respect to $x$ in bounded subsets of $H$.
We begin with a discussion of the wellposedness of \eqref{eq:N-S-abstract} and its controlled analogue.
The basic outline of the proof of wellposedness is similar to the approach taken in Section \ref{sec:uniqmildsolns}.
The key conditions for the proof there were Assumptions  \ref{assum:semigroup}, \ref{assum:mathcal-M}, and \ref{assum:G-new}.
As noted earlier, Assumption \ref{assum:semigroup} holds for the current example. In the next two sections (Sections \ref{sec:welposNS} and \ref{sec:vergnew}) we show that although Assumptions \ref{assum:mathcal-M} and \ref{assum:G-new} do not hold in the exact form stated in Section \ref{S:notes-assums}, a slightly modified versions of these assumptions do hold for the example. We then describe in Section
\ref{sec:wellposNavSto} how these suffice for the wellposedness of \eqref{eq:N-S-abstract} and its controlled analogue and finally in Section \ref{sec:uldpNS} we use the properties established in Sections \ref{sec:welposNS} and \ref{sec:vergnew} to establish the desired uniform LDP.


%
%
%
%
%
%
%
%
%
%
%

%
%
\subsubsection{Verification of a Modification of Assumption \ref{assum:mathcal-M}.}
\label{sec:welposNS}

As in Section \ref{sec:uniqmildsolns}, the wellposedness of \eqref{eq:N-S-abstract} and its controlled analogues can be analyzed by studying properties of a certain
deterministic map $\mathcal{M}$ on a suitable function space. However, it turns out that the appropriate map for the current setting does not satisfy Assumption \ref{assum:mathcal-M} in exactly the
form stated in Section \ref{S:notes-assums} and one needs to modify the definition of $\mathcal{M}$ appropriately. Roughly speaking, the map $\mathcal{M}$ will not act on
$L^\infty([0,T]:H)$ but rather on a subspace consisting of functions with additional spatial integrability properties.

Consider the equation
\begin{equation} \label{eq:NS-M}
dv(t) = [Av(t) - B(v(t) + \Psi(t))]dt, \ \ v(0) = 0.
\end{equation}

The following result introduces the map $\mathcal{M}$ that is appropriate for the current setting. Properties of the map given in this result
are the analogues of requirements in Assumption \ref{assum:mathcal-M} that are needed in the study of \eqref{eq:N-S-abstract}.
For $T>0$, let $\mathbf{L_T} \doteq L^4([0,T]:L^4)\cap L^\infty([0,T]:H)$.
The norm on $\mathbf{L_T}$ is given as $|\cdot|_{\mathbf{L_T}} \doteq |\cdot|_{L^4([0,T]:L^4)} + |\cdot|_{L^\infty([0,T]:H)}$.
\begin{theorem} \label{thm:NS-mathcal-M}
  For any $\Psi \in \mathbf{L_T}$, there is a 
   unique mild solution $v \in \mathbf{L_T}$ of  \eqref{eq:NS-M}, namely,
$$v(t) = -\int_0^t S(t-s) [B(s, v(s) + \Psi(s))] ds, \; t \in [0,T].$$
Define the map
	$\mathcal{M}$ from $\mathbf{L_T}$ to itself as $\mathcal{M}(\Psi) \doteq v+\Psi$ if $v$ solves
	\eqref{eq:NS-M} for $\Psi$. The map $\mathcal{M}$ has the following properties.
  \begin{enumerate}[(a)]
  \item If $\Psi \in C([0,T]:H) \cap \mathbf{L_T} $, then $M(\Psi) \in C([0,T]:H)\cap \mathbf{L_T}$.
  \item There exists a nondecreasing function $\gamma : [0,\infty) \to [0,\infty)$ such that for any $\Psi \in \mathbf{L_T}$ and $t \in [0,T]$
      \begin{equation} \label{eq:NS-M-growth}
        |\mathcal{M}(\Psi)|_{\mathbf{L_t}}  \leq \gamma(|\Psi|_{\mathbf{L_t}}).
      \end{equation}
  \item For any $R \in (0,\infty)$ there exists $C=C(R)\in (1,\infty)$ such that whenever $\Phi, \Psi \in \mathbf{L_T}$ with $|\Phi|_{\mathbf{L_T}} \leq R$ and $|\Psi|_{\mathbf{L_T}} \leq R$, and $t \in [0,T]$,
      \begin{equation} \label{eq:NS-M-Lipschitz-Linfty}
        |\mathcal{M}(\Phi) - \mathcal{M}(\Psi)|_{\mathbf{L_t}} \leq C |\Phi - \Psi|_{\mathbf{L_t}}.
      \end{equation}
  \item For any $R \in (0,\infty)$ and $p \in [2,\infty)$, there exists $C=C(R,p) \in (0,\infty)$ such that whenever $\Phi, \Psi \in \mathbf{L_T}$ with $|\Phi|_{\mathbf{L_T}} \leq R$ and $  |\Psi|_{\mathbf{L_T}} \leq R$, and $t \in [0,T]$
      \begin{equation} \label{eq:NS-M-Lipschitz-Lp}
        |\mathcal{M}(\Phi) - \mathcal{M}(\Psi)|_{L^p([0,t]:H)} \leq C |\Phi - \Psi|_{L^4([0,t]:H)}^{\frac{2}{p}}.
      \end{equation}
  \end{enumerate}
\end{theorem}

\begin{proof}
  We follow  \cite{bl-2004} and \cite{bl-2006}. Specifically, for $\Psi \in \mathbf{L_T}$,  existence and uniqueness of a mild solution $v \in \mathbf{L_T}$ of  \eqref{eq:NS-M} has been shown
in \cite[Theorem 4.5]{bl-2006}.
The same result shows that if $\Psi \in C([0,T]:H) \cap L^4([0,T]:L^4)$, then in fact $\mathcal{M}(\Psi) \in C([0,T]:H) \cap L^4([0,T]:L^4)$.  This proves part (a).

  Part (b) follows from \cite[Lemma 3.19]{bl-2004} where it is shown that there exists an nondecreasing function $\gamma_1$ such that
  \[|v|_{L^\infty([0,T]:H)} + |v|_{L^2([0,T]:H^1)} \leq \gamma_1( |\Psi|_{L^4([0,T]:L^4)}).\]
  Using the interpolation inequality, for $m_1<m_2$, $\theta \in (0,1)$, and  $u \in H^{m_2}$,
\begin{equation} \label{eq:interpolation}
  |u|_{H^{(1-\theta)m_1 + \theta m_2}} \leq |u|_{H^{m_1}}^{1-\theta} |u|_{H^{m_2}}^\theta,
\end{equation}
  \begin{align} \label{eq:NS-L4-embed}
    &|v|_{L^4([0,T]:L^4)}^4 \leq\kappa |v|_{L^4([0,T]:H^{1/2})}^4 \leq \kappa\int_0^T |v(s)|_{H^{1/2}}^4 ds \leq \kappa\int_0^T |v(s)|_H^2 |v(s)|_{H^1}^2 ds \nonumber \\
    &\leq \kappa|v|_{L^\infty([0,T]:H)}^2|v|_{L^2([0,T]:H^1)}^2.
  \end{align}
where the first inequality is from  Sobolev embedding.
  Then, since $\mathcal{M}(\Psi) = v + \Psi$,
  part (b) holds with $\gamma(R) = (1+\kappa^{1/4})\gamma_1(R) + R$.

  The proof of (c) follows the same energy arguments as used in proofs of parts (a) and (b) (cf. \cite{bl-2004} and \cite{bl-2006}). Specifically,
for $\Phi,\Psi$ and $R$ as in the statement of part (c),
let $v_1 \doteq \mathcal{M}(\Phi) - \Phi$ and $v_2 \doteq \mathcal{M}(\Psi) - \Psi$. Let $v \doteq v_1-v_2$. We also let $z \doteq \Phi-\Psi$. Notice that $\mathcal{M}(\Phi) - \mathcal{M}(\Psi) = v + z$. The weak derivative of $v$ is
  \[v'(t) = Av(t) - B(\mathcal{M}(\Phi)(t)) + B(\mathcal{M}(\Psi)(t)).\]
  Therefore,
  \begin{align*}
    \frac{d}{dt}|v(t)|_H^2 &= 2\left<v(t),v'(t)\right>_H\\
     &=-2|v(t)|_{H^1}^2 - 2b(\mathcal{M}(\Phi)(t), \mathcal{M}(\Phi)(t), v(t)) + 2b(\mathcal{M}(\Psi)(t),\mathcal{M}(\Psi)(t),v(t)).
  \end{align*}
  By the trilinearity of $b$ and the fact that $\mathcal{M}(\Phi) - \mathcal{M}(\Psi) = v + z$,
  \begin{align*}
    &\frac{d}{dt}|v(t)|_H^2 + 2|v(t)|_{H^1}^2 \\
&\quad = -2b(\mathcal{M}(\Phi)(t) - \mathcal{M}(\Psi)(t), \mathcal{M}(\Phi)(t),v(t))
     + 2b(\mathcal{M}(\Psi)(t), -\mathcal{M}(\Phi)(t) + \mathcal{M}(\Psi)(t),v(t)) \\
    &\quad= -2b(v(t)+z(t), \mathcal{M}(\Phi)(t), v(t)) - 2b(\mathcal{M}(\Psi)(t), v(t) + z(t), v(t)).
  \end{align*}
  By \eqref{eq:b-reg-L4},
  \[\frac{d}{dt}|v(t)|_H^2 + 2|v(t)|_{H^1}^2 \leq  \kappa_1\left( |v(t)|_{L^4} + |z(t)|_{L^4} \right) \left(|\mathcal{M}(\Phi)(t)|_{L^4}
  + |\mathcal{M}(\Psi)(t)|_{L^4}
  \right)|v(t)|_{H^1}.\]
  By the Sobolev Embedding Theorem and the interpolation inequality \eqref{eq:interpolation},
  \[|v(t)|_{L^4} \leq \kappa |v(t)|_{H^{1/2}} \leq \kappa|v(t)|_H^{\frac{1}{2}}|v(t)|_{H^1}^{\frac{1}{2}}\]
  and thus
  \[\frac{d}{dt}|v(t)|_H^2 + 2|v(t)|_{H^1}^2 \leq  \kappa_2 \Big(|\mathcal{M}(\Phi)(t)|_{L^4} + |\mathcal{M}(\Psi)(t)|_{L^4}\Big)  \left( |v(t)|_{H}^{\frac{1}{2}} |v(t)|_{H^1}^{\frac{3}{2}}
  + |z(t)|_{L^4}|v(t)|_{H^1}
  \right).\]
  By Young's inequality, we can absorb all of the $|v(t)|_{H^1}$ terms into the term on the left-hand side. It follows that
  \begin{align} \label{eq:NS-energy-Linfty}
    \frac{d}{dt}|v(t)|_H^2 + |v(t)|_{H^1}^2 \leq \kappa_3\Bigg(&|v(t)|_H^2 \Big(|\mathcal{M}(\Phi)(t)|_{L^4} + |\mathcal{M}(\Psi)(t)|_{L^4}\Big)^4 \nonumber\\&+ |z(t)|_{L^4}^2 \Big(|\mathcal{M}(\Phi)(t)|_{L^4} + |\mathcal{M}(\Psi)(t)|_{L^4}\Big)^2\Bigg).
  \end{align}
  By Gr\"onwall's inequality and the fact that $v(0)=0$,
  \begin{align*}
    |v(t)|_H^2 \leq &\kappa_4 \exp\left(\kappa_4 \Big(|\mathcal{M}(\Phi)|_{L^4([0,t]:L^4)}^4 + |\mathcal{M}(\Psi)|_{L^4([0,t]:L^4)}^4\Big)\right)\\
     &\times \int_0^t |z(s)|_{L^4}^2\left(|\mathcal{M}(\Phi)(s)|_{L^4} + |\mathcal{M}(\Psi)(s)|_{L^4} \right)^2ds
  \end{align*}
  By H\"older's inequality, part (b), and the assumption that $|\Phi|_{L^4([0,T]:L^4)} \leq R$ and $|\Psi|_{L^4([0,T]:L^4)} \leq R$,  we see that there is some $\kappa_5$ depending on $R$ such that for any $t \in [0,T]$
  \[|v|_{L^\infty([0,t]:H)} \leq \kappa_5 |z|_{L^4([0,t]:L^4)}.\]
  Plugging this back into \eqref{eq:NS-energy-Linfty}, shows that
  \[|v|_{L^2([0,t]:H^1)} \leq \kappa_6 |z|_{L^4([0,t]:L^4)}.\]
  We get a bound on $|v|_{L^4([0,T]:L^4)}$ by \eqref{eq:NS-L4-embed}.
  Finally, the result follows because $\mathcal{M}(\Phi) - \mathcal{M}(\Psi) = v+z$.

  Finally we prove part (d). Arguments will be a bit  different since we want these $L^p([0,T]:H)$ bounds of $\mathcal{M}(\Phi) - \mathcal{M}(\Psi)$ to not depend on the $L^4([0,T]:L^4)$ bound of $\Phi - \Psi$. To accomplish this, we use the same notation as part (c) but we do the energy estimates in the space $H^{-1}$. In particular,
  \begin{align*}
    &\frac{d}{dt}|v(t)|_{H^{-1}}^2 = 2\left< v'(t),(-A)^{-1}v(t)\right>_H \\
    & = -2|v(t)|_H^2 - 2b(\mathcal{M}(\Phi)(t),\mathcal{M}(\Phi)(t), (-A)^{-1}v(t)) + 2b(\mathcal{M}(\Psi)(t),\mathcal{M}(\Psi)(t),(-A)^{-1}v(t)).
  \end{align*}
  By the trilinearity of $b$,
  \begin{align*}
    \frac{d}{dt}|v(t)|_{H^{-1}}^2 + 2|v(t)|_H^2 = &-2b(v(t) + z(t), \mathcal{M}(\Phi)(t), (-A)^{-1}v(t))\\
     &- 2b(\mathcal{M}(\Psi)(t), v(t) + z(t), (-A)^{-1}v(t)).
  \end{align*}
  By \eqref{eq:b-reg-H32} and \eqref{eq:b-reg-H32-switch},
  \begin{align*}
    &\frac{d}{dt}|v(t)|_{H^{-1}}^2 + 2|v(t)|_H^2 \leq \kappa_1|v(t) + z(t)|_H \left(|\mathcal{M}(\Phi)(t)|_{L^4} + |\mathcal{M}(\Psi)(t)|_{L^4}\right) |(-A)^{-1}v(t)|_{H^{\frac{3}{2}}}
  \end{align*}
  From the definition of the $H^\delta$ spaces, $|(-A)^{-1}v(t)|_{H^{\frac{3}{2}}} = |v(t)|_{H^{-\frac{1}{2}}}$. And by interpolation \eqref{eq:interpolation}, $|v(t)|_{H^{-\frac{1}{2}}} \leq |v(t)|_H^{\frac{1}{2}}|v(t)|_{H^{-1}}^{\frac{1}{2}}$. Combining these estimates,
  \begin{align*}
    &\frac{d}{dt}|v(t)|_{H^{-1}}^2 + 2|v(t)|_H^2 \leq \kappa_2 |v(t)|_H^{\frac{3}{2}}|v(t)|_{H^{-1}}^{\frac{1}{2}} \left(|\mathcal{M}(\Phi)(t)|_{L^4} + |\mathcal{M}(\Psi)(t)|_{L^4}\right) \\
    &+\kappa_2|z(t)|_H \left(|\mathcal{M}(\Phi)(t)|_{L^4} + |\mathcal{M}(\Psi)(t)|_{L^4}\right) |v(t)|_{H^{-\frac{1}{2}}}.
  \end{align*}
  For the second term on the right hand side we can use the crude estimate $|v(t)|_{H^{-\frac{1}{2}}} \leq \kappa |v(t)|_H$. Applying Young's inequality to both terms, we can absorb the $|v(t)|_H$ terms into the term on the left hand side.
  \begin{align} \label{eq:NS-energy-Lp}
    &\frac{d}{dt}|v(t)|_{H^{-1}}^2 + |v(t)|_H^2 \leq \kappa_3 |v(t)|_{H^{-1}}^2\left(|\mathcal{M}(\Phi)(t)|_{L^4} + |\mathcal{M}(\Psi)(t)|_{L^4}\right)^4 \nonumber\\
    &+ \kappa_3|z(t)|_H^2 \left(|\mathcal{M}(\Phi)(t)|_{L^4} + |\mathcal{M}(\Psi)(t)|_{L^4}\right)^2.
  \end{align}
  By the Gr\"onwall inequality and the fact that $v(0)=0$,
  \[|v(t)|_{H^{-1}}^2 \leq \kappa_4e^{\kappa_4 \left(|\mathcal{M}(\Phi)|_{L^4([0,t]:L^4)}^4 + |\mathcal{M}(\Psi)|_{L^4([0,t]:L^4)}^4\right)} \int_0^t |z(s)|_H^2\left(|\mathcal{M}(\Phi)(s)|_{L^4} + |\mathcal{M}(\Psi)(s)|_{L^4}\right)^2ds.  \]
  By H\"older's inequality and the fact that it follows from (b) that $|\mathcal{M}(\Phi)|_{L^4([0,T]:L^4)} $ and $|\mathcal{M}(\Psi)|_{L^4([0,T]:L^4)} $ are bounded, we see that
  \[|v(t)|_{H^{-1}}^2 \leq \kappa_5 |z|_{L^4([0,t]:H)}^2.\]
  Plugging this back into \eqref{eq:NS-energy-Lp}, we see that
  \[\int_0^t |v(s)|_H^2 ds \leq \kappa_6 |z|_{L^4([0,t]:H)}^2.\]
  Note that $\kappa_6$ depends on $R$.
  It follows from the H\"older inequality that
  \begin{equation}  \label{eq:NS-L2-bound}
    |\mathcal{M}(\Phi) - \mathcal{M}(\Psi)|_{L^2([0,t]:H)} \leq |v|_{L^2([0,t]:H)} + |z|_{L^2([0,t]:H)} \leq \kappa_7 |z|_{L^4([0,t]:H)}.
  \end{equation}
  For $p\geq 2$, we notice that
  \begin{align*}
	&\int_0^t |\mathcal{M}(\Phi)(s) - \mathcal{M}(\Psi)(s)|^p_H ds \\
	&\quad \leq \left(|\mathcal{M}(\Phi)|_{L^\infty([0,T]:H)} +| \mathcal{M}(\Psi)|_{L^\infty([0,T]:H)}\right)^{p-2} \int_0^t |\mathcal{M}(\Phi)(s) - \mathcal{M}(\Psi)(s)|_H^2.
\end{align*}
  By part (b), \eqref{eq:NS-L2-bound}, and the assumption that $|\Phi|_{\mathbf{L_T}} \leq R$, $|\Psi|_{\mathbf{L_T}} \leq R$, we have
  \[\int_0^t |\mathcal{M}(\Phi)(s) - \mathcal{M}(\Psi)(s)|^p_H ds \leq \kappa_8 |z|_{L^4([0,t]:H)}^2 \]
where $\kappa_8$ depends on $R$ and $p$.
\end{proof}

\subsubsection{Verification of a modification of Assumption \ref{assum:G-new}}
\label{sec:vergnew}
In Section \ref{sec:suff-conds} we saw that Assumption \ref{assum:G-sufficient} together with Assumption \ref{assum:semigroup} ensures that Assumption \ref{assum:G-new} is satisfied. In this section we will first
verify in Lemma \ref{lem:NS-assum-G} an analogue of Assumption \ref{assum:G-sufficient} and then use it to establish an analogue of Assumption \ref{assum:semigroup} in Theorem \ref{thm:NS-G-new}. We begin with the following lemma.
Recall the eigenvectors and eigenvalues $\{f_k,\lambda_k\}$ of $Q$ introduced in Assumption \ref{assum:NS-noise-reg}.
Also recall the definitions of $G$ and $\tilde G$ given below  Lemma \ref{eq:b-reg-L4}.
\begin{lemma}  \label{lem:NS-G-ek-reg}
 There exists $C\in (0,\infty)$ such that for any  $\varphi \in H$, $\psi \in H$, and $k\in \mathbb{N}$,
  \begin{equation} \label{eq:NS-S-G-ek-bound}
    |G(\varphi)f_k|_H \leq C \lambda_k
  \end{equation}
  and
  \begin{equation} \label{eq:NS-S-G-ek-Lipschitz}
    |(G(\varphi) - G(\psi))f_k |_H \leq C\lambda_k|\varphi - \psi|_H.
  \end{equation}
\end{lemma}
\begin{proof}
  By Assumption \ref{assum:NS-noise-reg}, for some $\kappa_1 \in (0,\infty)$,
$\sup_{k \in \NN} |f_k|_{L^\infty}\leq \kappa_1$. Therefore by Assumption \ref{assum:NS-g-reg}, and the fact that $P: L^2 \to H$ is a projection,
  \[|G(\varphi)f_k|_H^2 \leq |\tilde G(\varphi)f_k|_{L^2}^2 = \lambda_k^2 \int\limits_{\mathcal{O}} |g(\varphi(\xi))|_{\RR}^2|f_k(\xi)|_{\RR^2}^2d\xi \leq \kappa_1^2\lambda_k^2 |g|_{L^\infty(\RR^2:\RR)}^2\]
  and by the same argument
  \[|(G(\varphi) - G(\psi))f_k|_H^2 \leq \kappa_1^2\lambda_k^2|g|_{\text{Lip}}^2|\varphi - \psi|_H^2.\]
The result follows.
\end{proof}

Now we verify a modification of Assumption \ref{assum:G-sufficient}.
\begin{lemma} \label{lem:NS-assum-G}
  For any $\alpha \in (1/4,1/2)$, $p>\frac{1}{\alpha-\frac{1}{4}}$ and any $\mathcal{F}_t$-progressively measurable $\varphi$
with values in $L^\infty([0,T]:H)$,
  \[
    Z_\alpha(\varphi)(t) \doteq\int_0^t (t-s)^{-\alpha} S(t-s)G(\varphi(s))dw(s)
  \]
  is a well defined $H$-valued random variable. Furthermore, for $t \in [0,T]$,
  \[\E \left|Z_\alpha(\varphi) (t)\right|_H^p \leq C\E \left(\int_0^t (t-s)^{-2\alpha}ds \right)^{\frac{p}{2}}\]
  and
  \[\E \left|Z_\alpha(\varphi)(t) - Z_\alpha(\psi)(t) \right|_H^p \leq C\E\left(\int_0^t (t-s)^{-2\alpha} |\varphi(s)-\psi(s)|_H^2 ds \right)^{\frac{p}{2}}\]
\end{lemma}

\begin{proof}
  By the Ito isometry for Hilbert-space-valued stochastic integrals (cf. Section 4.2.1 of \cite{DaP-Z}) and Lemma \ref{lem:NS-G-ek-reg} the stochastic integrals in the
statement of the lemma are well defined $H$-valued random variables. We only prove the second inequality as the proof of the first inequality is similar.
By the Burkholder-Gundy-Davis inequality,
  \begin{align*}
    &\E \left|\int_0^t(t-s)^{-\alpha} S(t-s)(G(\varphi(s)) - G(\psi(s)))dw(s) \right|_{H}^p \\
    &\leq \kappa_1\E \left(\sum_{k=1}^\infty \int_0^t (t-s)^{-2\alpha}|S(t-s)(G(\varphi(s)) - G(\psi(s))) f_k|_H^2 ds \right)^{\frac{p}{2}}\\
    &\leq \kappa_2 \E \left( \int_0^t (t-s)^{-2\alpha} |\varphi(s)-\psi(s)|_H^2 \Tr(Q^2)ds \right)^{\frac{p}{2}}.
  \end{align*}
  The result follows.
\end{proof}

Because $S(t)$ is an analytic semigroup, for $x \in H$ and $t>0$, $|S(t)x|_{H^{1/2}} \leq M t^{-1/4}|x|_H$ \cite[Theorem 2.6.13]{pazy}.
From this fact and Lemma \ref{lem:NS-assum-G} it follows that an analogue of Assumption \ref{assum:G-sufficient} is satisfied where the spaces $E, E_1, E_2, H$ in Assumption \ref{assum:G-sufficient} are replaced by
$H^{\frac{1}{2}},H, H, L^2$, $r= 1/4$, $\zeta(t) = t^{-2\alpha}$ and the function $\gamma^{-1}$ is replaced by the constant function.

By following the methods of Section \ref{sec:suff-conds} and Lemma \ref{lem:NS-assum-G}  we can verify the following modification of Assumption \ref{assum:G-new}. We omit the proof.
For  $\mathcal{F}_t$-progressively measurable $\varphi \in L^\infty([0,T]:H)$, let
\[Z(\varphi)(t) \doteq \int_0^t S(t-s)G(\varphi(s))dw(s)\]
and for $\varphi \in L^\infty([0,T]:H)$, $u \in L^2([0,T]:L^2)$,
\[\left(\mathcal{L}(\varphi)u\right)(t) \doteq \int_0^t S(t-s)G(\varphi(s))u(s)ds.\]
Define for $N \in \NN$, $\mathcal{S}^N \doteq \{u \in L^2([0,T]:L^2): \int_0^T |u(s)|_{L^2}^2 \leq N \}$ endowed with the metric of weak convergence.

\begin{theorem} \label{thm:NS-G-new}
  There exists $p>2$ and  $C\in (0,\infty)$ such that the following  hold.
  \begin{enumerate}[1.]
    \item For any $\mathcal{F}_t$-progressively measurable $\varphi \in L^\infty([0,T]:H)$ with probability one, $Z(\varphi)$ is $C([0,T]:H^{1/2})$-valued and any $t \in [0,T]$,
    \begin{equation}
      \E\left|Z(\varphi) \right|_{C([0,t]:H^{1/2})}^p \leq C.
    \end{equation}
    \item For any $\mathcal{F}_t$-progressively measurable $\varphi, \psi \in L^p(\Omega:L^\infty([0,T]:H))$, and $t \in [0,T]$,
    \begin{equation}
      \E \left|Z(\varphi) - Z(\psi) \right|_{C([0,t]:H^{1/2})}^p \leq C \E \int_0^t |\varphi(s) - \psi(s)|_H^pds.
    \end{equation}
    \item For any $R>0$, the collection
    \[\left\{Z(\varphi): \varphi \text{ is } \mathcal{F}_t\text{-progressively measurable and } \varphi \in L^\infty([0,T]:H) \right\}\]
     is tight in $C([0,T]:H^{1/2})$.
    \item For any $\varphi \in L^\infty([0,T]:H)$ and $u \in L^2([0,T]:L^2)$, $\mathcal{L}(\varphi)u$ is well defined and $C([0,T]:H^{1/2})$-valued and for any $t \in [0,T]$,
    \begin{equation}
      \left|\mathcal{L}(\varphi)u \right|_{C([0,t]:H^{1/2})}^p \leq C|u|_{L^2([0,t]:L^2)}^p.
    \end{equation}
    \item For any $\varphi, \psi \in L^\infty([0,T]:H)$ and $u \in L^2([0,T]:L^2)$,
    \begin{equation}
       \left| \mathcal{L}(\varphi)u - \mathcal{L}(\psi)u\right|^p_{C([0,t]:H^{1/2})} \leq C|u|_{L^2([0,T]:L^2)}^p \int_0^t |\varphi(s) - \psi(s)|_H^pds.
    \end{equation}
    \item For any $R,N>0$, the collection
    \[\left\{\mathcal{L}(\varphi)u : |\varphi|_{L^\infty([0,T]:H)} \leq R, u \in \mathcal{S}^N  \right\}\]
    is pre-compact in $C([0,T]:H^{1/2})$.
  \end{enumerate}
\end{theorem}
We remark on the differences between properties established in Theorem \ref{thm:NS-G-new} and Assumption \ref{assum:G-new}.
The theorem shows that the first three lines of Assumption \ref{assum:G-new} are satisfied exactly with $(E, E_2,H)$ replaced with $(H^{1/2}, H, L^2)$. Parts 1-6 of the theorem are analogous to parts (a)-(f) of Assumption \ref{assum:G-new}.
The main difference is that $Z(\varphi)$ and $\mathcal{L}(\varphi)u$ are $C([0,T]:H^{1/2})$-valued while $\varphi \in L^\infty([0,T]:H)$. In Assumption \ref{assum:G-new}, we used the same underlying space, $E$ for both.
Proofs in Section  \ref{sec:suff-conds} do not use in an essential way that the spaces
 of $\varphi$ and  $Z(\varphi)$ are the same and the proof of Theorem \ref{thm:NS-G-new} follows using similar arguments.
Finally, the assumption that $g$ is uniformly bounded simplifies the bounds in Theorem \ref{thm:NS-G-new} parts 1 and 4
and the tightness statement in 3 since the function $\gamma^{-1}$ is simply  replaced by the constant function.

%

\subsubsection{Wellposedness}
\label{sec:wellposNavSto}
We begin with the following lemma that shows that for every $x\in H$
 $S(\cdot)x + Y^\e_x$ is in the domain of $\mathcal{M}$.

\begin{lemma} \label{lem:NS-S-dot-in-L4}
  For any $x \in H$, $S(\cdot)x \in C([0,T]:H) \cap L^4([0,T]:L^4)$ and there exists $C>0$ such that
  \[|S(\cdot)x|_{C([0,T]:H)\cap L^4([0,T]:L^4)} \leq C|x|_H.\]
\end{lemma}
\begin{proof}
  The fact that for every $x\in H$, $S(\cdot)x \in C([0,T]:H)$ and $|S(\cdot)x|_{C([0,T]:H)} \le \kappa_1 |x|_H$
is immediate since $\{S(t)\}$ is a  $C_0$ semigroup. Thus  we only need to show that $S(\cdot): H \to L^4([0,T]:L^4)$ is a bounded linear operator. By the Sobolev embedding theorem, $L^4 \supset H^{1/2}$. Therefore, for any $x \in H$,
  \begin{align}
  \label{eq:NS-semigorup-bounded}
  &\int_0^T |S(t)x|_{L^4}^4 dt \leq \kappa_2 \int_0^T |S(t)x|_{H^{1/2}}^4 dt
  \leq \kappa_2 \int_0^T \left(\sum_{k=1}^\infty \gamma_k^{\frac{1}{2}}e^{-2\gamma_k t}\left<x,e_k\right>_H^2\right)^2dt \nonumber\\
  & \leq \kappa_2 \int_0^T \left(\sum_{k=1}^\infty \gamma_k e^{-4\gamma_k t} \left<x,e_k\right>_H^2 \right)\left(\sum_{k=1}^\infty \left<x,e_k\right>_H^2 \right)dt \leq \frac{\kappa}{4}|x|_H^4.
\end{align}
The result follows.
\end{proof}
Let $\mathcal{P}_2^N$ be the collection of $\mathcal{F}_t$-progressively measurable processes that are in $\mathcal{S}^N$ with probability one.
The following uniqueness result is analogous to Theorem \ref{thm:Yeux-uniqueness}.
\begin{theorem} \label{thm:NS-Y-exist}
  For any $x \in H$, $\e\ge 0$, $N\in \NN$ and $u \in \mathcal{P}_2^N$, there exists a unique  $Y^{\e,u}_x \in L^p(\Omega:C([0,T]:H^{1/2}))$ solving
  \begin{align}
    Y^{\e,u}_x(t) = \sqrt{\e} &\int_0^t S(t-s)G(\mathcal{M}(S(\cdot)x + Y^\e_x)(s))dw(s)\nonumber\\
     &+ \int_0^t S(t-s)G(\mathcal{M}(S(\cdot)x + Y^\e_x)(s))u(s)ds.\label{eq:yeuxNS}
  \end{align}
\end{theorem}
\begin{proof}
  This is proven in much the same way as Theorem \ref{thm:Yeux-uniqueness}. The only thing that we need to be careful of is the fact that for general $x \in H$, $S(\cdot)x \not \in C([0,T]:H^{1/2})$. As we proved in Lemma \ref{lem:NS-S-dot-in-L4}, $S(\cdot)x \in C([0,T]:H)\cap L^4([0,T]:L^4)$, which is contained in the domain of $\mathcal{M}$.
  We define $\T_R$ as in Lemma \ref{lem:cutoff} with $E=H^{1/2}$. Let
  \[\mathcal{E}_R \doteq \{\varphi \in C([0,T]:H)\cap L^4([0,T]:L^4) : |\varphi|_{\mathbf{L_T}} \leq R\}\]
  and
  \[\tilde{\mathcal{E}}_R \doteq \{\varphi \in C([0,T]:H^{1/2}) : |\varphi|_{C([0,T]:H^{1/2})} \leq R\}.\]
    If $\varphi \in \mathcal{E}_R$ and $x \in H$, then by Lemma \ref{lem:NS-S-dot-in-L4}, $S(\cdot)x + \varphi \in \mathcal{E}_{R + C|x|_H}$. By Theorem \ref{thm:NS-mathcal-M}(c), $\mathcal{M}$ is Lipschitz continuous from $\mathcal{E}_{R + C|x|_H} \to C([0,T]:H)$ for any $R>0$, $x\in H$. Let $\mathscr{K}^{\e,u}_{x,R}:
L^p(\Omega: \mathcal{E}_R) \to L^p(\Omega: \tilde{\mathcal{E}}_R)$ be defined by \eqref{eq:mathscr-K-def}.
With $x\in H$, $u \in \mathcal{P}_2^N$ and $\e\ge 0$, consider $Y^{\e,u}_{x,R}$ as in \eqref{eq:eqyeuxr}.
Then along the lines of Lemma \ref{lem:a-priori-bounds} we have that there is a unique $\mathcal{F}_t$-progressively measurable continuous $\tilde{\mathcal{E}}_R$ valued process $Y^{\e,u}_{x,R}$ that solves \eqref{eq:eqyeuxr}.
Furthermore, the $L^p$ estimate in the statement of Lemma \ref{lem:a-priori-bounds} continues to hold with $E$ on the left side replaced by $H^{1/2}$ and $E^{\star\star}$ on the right replaced with $H$. Using this the proof of unique solvability of \eqref{eq:yeuxNS} is completed exactly as in Theorem \ref{thm:Yeux-uniqueness}. We omit the details.
\end{proof}

Consider now, for $x\in H$, $u \in \mathcal{P}_2^N$, and $\e\ge 0$,  the controlled analogue of \eqref{eq:N-S-abstract}:
\begin{equation} \label{eq:N-S-abstractcontNS}
  dX^{\e,u}_x(t) = [A X^{\e,u}_x(t) - B(X^{\e,u}_x(t))]dt + \sqrt{\e}G(X^{\e,u}_x(t))dw(t)+ G(X^{\e,u}_x(t))u(t)dt, \ \  X^{\e,u}_x(0) = x.
\end{equation}
By a mild solution of \eqref{eq:N-S-abstractcontNS} we mean a $\{\clf_t\}$-progressively measurable $H$ valued continuous stochastic process
$X^{\e,u}_x$ with sample paths in $C([0,T]:H) \cap L^4([0,T]:L^4)$ such that
  \begin{align} \label{eq:XeuxNS}
    X^{\e,u}_x = &\mathcal{M} \left(S(\cdot)x
    +\sqrt{\e} \int_0^\cdot S(\cdot-s)G(X^{\e,u}_x(s))dw(s)
    + \int_0^\cdot S(\cdot-s)G(X^{\e,u}_x(s))u(s)ds\right).
  \end{align}
The following result is analogous to Theorems \ref{thm:uniqmilsoln} and \ref{thm:uniqdetcont} and, given Theorem \ref{thm:NS-Y-exist}, the proof follows by similar arguments as used for these results given in Section \ref{sec:uniqmildsolns}. Details are omitted.
\begin{theorem}
	\label{thm:uniqxeuxNS}
	For every $x\in H$, $u \in \mathcal{P}_2^N$, and $\e\ge 0$ there exists  a unique mild solution
	$X^{\e,u}_x$ of \eqref{eq:N-S-abstractcontNS}.
\end{theorem}

\subsubsection{Uniform large deviation principle for $X^\e_x$}
\label{sec:uldpNS}
When $u=0$, we denote the unique mild solution of \eqref{eq:N-S-abstractcontNS}  as $X^\e_x$ and the unique solution
of \eqref{eq:yeuxNS} as $Y^\e_x$. The main result of this section is the following uniform large deviation principle for
$X^\e_x$ in $C([0,T]:H)$. Define for $x\in H$ and $\varphi \in C([0,T]:H)$, $I_x(\varphi)$ by \eqref{eq:rate-fct-X}
with $H$ there replaced by $L^2$ and $X^{0,u}_x$ given by the solution of \eqref{eq:N-S-abstractcontNS} (with $\e=0$) rather than \eqref{eq:eq1214ab1}.
\begin{theorem} \label{thm:NS-LDP}
	The collection $\{X^{\varepsilon}_x\}_{\e>0}$ satisfies a large deviation principle in $C([0,T]:H)$ with respect to the rate function
	$I_x$, uniformly in the class $\mathscr{A}$ of all bounded subsets of $H$. That is, for
	any bounded subset $E_0$ of $H$, the  uniform lower and upper bounds  in \eqref{eq:LDP-lower-unif} and  \eqref{eq:LDP-upper-unif} hold with $E$ there replaced with $H$.
\end{theorem}
As was the case for the proof of Theorem \ref{thm:LDP}, the key step in the proof of the above theorem is establishing a
uniform Laplace principle for $\{Y^{\e}_x\}$. For $x \in H$ and $\varphi \in C([0,T]:H^{1/2})$, define
$\tilde I_x(\varphi)$ by \eqref{eq:rate-fct-Y} with $H$ replaced by $L^2$ and $Y^{0,u}_x$ given as the
unique solution of \eqref{eq:yeuxNS} (with $\e=0$) rather than \eqref{eq:youxt}. Then we have the following uniform
Laplace principle for $\{Y^{\e}_x\}$. Note the difference from Theorem \ref{thm:unif-Laplace-Y} -- $E$ is replaced with $H^{1/2}$ however the uniformity parameter $x$ is in $H$ rather than in $E^{\star\star} = H^{1/2}$.
One problem that we will encounter is that the semigroup, which by Lemma \ref{lem:U-compact} is compact from $H \to L^p([0,T]:H)$ for any $p\geq 1$, is not compact from $H \to L^4([0,T]:L^4)$.
Despite this difficulty, the proof can be completed as it turns out that it is enough for $S(\cdot)$ to be bounded as an operator from  $H \to  L^4([0,T]:L^4)$ and compact as an operator from $H \to L^4([0,T]:H)$.
\begin{theorem} \label{thm:NS-Y-ULP}
	The collection $\{\tilde I_x\}_{x \in H}$ is a family of rate functions with compact level sets on closed and bounded subsets of $H$.
  $\{Y^{\e}_x\}$ satisfies a uniform Laplace principle in $C([0,T]:H^{1/2})$  with rate function $\tilde{I}_x$ that is uniform over $x$ in bounded subsets of $H$.

Specifically, for any bounded subset $E_0 \subset H$ and any bounded continuous $h: C([0,T]:H^{1/2}) \to \mathbb{R}$,
\[\lim_{\e \to 0} \sup_{x \in E_0}\left|\e \log\E\left(\exp(-\e^{-1}h(Y^\e_x)) \right) + \inf_{\phi \in C([0,T]:H^{1/2})}\{\tilde{I}_x(\phi) + h(\phi)\} \right| = 0.\]
\end{theorem}
\begin{proof}
The proof is similar to that of Theorem \ref{thm:unif-Laplace-Y}. The main step is to	  show that if $E_0\ni x_n \to x$ weakly in $H$, $[0,1)\ni \e_n \to 0$, and $u_n \to u$ in distribution in $\mathcal{S}^N$, then $Y^{\e_n,u_n}_{x_n} \to Y^{0,u}_x$ in distribution in $C([0,T]:H^{1/2})$.


 The proof of this statement follows in a similar way as the proof of Theorem \ref{thm:unif-conv}. The only step that we need to be careful about is how to go from \eqref{eq:Y5} to \eqref{eq:Y6}.
By Theorem \ref{thm:NS-G-new} (3) and (5) $\{Y^{\e_n,u_n}_{x_n}\}$ is tight in $C([0,T]:H^{1/2})$.
As in the proof of Theorem \ref{thm:unif-conv}, using Skorohod's lemma we can find a probability space and a subsequence such that $Y_n \doteq Y^{\e_n,u_n}_{x_n} \to \tilde{Y}$ in $C([0,T]:H^{1/2})$ almost surely.
In particular, $Y_n \to \tilde Y$ in $C([0,T]:H) \cap L^4([0,T]:L^4)$,
   therefore for almost any $\omega \in \Omega$, for large enough $n$ (depending on $\omega$),
   \[|Y_n|_{\mathbf{L_T}} \leq |\tilde{Y}|_{\mathbf{L_T}}+1.\]
    Furthermore, by Lemma \ref{lem:NS-S-dot-in-L4},
    \[\sup_{n \in \NN}|S(\cdot)x_n|_{\mathbf{L_T}} \leq C\sup_{n \in \NN}|x_n|_H.\]
    By Theorem \ref{thm:NS-mathcal-M}(d) with $R = |\tilde{Y}|_{\mathbf{L_T}} + C\sup_{n \in \NN}|x_n|_H+1$
and $p\ge 2$,
  \[|\mathcal{M}(S(\cdot) x_n + Y_n) - \mathcal{M}(S(\cdot)x + \tilde{Y})|_{L^p([0,T]:H)} \leq \kappa \left(|S(\cdot)(x-x_n)|_{L^4([0,T]:H)}^{\frac{2}{p}} + |Y_n - \tilde Y|_{L^4([0,T]:H)}^{\frac{2}{p}} \right).\]
  As noted earlier, although the mapping $x \in H \mapsto S(\cdot)x \in L^4([0,T]:L^4)$ is not compact, the map $x
	  \mapsto S(\cdot)x \in L^p([0,T];H)$ is compact. This combined with the fact that $Y_n \to \tilde{Y}$ in $C([0,T]:H^{1/2})$ we have that the left side
converges to zero as $n\to \infty$.

The rest of the proof of Theorem \ref{thm:unif-Laplace-Y} follows with minor modifications to account for the fact that $Y^{\e,u}_x \in C([0,T]:H^{1/2})$ while $x \in H$ (in the proofs of Section \ref{sec:pfuniflap} $Y^{\e,u}_x \in C([0,T]:E)$ and $x \in E^{\star\star}$).
\end{proof}

We can now complete the proof of Theorem \ref{thm:NS-LDP}.\\

{\bf Proof of Theorem \ref{thm:NS-LDP}}

We observe that if for $x\in H$ and some $u \in \mathcal{P}_2^N$, $Y^{\e,u}_x$ is the unique solution of \eqref{eq:yeuxNS}  	then $X^{\e,u}_x = \mathcal{M}(S(\cdot)x + Y^{\e,u}_x)$ is the unique solution of \eqref{eq:N-S-abstractcontNS}.

  By Lemma \ref{lem:NS-S-dot-in-L4}, $S(\cdot) \in \mathscr{L}(H,\mathbf{L_T})$. Let $\psi \in C([0,T]:H^{1/2})$. Let $R = \|S(\cdot)\|_{\mathscr{L}(H,\mathbf{L_T})} + |\psi|_{\mathbf{L_T}} + 1$. Let $C=C(R)$ be as from Theorem \ref{thm:NS-mathcal-M}(c).
 Let $\kappa \in (0,\infty)$ be such that $|z|_{\mathbf{L_T}} \leq \kappa |z|_{C([0,T]:H^{1/2})}$ for all
$z \in C([0,T]:H)\cap L^4([0,T]:L^4)$. Then whenever $|Y^\e_x - \psi|_{C([0,T]:H^{1/2})} < \delta/(\kappa C)$, it follows that
  \begin{align*}
    |X^\e_x - \mathcal{M}(S(\cdot)x+\psi)|_{C([0,T]:H)} &\leq C|(S(\cdot)x +Y^\e_x) - (S(\cdot)x + \psi)|_{\mathbf{L_T}}\\
    & \leq C \kappa |Y^\e_x - \psi|_{C([0,T]:H^{1/2})}  < \delta.
  \end{align*}
The rest of the proof is the same as for Theorem \ref{thm:LDP}.
\hfill \qed

\appendix

\section{Compactness and the double dual space}

\label{S:doubledual} Throughout this section $E$ will be a separable Banach
space and $E^\star$ will denote its dual space. Let $E^{\star\star}$ be the
dual of $E^\star$. There is a natural identification of $E$ as a subset of $%
E^{\star\star}$ because the evaluation operation is a bounded linear
functional. Specifically, we define the \emph{canonical map} $J_E: E \to
E^{\star\star}$ as follows. For any $x \in E$, $J_E(x): E^\star \to \mathbb{R%
}$ by $J_E(x)(x^\star) = x^\star(x)$. In the duality notation,
\begin{equation*}
\left<x^\star, J_E(x)\right>_{E^\star, E^{\star\star}} \doteq \left<x,
x^\star\right>_{E,E^\star}
\end{equation*}

If $J_E(E) = E^{\star\star}$, then $E$ is called a reflexive Banach space.
In the context of SPDEs, we are also interested in non-reflexive Banach
spaces like the space of continuous functions, H\"older spaces, and Sobolev
spaces.

If $E$ is an infinite dimensional Banach space, then the closed unit ball of
$E$ is not compact in the norm topology.
Alaoglu's Theorem, which we state below, guarantees that the closed unit
ball of a dual space is compact in the weak-$\star$ topology. We also state
Goldstine's theorem, which proves that the of the unit ball of $E$ under the
canonical embedding is dense in the unit ball of $E^{\star\star}$ in the
weak-$\star$ topology. This means that even if $E$ is not reflexive, all
elements of $E^{\star\star}$ can be approximated by a net of elements of $E$%
. For proofs see \cite[Section 15.1 and 15.3]{fitzpatrick}.

\begin{theorem}[Alaoglu's Theorem]
\label{thm:alaoglu} Every closed, bounded ball in a dual space is weak-$%
\star $ compact.
\end{theorem}

\begin{theorem}[Goldstine's Theorem]
\label{thm:Goldstine} Let $B_1 \doteq \{x \in E: |x|_{E}\leq 1\}$ be the closed
unit ball in $E$ and $B_1^{\star\star}\doteq \{x^{\star\star} \in E^{\star\star}:
|x^{\star\star}|_{E^{\star\star}} \leq 1\}$. Let $J_E: E \to E^{\star\star}$
be the canonical embedding. The closure of $J_E(B_1)$ in the weak-$\star$
topology in $E^{\star\star}$ is $B_1^{\star\star}$.
\end{theorem}

Let $E$ and $F$ be Banach spaces and let $J_F: F \to F^{\star\star}$ and $%
J_E: E \to E^{\star\star}$ be the canonical embeddings. For any bounded
linear operator $L:E \to F$, the adjoint of $L$ is the unique bounded linear
operator $L^\star: F^\star \to E^\star$ such that
\begin{equation*}
\left<Lx,y^{\star} \right>_{F, F^\star} = \left<x, L^\star y^\star
\right>_{E,E^\star}.
\end{equation*}
In general, $L^{\star\star}$ is a map from $E^{\star\star} \to
F^{\star\star} $ that extends $L$ in the sense that for any $x \in E$, $%
J_F(Lx) = L^{\star\star}J_E(x)$. The following theorem says that if $L$ is a
compact operator, then $L$ can be extended to a compact linear operator from
$E^{\star\star} \to F$.

\begin{theorem}
\label{thm:doubledualextend} Let $E$ and $F$ be Banach spaces. Let $L:E \to
F $ be a compact linear operator. Then $L$ can be uniquely extended to a
compact linear operator from $E^{\star\star} \to F$ by setting
\begin{equation*}
Lx^{\star\star} \doteq J_F^{-1}L^{\star\star}x^{\star\star}.
\end{equation*}
\end{theorem}

\begin{proof}
By Goldstine's Theorem (Theorem \ref{thm:Goldstine}), $E^{\star\star}$ is the weak-$\star$ completion of $J(E)$. Therefore, for any $x^{\star\star} \in E^{\star\star}$, there exists an index set $\cli$ and a net $\{x^i\}_{i \in \cli} \subset E$ such that $J_E(x^i) \to x^{\star\star}$ in the weak-$\star$ topology. By the compactness of $L$, we can find a subnet of $x^i$ (relabeled $x^i$) and a $y \in F$ such that $|L x^i - y|_F \to 0$. This means that $|L^{\star\star}J_E(x^i) - J_F(y)|_{F^{\star\star}} \to 0$ also.

Furthermore, by the weak-$\star$ convergence of $J_E(x^i) \to x^{\star\star}$, for any $y^{\star} \in F^{\star}$,
\begin{align*}
  &\left<Lx^i, y^\star \right>_{F,F^\star} = \left<x^i, L^\star y^\star \right>_{E,E^\star} = \left< L^{\star}y^{\star}, J_E(x^i)\right>_{E^\star, E^{\star\star}}\\
  &\to \left< L^{\star}y^\star, x^{\star\star} \right>_{E^\star, E^{\star\star}} = \left<y^\star, L^{\star\star} x^{\star\star} \right>_{F^\star,F^{\star\star}}.
\end{align*}
On the other hand, because $L x^i \to y \in F$,
\[\left<Lx^i, y^\star \right>_{F, F^\star} \to \left< y, y^\star\right>_{F,F^\star} = \left< y^\star,J_F(y)\right>_{F^\star,F^{\star\star}}.\]
This implies that $J_F(y) = L^{\star\star}x^{\star\star}$ and that $y$ is independent of the net $x^i \to x^{\star\star}$. Since $x^{\star\star} \in E^{\star\star}$ was arbitrary, we have proven that $L^{\star\star}(E^{\star\star}) \subset J_F(F)$.

We can then extend $L$ to an operator from $E^{\star\star} \to F$ by letting $Lx^{\star\star} = J_F^{-1}(L^{\star\star}x^{\star\star})$.

These arguments show that
\[\overline{\{Lx: x \in B_1\}} = \{Lx^{\star\star}: x^{\star\star} \in B^{\star\star}_1\} \]
where the overline denotes the closure in $F$ norm. Because the original $L$ was a compact operator, these sets are compact, and therefore, the extension $L$ is also a compact operator.
\end{proof}

Now we present some properties of compact semigroups and the double dual
space of a Banach space.

\begin{definition}
\label{def:c0semigroup} A collection $\{S(t)\}_{t\ge 0}$ of bounded linear
operators on $E$ is called a $C_0$-semigroup if $S(0)=I$ (the identity
operator), $S(t+s) = S(t)S(s)$ for all $t,s\ge 0$, and for any $x \in E$, $%
t\mapsto S(t)x$ is continuous from $[0,\infty)$ to $E$ (with norm topology
on $E$). The infinitesimal generator $A$ of a semigroup is a linear operator
with domain
\begin{equation*}
D(A) \doteq \left\{x \in E: \lim_{h \downarrow 0} \frac{S(t)x -x}{h}
\mbox{
exists}\right\}
\end{equation*}
and
\begin{equation*}
Ax \doteq \lim_{h \downarrow 0} \frac{S(t)x -x}{h}, \; x \in D(A).
\end{equation*}
We say the $C_0$-semigroup $\{S(t)\}_{t\ge 0}$ is a compact semigroup if for
every $t> 0$, $S(t)$ is a compact operator.
\end{definition}

We note that by \cite[Theorem 1.2.2]{pazy}, for a semigroup as in Definition %
\ref{def:c0semigroup},
\begin{equation*}
\sup_{0\le t \le T}\|S(t)\|_{\mathscr{L}(E)} <\infty \mbox{ for every } T>0.
\end{equation*}


\begin{lemma}
\label{lem:U-compact} Assume that $S(t)$ is a compact $C_0$ semigroup.
Denote for $t>0$, the extension of $S(t)$ from  $E$ to $E^{\star\star}$ given
by Theorem \ref{thm:doubledualextend} once again as $S(t)$.

\begin{enumerate}

\item For every $x^{\star\star} \in E^{\star\star}$ and $t_1 \in (0,T)$, the
map $t \mapsto S(t)x^{\star\star}$ from $[t_1,T]$ to $E$ is continuous.
Setting $S(0)x^{\star\star} =0$, $|S(t) x^{\star\star}|_E \le \|S(t)\|_{%
\mathscr{L}(E)}|x^{\star\star}|_{E^{\star\star}}$ for every $t\ge 0$ and the
map $t \mapsto S(t)x^{\star\star}$ from $[0,T]$ to $E$ is in $L^p([0,T]:E)$
for all $p\in [1,\infty]$.

\item Let ${\mathcal{I}}$ be a directed set and let $\{x^i\}_{i\in {\mathcal{%
I}}} \subset E^{\star\star}$ be a bounded net converging to $x\in
E^{\star\star}$ (in the weak- $\star$ topology). Then the net $%
\{S(\cdot)x^i\}_{i\in {\mathcal{I}}}$ converges, in $C([t_1,T]:E)$ for any $%
0<t_1<T$ and in $L^p([0,T]: E)$ for any $p\in[1,\infty)$, to $S(\cdot)x$.

\item Let $\{x_n\}_{n \in {\mathbb{N}}} \subset E^{\star\star}$ be such that
for some $C<\infty$, $|x_n|_{E^{\star\star}} \leq C$ for every $n$. Then
there exists a subsequence (relabeled $x_n$) such that $S(t)x_n$ converges
for all $t>0$. As a function of time $S(\cdot)x_n$ converges in $%
C([t_1,T]:E) $ for any $0<t_1\leq T$ and in $L^p([0,T]:E)$ for any $p \in
[1,\infty)$. 
Moreover, there exists $x^{\star\star} \in E^{\star\star}$ such that
\begin{equation}  \label{eq:eq748}
\lim_{n \to \infty} |S(\cdot)x_n - S(\cdot)x^{\star\star}|_{C([t_1,T]:E)} =
0 \text{ and } \lim_{n \to \infty} |S(\cdot)x_n -
S(\cdot)x^{\star\star}|_{L^p([0,T]:E)} = 0.
\end{equation}
\end{enumerate}
\end{lemma}

\begin{proof}
Consider first part 1.	Fix $x^{\star\star} \in E^{\star\star}$. Then by Theorem \ref{thm:Goldstine},  there exists a net $\{x^i\}_{i \in \cli} \in E$ such that $|x^i|_E \le |x^{\star\star}|_{E^{\star\star}}$
	and $J_E(x^i) \to x^{\star\star}$ in the weak-$\star$ topology. From Theorem \ref{thm:doubledualextend}, for any
	$t>0$, $S(t)x^i \to S(t)x^{\star\star}$ in $E$. Additionally, for all $t>0$
	$$|S(t)x^{\star\star}|_E \le \sup_{i\in \cli} |S(t)x^i|_E \le  \|S(t)\|_{\mathscr{L}(E)}
	|x^{\star\star}|_{E^{\star\star}} <\infty .$$
	Since $\sup_{0\le t \le T} \|S(t)\|_{\mathscr{L}(E)} <\infty$, $t\mapsto S(t)x^{\star\star}$ is in $L^p([0,T]:E)$ for any $p\in [1,\infty]$. Also for $t_n, t \ge t_0 >0$ such that
	$t_n \to t$, $S(t_n)x^{\star\star} = S(t_n-t_0)y$, $S(t)x^{\star\star} = S(t-t_0)y$ where
	$y= S(t_0)x^{\star\star}$. Since $S(t)$ is a $C_0$-semigroup on $E$ and $y \in E$,
$$S(t_n)x^{\star\star} = S(t_n-t_0)y \to S(t-t_0)y = 	S(t)x^{\star\star}$$
as $n\to \infty$. This shows that for any $t_1 \in (0,T)$, $t \mapsto S(t)x^{\star\star}$ is in
$C([t_1,T]:E)$ completing the proof of part 1.

Now consider part 2.  By the proof of Theorem \ref{thm:doubledualextend} and the fact that $S(t)$ is a compact semigroup, $|S(t)x^i - S(t)x|_E \to 0$ for any $t>0$. If we fix $0<t_1<T$, then
\[\sup_{t \in [t_1,T]} |S(t)(x^i-x)|_E \leq \sup_{t \in [t_1,T]} \|S(t-t_1)\|_{\mathscr{L}(E)}|S(t_1)(x^i-x)|_E \to 0,\]
proving the convergence in $C([t_1,T]:E)$.
For the $L^p([0,T]:E)$ convergence, observe that $|S(\cdot)x^i - S(\cdot)x|_{L^p([0,T]:E)}^p = \int_0^T |S(t)x^i - S(t)x|^p dt.$ This converges to zero by dominated convergence theorem because $|S(t)x^i - S(t)x|_E$ converges to zero for each $t>0$ and \\$\sup_{t \in [0,T]}\sup_{i \in \mathcal{I}} |S(t)(x^i-x)|<\infty.$

As for part 3, let $\{x_n\}_{n \in \NN} \subset E^{\star\star}$ be such that for some $C<\infty$, $|x_n|_{E^{\star\star}} \leq C$ for every $n$.
Since $S(1)$ is a compact operator, we can find a subsequence of $x_n$ such that $S(1)x_n$ converges in $E$. Similarly, we can find a subsequence of that subsequence  such that $S(1/2)x_n$ converges. Using the diagonalization method, we can find a subsequence such that $S(1/N)x_n$ converges for any $N \in \mathbb{N}$.
Since the semigroup is bounded, we now have
 uniform convergence of $S(t)x_n$ in $E$ over $[1/N,T]$ each all $N\in \NN$.
Indeed, letting $y_N = \lim_{n \to \infty} S(1/N)x_n$, we have,
  \begin{equation}\label{eq:eq810}
	\sup_{\frac{1}{N} \le t\le T} |S(t)x_n - S(t-{1}/{N}) y_N|_E \leq  \sup_{0\le t \le T} \|S(t)\|_{\mathscr{L}(E)}|S(1/N)x_n - y_N|_E,
\end{equation}
proving the convergence in $C([t_1,T]:E)$ for any $t_1\in (0,T]$.
The bounded convergence theorem now implies that the subsequence converges in $L^p([0,T]:E)$ for any $p \in [1,\infty)$.

 Finally we show \eqref{eq:eq748}. By Theorem \ref{thm:alaoglu}, there exists $x^{\star\star} \in E^{\star\star}$ that is an accumulation point of the sequence $x_n$ in the weak-$\star$ topology. There exists a subnet $\{x^i\}_{i \in \cli}$ of $\{x_n\}$ (although not necessarily a subsequence) such that $x^i \to x^{\star\star}$ in the weak-$\star$ topology.
 By part 2, the subnet $S(\cdot)x^i$ converges to $S(\cdot)x^{\star\star}$ in $C([t_1,T]:E)$ and $L^p([0,T]:E)$ for any $0<t_1<T$ and $p \in [1,\infty)$, proving \eqref{eq:eq748}.
\end{proof}

\bibliographystyle{plain}
\bibliography{LDP}

\vspace{\baselineskip}

\textsc{\noindent M. Salins\newline Department of Mathematics and Statistics\newline
Boston University\newline  Boston, MA 02215, USA\newline email:
msalins@bu.edu \vspace{\baselineskip}}

\textsc{\noindent A. Budhiraja \newline Department of Statistics and
Operations Research\newline University of North Carolina\newline Chapel Hill,
NC 27599, USA\newline email: budhiraj@email.unc.edu \vspace{\baselineskip} }

\textsc{\noindent P. Dupuis\newline Division of Applied Mathematics\newline
Brown University\newline Providence, RI 02912, USA\newline email:
Paul\_Dupuis@brown.edu \vspace{\baselineskip} }

\end{document}